\documentclass[reqno]{amsart}
\usepackage{color}
\usepackage{amsmath,amssymb,amsfonts,amsthm,bm}
\usepackage{mathrsfs}
\usepackage{graphicx}
\usepackage{enumerate}
\usepackage[numbers, sort&compress]{natbib}
\usepackage[shortlabels]{enumitem}
  \usepackage{newtxtext}
 \usepackage[varvw]{newtxmath}
 \usepackage{textcomp}
\usepackage[reversemp,paperwidth=155mm,paperheight=235mm,top={22mm},headheight={5.5pt},headsep={6.0mm},text={123mm,195mm},marginparsep=5mm,marginparwidth=12mm, bindingoffset=6mm, footskip=10mm]{geometry}
 \makeatletter

 \newbox \abstractbox
\renewenvironment{abstract}{\global\setbox\abstractbox=\vbox\bgroup
 \hsize=\textwidth
  \vskip 1.2cm
  \noindent\unskip \textbf{Abstract.}
 }
 {%\vskip12pt \hrule
 \egroup}

\@namedef{subjclassname@2020}{%
	\textup{2020} Mathematics Subject Classification}

\def\@settitle{%
  \bgroup
  \centering
  \vglue1cm
  \fontsize{12}{15}\fontseries{b}\selectfont
  \@title
  \vskip 20pt plus 6pt minus 8pt
  \egroup
}

\def\@setauthors{%
  \begingroup
  \trivlist
  \centering %\bfseries
 \normalsize\@topsep30\p@\relax
  \advance\@topsep by -\baselineskip
  \item\relax
  \andify\authors
 {\rmfamily\authors}%
  \endtrivlist
  \endgroup
}

\def\@setaddresses{\par
  \nobreak \begingroup
\normalsize
  \def\author##1{\nobreak\addvspace\bigskipamount}%
  \def\\{\unskip, \ignorespaces}%
  \interlinepenalty\@M
  \def\address##1##2{\begingroup
    \par\addvspace\bigskipamount\noindent
    \@ifnotempty{##1}{(\ignorespaces##1\unskip) }%
    {\ignorespaces##2}\par\endgroup}%
  \def\curraddr##1##2{\begingroup
    \@ifnotempty{##2}{\nobreak\indent{\itshape Current address}%
      \@ifnotempty{##1}{, \ignorespaces##1\unskip}\/:\space
      ##2\par}\endgroup}%
  \def\email##1##2{\begingroup
    \@ifnotempty{##2}{\nobreak\noindent{\itshape E-mail address}%
      \@ifnotempty{##1}{, \ignorespaces##1\unskip}\/:
       ##2\par}\endgroup}%
   \def\urladdr##1##2{\begingroup
    \@ifnotempty{##2}{\nobreak\indent{\itshape URL}%
      \@ifnotempty{##1}{, \ignorespaces##1\unskip}\/:\space
      \ttfamily##2\par}\endgroup}%
  \addresses
  \endgroup
}

  \renewcommand\section{\@startsection{section}{1} %
  \z@{.5\linespacing\@plus.7\linespacing}{.5\linespacing}
 {\normalfont\large\bfseries\boldmath}}

  \def\subsection{\@startsection{subsection}{2}%
  \z@{.5\linespacing\@plus.7\linespacing}{.2\linespacing}%
  {\raggedright\normalfont\bfseries\boldmath}}
  
\def\subsubsection{\@startsection{subsubsection}{3}%
  \z@{.5\linespacing\@plus.7\linespacing}{-.5em}%
  {\normalfont\bfseries}}

\makeatother

 \newtheorem{theorem}{Theorem}[section]
\newtheorem{lemma}[theorem]{Lemma}
\newtheorem{corollary}[theorem]{Corollary}
\newtheorem{proposition}[theorem]{Proposition}
\newtheorem{definition}[theorem]{Definition}
 \theoremstyle{definition}
\newtheorem{remark}[theorem]{Remark}

\numberwithin{equation}{section}

\newcommand{\eps}{\varepsilon}
\newcommand{\norm}[1]{\Vert#1\Vert}
\newcommand{\abs}[1]{\left\vert#1\right\vert}

\newcommand{\inner}[1]{\left(#1\right)}
\newcommand{\comi}[1]{\left<#1\right>}

\newcommand{\normm}[1]{{ \vert\kern-0.25ex \vert\kern-0.25ex \vert #1
		\vert\kern-0.25ex \vert\kern-0.25ex \vert}}

\begin{document}
 
\title[Radius of analyticity and Gevrey regularity  for  the  Boltzmann  equation]{On the radius of analyticity and Gevrey regularity  for  the  Boltzmann  equation}

\author[W.-X.Li, L.Liu and H.Wang]{Wei-Xi Li, Lvqiao Liu and Hao Wang}

 %\address[X. Hu] {Wuhan Institute for Math \& AI,   Wuhan University, Wuhan 430072, China}
 %\email{hux@whu.edu.cn}

\address[W.-X. Li]{School of Mathematics and Statistics,   \&  Hubei Key Laboratory of Computational Science,  Wuhan University, Wuhan 430072,  P. R. China} \email{wei-xi.li@whu.edu.cn}

\address[L. Liu]{School of Mathematics and Statistics, Anhui Normal University,  Wuhu 241002,  China}\email{lvqiaoliu@ahnu.edu.cn}

\address[H.Wang]{Faculty of Mathematics and Statistics \& Hubei Key Laboratory of Applied Mathematics,  Hubei University,  Wuhan 430062,   China} \email{wang-h@whu.edu.cn}

 

\begin{abstract}
	This paper investigates the non-cutoff Boltzmann equation for hard potentials in a perturbative setting.
		 We first establish a sharp short-time estimate  on the radius of analyticity and Gevrey regularity of mild solutions. Furthermore,   we  obtain a global-in-time radius estimate in  Gevrey space. The proof combines hypoelliptic estimates with the macro-micro decomposition.
\end{abstract}
 \maketitle

% Some math journals (FLO) require a table of contents. Comment out this line if no ToC is needed.
%\localtableofcontents

 \section{Introduction }

 The Boltzmann equation stands as the most classical and fundamental mathematical model in collisional kinetic theory, describing the behavior of a dilute gas at the macroscopic level. It is well recognized that kinetic theory links the microscopic and macroscopic scales. At the macroscopic level, where gases and fluids are treated as continua, their motion is governed by the compressible or incompressible Euler and Navier–Stokes equations. The hydrodynamic limits of the Boltzmann equation provide a crucial procedure for deriving the macroscopic Euler and Navier–Stokes equations from the mesoscopic Boltzmann description.

In this work, we are concerned with the transition of regularity in the hydrodynamic limit process. There have been extensive research devoted to the regularization effects in the Navier–Stokes equations. For instance, in the classical work by Foias and Temam \cite{MR1026858}, space-time analyticity for positive times was established via Fourier analytical methods under low-regularity initial data. This Fourier-based approach, along with the subsequently developed
more modern  techniques, has been widely applied to study the analytic smoothing effect of the Navier–Stokes equations and more general parabolic equations in various function spaces; see, e.g., \cite{MR4601187, MR4216598, MR1607936, MR1938147, MR3369268, MR2836112, MR2960037, MR1748305, MR2145938,MR575737,MR699168}. It should be noted that such regularization effect stems from the diffusion term; hence, an analogous parabolic smoothing effect cannot be expected for the Euler equations.

On the other hand, regularizing properties of the Boltzmann equation were first observed by Lions \cite{MR1278244} and later confirmed in \cite{MR1324404, MR1639275, MR1475459}. Here we particularly mention the recent works \cite{MR4612704, MR4930523} concerning the smoothing effect of the Boltzmann and Landau equations in the settings of analytic and optimal Gevrey-class regularity, where the estimates are local in time.

This work investigates the radius of analyticity or Gevrey-class regularity for solutions to the Boltzmann equation. For solutions with analytic or Gevrey-class regularity, we first establish a refined local-in-time estimate of the radius near time zero, which improves  the one given in \cite{MR4930523}. This estimate appears to be optimal, as it agrees  with observations from toy models. Moreover, we derive uniform lower bounds on the radius of analyticity or Gevrey-class regularity for all large times.

\subsection{Boltzmann  equation}
 
  The spatially inhomogeneous Boltzmann equation on the torus is given by
\begin{equation}\label{1}
\partial_t F + v \cdot \nabla_x F = Q(F, F), \quad F|_{t=0} = F_0,
\end{equation}
where $F(t, x, v)$ denotes the probability density function at position $x \in \mathbb{T}^3$, time $t \geq 0$, with velocity $v \in \mathbb{R}^3$. When $F = F(t, v)$ is independent of $x$, equation \eqref{1} reduces to the spatially homogeneous Boltzmann equation.

The Boltzmann collision operator on the right-hand side of \eqref{1} is a bilinear operator defined by
\begin{equation}\label{collis}
Q(G, F)(t, x, v) = \int_{\mathbb{R}^3} \int_{\mathbb{S}^2} B(v - v_*, \sigma)\big( G'_* F' - G_* F \big)  dv_*  d\sigma,
\end{equation}
where, throughout this paper, we use the standard abbreviations:
$F' = F(t, x, v')$, $F = F(t, x, v)$, $G'_* = G(t, x, v'_*)$, and $G_* = G(t, x, v_*)$.
The pairs $(v, v_*)$ and $(v', v'_*)$ represent the post and pre-collision velocities, respectively, satisfying the conservation of momentum and energy:
\begin{equation*}
v' + v'_* = v + v_*, \quad |v'|^2 + |v'_*|^2 = |v|^2 + |v_*|^2.
\end{equation*}
These relations yield the $\sigma$-representation for $\sigma \in \mathbb{S}^2$:
\begin{equation*}
\left\{
\begin{aligned}
v' &= \frac{v + v_*}{2} + \frac{|v - v_*|}{2} \sigma, \\
v'_* &= \frac{v + v_*}{2} - \frac{|v - v_*|}{2} \sigma.
\end{aligned}
\right.
\end{equation*}
The cross-section $B(v - v_*, \sigma)$ in \eqref{collis} depends on the relative speed $|v - v_*|$ and the deviation angle $\theta$, where
\begin{equation*}
\cos \theta = \frac{v - v_*}{|v - v_*|} \cdot \sigma.
\end{equation*}
Without loss of generality, we assume $B(v - v_*, \sigma)$ is supported on $0 \leq \theta \leq \pi/2$ (so that $\cos \theta \geq 0$) and takes the form
\begin{equation}\label{kern}
B(v - v_*, \sigma) = |v - v_*|^\gamma  b(\cos \theta),
\end{equation}
where $|v - v_*|^\gamma$ is the kinetic part with $-3 < \gamma \leq 1$, and $b(\cos \theta)$ is the angular part satisfying
\begin{equation}\label{angu}
0 \leq \sin \theta \, b(\cos \theta) \approx \theta^{-1 - 2s}
\end{equation}
for $0 < s < 1$. Here and throughout, $p \approx q$ means $C^{-1} q \leq p \leq C q$ for some generic constant $C \geq 1$. 
The angular part $b(\cos \theta)$ thus has a non-integrable singularity at $\theta = 0$:
\begin{equation*}
\int_0^{\pi/2} \sin \theta \, b(\cos \theta)  d\theta = +\infty.
\end{equation*}
The cases $-3 < \gamma < 0$, $\gamma = 0$, and $0 < \gamma \leq 1$ are referred to as soft potential, Maxwellian molecules, and hard potential, respectively.

This work studies solutions of the Boltzmann equation \eqref{1} near the normalized global Maxwellian
\begin{equation}\label{gmu}
\mu(v) = (2\pi)^{-3/2} e^{-|v|^2 / 2}.
\end{equation}
Setting $F(t, x, v) = \mu(v) + \sqrt{\mu(v)} f(t, x, v)$ and similarly for the initial datum $F_0$, the perturbation $f = f(t, x, v)$ satisfies
\begin{equation}\label{3}
\partial_t f + v \cdot \nabla_x f - \mathcal{L} f = \Gamma(f, f), \quad f|_{t=0} = f_0,
\end{equation}
where the linearized collision operator $\mathcal{L}$ and the nonlinear operator $\Gamma$ are given by
\begin{equation}\label{colli}
	\mathcal{L}f =\mu^{-\frac12}Q(\mu,\sqrt{\mu} f)
		+\mu^{-\frac12}Q(\sqrt{\mu}f,\mu)\  \textrm { and } \ \Gamma(f,g)=\mu^{-\frac12}Q(\sqrt{\mu}f,\sqrt{\mu}g).
\end{equation}
Since the Boltzmann collision term $Q(F, F)$ admits the five collision invariants $1$, $v$, and $|v|^2$, sufficiently regular and integrable solutions of \eqref{1} conserve mass, momentum, and energy. For simplicity, we therefore assume that $f(t, x, v)$ satisfies
\begin{equation} \label{law}
\int_{\mathbb{T}^3} \int_{\mathbb{R}^3} \big(1, v, |v|^2\big) \sqrt{\mu(v)}  f(t, x, v)  dv  dx = 0
\end{equation}
for all $t \geq 0$. In particular, if \eqref{law} holds initially, it remains valid for all $t > 0$.

\subsection{Notations and function spaces}\label{notafun}

Given two operators $Q_1$ and $Q_2$, we denote by $[Q_1, Q_2]$ their commutator:
\begin{equation*}
[Q_1, Q_2] = Q_1 Q_2 - Q_2 Q_1.
\end{equation*}
Moreover, we define the weight function $\comi \cdot := \big( 1 + |\cdot|^2 \big)^{\frac12}$.

Throughout the paper, we denote by $\hat f$ or $\mathscr F_x f$ the partial Fourier transform of $f(t,x,v)$ with respect to $x \in \mathbb{T}^3$:
\begin{equation*}
\hat{f}(t, k, v) = \mathscr F_x f(t, k, v) = \int_{\mathbb{T}^3} e^{-ik\cdot x} f(t, x, v)  dx, \quad k \in \mathbb{Z}^3,
\end{equation*}
where $k \in \mathbb{Z}^3$ is the Fourier dual variable of $x \in \mathbb{T}^3$. Similarly, $\mathscr F_{x,v} f$ denotes the full Fourier transform with respect to $(x,v)$, with $(k,\eta)$ being the dual variables.

For convenience, we denote by $\hat{\Gamma}(\hat{f},\hat{g})$ the partial Fourier transform of $\Gamma(f,g)$ defined in \eqref{colli}:
\begin{multline*}
 \hat{\Gamma}(\hat{f},\hat{g})(t,k,v) := \mathscr F_x \left[ \Gamma(f,g) \right](t,k,v) \\
 = \int_{\mathbb{R}^3} \int_{\mathbb{S}^{2}} B(v-v_*,\sigma)  \mu^{\frac12}(v_*) \left( \big[\hat{f}(v_*') * \hat{g}(v')\big](k) - \big[\hat{f}(v_*) * \hat{g}(v)\big](k) \right)  d\sigma  dv_*,
\end{multline*} 
where the convolutions are taken with respect to $k \in \mathbb{Z}^3$:
\begin{equation}\label{def:conv}
[\hat{f}(u) * \hat{g}(v)](k) := \int_{\mathbb{Z}^3_\ell} \hat{f}(t,k-\ell,u) \hat{g}(t,\ell,v)  d\Sigma(\ell),
\end{equation}
for any velocities $u,v \in \mathbb{R}^3$. Here, $d\Sigma(k)$ denotes the discrete measure on $\mathbb{Z}^3$:
\begin{equation*}
\int_{\mathbb{Z}^3} g(k)  d\Sigma(k) := \sum_{k \in \mathbb{Z}^3} g(k),
\end{equation*}
for any summable function $g$ on $\mathbb{Z}^3$.

When applying Leibniz's formula, it is convenient to work with the trilinear operator $\mathcal T$ defined by
\begin{equation}\label{matht}
\mathcal T(g,h,\omega) = \iint B(v-v_*,\sigma) \omega_* \left( g_*' h' - g_* h \right)  dv_*  d\sigma,
\end{equation}
where $B$ is given in \eqref{kern}, and $\omega$ depends only on $v$. The bilinear operator $\Gamma$ in \eqref{colli} is related to $\mathcal T$ by
\begin{equation}\label{trb}
\Gamma(g,h) = \mathcal T(g,h,\mu^{\frac12}).
\end{equation}
Similarly, we denote by $\hat{\mathcal{T}}(\hat{g}, \hat{h}, \omega)$ the partial Fourier transform of $\mathcal T(g,h,\omega)$ with respect to $x$:
\begin{multline*}\label{Foutri}
\hat{\mathcal{T}}(\hat{g}, \hat{h}, \omega)(k,v) = \mathscr F_x \left[ \mathcal T(g,h,\omega) \right](k,v) \\
= \iint B(v-v_*,\sigma) \omega(v_*) \left( [\hat{g}(v'_*) * \hat{h}(v')](k) - [\hat{g}(v_*) * \hat{h}(v)](k) \right)  dv_*  d\sigma,
\end{multline*}
with the convolutions defined as in \eqref{def:conv}.

We denote by $L^2_v$ the standard Lebesgue space $L^2$ in the $v$ variable, and similarly for $L^2_{x,v}$. Denote by $H^p_v$   the classical Sobolev space in $v$ variable. 

We recall the mixed Lebesgue spaces $L^p_k L^q_T L^r_v$ from \cite{MR4230064}:
\begin{equation*}
L^p_k L^q_T L^r_v = \left\{ g = g(t,x,v) : \| g \|_{L^p_k L^q_T L^r_v} < +\infty \right\},
\end{equation*}
where
\begin{equation*}
\| g \|_{L^p_k L^q_T L^r_v} := 
\begin{cases}
\left( \int_{\mathbb{Z}^3} \left( \int_0^T \| \hat{g}(t,k,\cdot) \|_{L^r_v}^q  dt \right)^{\frac{p}{q}}  d\Sigma(k) \right)^{\frac{1}{p}}, & q < \infty, \\
\left( \int_{\mathbb{Z}^3} \left( \sup_{0 < t < T} \| \hat{g}(t,k,\cdot) \|_{L^r_v} \right)^p  d\Sigma(k) \right)^{\frac{1}{p}}, & q = \infty,
\end{cases}
\end{equation*}
for $1 \leq p, r < \infty$ and $1 \leq q \leq \infty$. In particular,
\begin{equation*}
L^p_k L^r_v = \left\{ g = g(x,v) : \| g \|_{L^p_k L^r_v} = \left( \int_{\mathbb{Z}^3} \| \hat{g}(k,\cdot) \|_{L^r_v}^p  d\Sigma(k) \right)^{\frac{1}{p}} < +\infty \right\},
\end{equation*}
and
\begin{equation*}
L_k^1 = \left\{ g = g(x) : \| g \|_{L^1_k} = \int_{\mathbb{Z}^3} |\hat{g}(k)|  d\Sigma(k) < +\infty \right\}.
\end{equation*}
Finally, we recall the triple norm $\normm{\cdot}$ introduced in \cite{MR2863853}:
\begin{equation}\label{trinorm}
\begin{aligned}
\normm{f}^2 &:= \int_{\mathbb{R}^3} \int_{\mathbb{R}^3} \int_{\mathbb{S}^2} B(v-v_*,\sigma) \mu_* ( f - f' )^2  d\sigma  dv  dv_* \\
&\quad + \int_{\mathbb{R}^3} \int_{\mathbb{R}^3} \int_{\mathbb{S}^2} B(v-v_*,\sigma) f_*^2 \big( \sqrt{\mu'} - \sqrt{\mu} \big)^2  d\sigma  dv  dv_*.
\end{aligned}
\end{equation}
This triple norm is equivalent to the anisotropic norm $|\cdot|_{N^{\gamma,s}}$ introduced in \cite{MR2784329}. Both norms can be characterized explicitly via pseudo-differential operators (cf. \cite{MR3950012}).

Using the space $L_k^1 L_v^2$, we now define the Gevrey space considered in this work.

\begin{definition}\label{de}
Let $r > 0$. We denote by $\mathcal{G}^r$ the space of all $C^\infty$ functions $f(x,v)$ satisfying
\begin{equation}\label{degev}
\forall \alpha, \beta \in \mathbb{Z}^3_+, \quad \| \partial_x^\alpha \partial_v^\beta f \|_{L^1_k L^2_v} \leq C^{|\alpha| + |\beta| + 1} \left[ (|\alpha| + |\beta|)! \right]^r,
\end{equation}
for some constant $C > 0$ independent of $\alpha$ and $\beta$. The number $r$ is called the Gevrey index.

An equivalent characterization of \eqref{degev} is given by the Fourier multiplier:
\begin{equation*}
e^{c_* (-\Delta_x - \Delta_v)^{\frac{1}{2r}}} f \in L^1_k L^2_v
\end{equation*}
for some $c_* > 0$, where $e^{c_* (-\Delta_x - \Delta_v)^{\frac{1}{2r}}}$ is defined via
\begin{equation*}
\mathscr{F}_{x,v} \left( e^{c_* (-\Delta_x - \Delta_v)^{\frac{1}{2r}}} f \right)(k,\eta) = e^{c_* (|k|^2 + |\eta|^2)^{\frac{1}{2r}}} \mathscr{F}_{x,v} f(k,\eta),
\end{equation*}
with $\mathscr{F}_{x,v}$ being the full Fourier transform in $(x,v)$ and $(k,\eta)$ the dual variables.
\end{definition}

\subsection{Regularity and radius analysis for toy models}\label{subsec:toy}

Our main results are inspired by observations on simplified models of the Boltzmann equation. Consider the fractional Kolmogorov equation on $\mathbb{T}^3_x \times \mathbb{R}^3_v$:
\begin{equation*}\label{se1}
\begin{cases}
\partial_t g + v \cdot \partial_x g + (-\Delta_v)^s g = 0, & 0 < s < 1, \\
g|_{t=0} = g_0 \in L^2_{x,v}.
\end{cases}
\end{equation*}
Applying the Fourier transform, the solution is explicitly given by
\begin{equation}\label{se2}
(\mathscr{F}_{x,v} g)(t,k,\eta) = e^{-\int_0^t |\eta + \rho k|^{2s}  d\rho} (\mathscr{F}_{x,v} g_0)(k,\eta + t k).
\end{equation}
Moreover, as shown in \cite[Lemma 3.1]{MR2523694}, we have
\begin{equation}\label{se3}
- \left( t |\eta|^{2s} + t^{2s+1} |k|^{2s} \right) / c_s \leq -\int_0^t |\eta + \rho k|^{2s}  d\rho \leq -c_s \left( t |\eta|^{2s} + t^{2s+1} |k|^{2s} \right),
\end{equation}
for some constant $c_s > 0$ depending only on $s$.

Using \eqref{se2} and \eqref{se3}, we find a constant $C > 0$ depending only on $c_s$ such that for any $T > 0$,
\begin{align*}
& \| t^{\frac{2s+1}{2s} m} \partial_{x_j}^m g \|_{L^1_k L^\infty_T L^2_v} = \int_{\mathbb{Z}^3} \sup_{t \leq T} t^{\frac{2s+1}{2s} m} \| k_j^m \hat{g}(t,k,\cdot) \|_{L^2_v}  d\Sigma(k) \\
&\leq \int_{\mathbb{Z}^3} \sup_{t \leq T} \left( t^{2s+1} |k|^{2s} \right)^{\frac{m}{2s}} \| e^{-c_s (t |\eta|^{2s} + t^{2s+1} |k|^{2s})} \mathscr{F}_{x,v} g_0(k,\eta + t k) \|_{L^2_\eta}  d\Sigma(k) \\
&\leq C^m (m!)^{\frac{1}{2s}} \int_{\mathbb{Z}^3} \sup_{t \leq T} \| \mathscr{F}_{x,v} g_0(k,\eta + t k) \|_{L^2_\eta}  d\Sigma(k) \leq C^m (m!)^{\frac{1}{2s}} \| g_0 \|_{L_k^1 L_v^2},
\end{align*}
where we used the inequality
\begin{equation*}
\forall t \geq 0, \quad \left( t^{2s+1} |k|^{2s} \right)^{\frac{m}{2s}} e^{-c_s t^{2s+1} |k|^{2s}} \leq C^m (m!)^{\frac{1}{s}}.
\end{equation*}
Similarly,
\begin{equation*}
\| t^{\frac{1}{2s} m} \partial_{v_j}^m g \|_{L^1_k L^\infty_T L^2_v} \leq C^m (m!)^{\frac{1}{2s}} \| g_0 \|_{L_k^1 L_v^2},
\end{equation*}
with $C$ independent of $T$. Hence, there exists $C > 0$ depending only on $c_s$ such that for any $T > 0$,
\begin{equation}\label{geest}
\| t^{\frac{1+2s}{2s} |\alpha| + \frac{1}{2s} |\beta|} \partial_x^\alpha \partial_v^\beta g \|_{L^1_k L^\infty_T L^2_v} \leq C^{|\alpha| + |\beta|} \left[ (|\alpha| + |\beta|)! \right]^{\frac{1}{2s}} \| g_0 \|_{L_k^1 L_v^2}.
\end{equation}
By Definition \ref{de}, this implies
\begin{equation*}
\forall \ t > 0, \quad g(t,\cdot,\cdot) \in \mathcal{G}^{\frac{1}{2s}}(\mathbb{T}_x^3 \times \mathbb{R}_v^3),
\end{equation*}
with the spatial and velocity Gevrey radii bounded below by constant multiples of $t^{\frac{1+2s}{2s}}$ and $t^{\frac{1}{2s}}$, respectively.

For the Boltzmann equation with hard potentials, a more appropriate model is
\begin{equation}\label{fFP1}
\begin{cases}
\partial_t g + v \cdot \partial_x g + \comi v^\gamma (-\Delta_v)^s g = 0, & 0 < \gamma \leq 1, \ 0 < s < 1, \\
g|_{t=0} = g_0 \in L_{x,v}^2,
\end{cases}
\end{equation}
where $\comi v = (1 + |v|^2)^{\frac12}$. Since the coefficient $\comi v^\gamma$ is analytic but not ultra-analytic for $0 < \gamma \leq 1$, ultra-analyticity may not be achievable, and analyticity is likely the best possible regularity. This, together with \eqref{geest}, suggests that the optimal smoothing effect for \eqref{fFP1} occurs in the Gevrey space $\mathcal{G}^\tau$, where
\begin{equation*}\label{tau}
\tau := \max\left\{ 1, \frac{1}{2s} \right\}.
\end{equation*}
Moreover, the absence of large-time decay for $\comi v^\gamma$ prevents uniform-in-time estimates on the velocity analyticity radius. These observations lead to the following expected estimates for \eqref{fFP1}: for any $T > 0$, there exists a constant  $C(T) > 0$ depending on $T,$ such that 
\begin{equation}\label{localestimate}
\forall   \alpha, \beta \in \mathbb{Z}_+^3, \ \| t^{\frac{1+2s}{2s} |\alpha| + \frac{1}{2s} |\beta|} \partial_x^\alpha \partial_v^\beta g \|_{L^1_k L^\infty_T L^2_v} \leq C(T)^{|\alpha| + |\beta|} \left[ (|\alpha| + |\beta|)! \right]^\tau \| g_0 \|_{L_k^1 L_v^2}.
\end{equation}
And moreover, there exists a constant $C$ independent of $T$, such that 
\begin{equation}\label{shpest}
\forall\  \alpha \in \mathbb{Z}_+^3, \ \forall \  T > 0, \quad \| t^{\frac{1+2s}{2s} |\alpha|} \partial_x^\alpha g \|_{L^1_k L^\infty_T L^2_v} \leq C^{|\alpha|} (|\alpha|!)^\tau \| g_0 \|_{L_k^1 L_v^2}.
\end{equation}
These estimates, though inspired by \eqref{geest}, require  non-trivial proofs in the case $\gamma > 0$.

As shown below, we establish a local estimate for the Boltzmann equation \eqref{3} that is analogous to \eqref{localestimate}. Although the global estimate \eqref{shpest} remains open, we derive a weaker version of it; the details  are provided in the next subsection.

\subsection{Main results}

We begin by recalling recent well-posedness results for the Boltzmann equation. Suppose the initial datum $f_0(x,v)$ in \eqref{3} satisfies
\begin{equation}\label{smallass}
\| f_0 \|_{L^1_k L^2_v} \leq \epsilon
\end{equation}
for some $\epsilon > 0$, and
\begin{equation}\label{ini-law}
\int_{\mathbb{T}^3} \int_{\mathbb{R}^3} (1, v, |v|^2) \sqrt{\mu} f_0(x,v)  dv  dx = 0.
\end{equation}
By \cite[Theorem 2.1]{MR4230064}, if $\epsilon$ is sufficiently small, then the Boltzmann equation \eqref{3} admits a unique global mild solution $f \in L^1_k L^\infty_T L^2_v$ for any $T > 0$, satisfying \eqref{law} and that 
	\begin{equation} \label{v2}
		\begin{aligned}
	&\int_{\mathbb{Z}^3}   \sup_{ t\geq 0} \norm{ \hat  f (t,k)}_{L^2_v}
 	d\Sigma(k) +\int_{\mathbb{Z}^3}  \bigg( \int^{+\infty}_0    \normm{ \hat f(t,k)}^2dt  \bigg)^{\frac12}d\Sigma(k)   \leq C \epsilon
	\end{aligned}
	\end{equation}
	and  
		\begin{equation}\label{longtime}
	\forall \  t\geq 0,\quad 	\int_{\mathbb{Z}^3}   \norm{ \hat  f (t,k)}_{L^2_v}
 	d\Sigma(k)\leq C \epsilon e^{-A t},
	\end{equation}
for some constants $C, A > 0$.

Besides, by \cite[Proposition 3.7]{MR4356815}, if the constant  $\epsilon$ in \eqref{smallass} is sufficiently small, the mild solution also satisfies
\begin{equation}\label{v33}
    \int_{\mathbb{Z}^3}  \comi{k}^{\frac{s}{1+2s}} \bigg(\int^{1}_0    \norm{ \hat f(t,k)}^2_{L^2_v}dt  \bigg)^{\frac12}d\Sigma(k)   \leq C \epsilon.
\end{equation}
Recently, \cite{MR4930523} showed that this mild solution is analytic or sharp Gevrey regular for positive times: for any $T > 0$ and $\lambda > \frac{1+2s}{2s}$, there exists $C_{\lambda,T} > 0$ such that
\begin{equation}\label{sgv}
\int_{\mathbb{Z}^3} \sup_{t \leq T} t^{(\lambda+1)|\alpha| + \lambda |\beta|} \| \mathscr F_x (\partial_x^\alpha \partial_v^\beta f)(t,k) \|_{L^2_v}  d\Sigma(k) \leq  C_{\lambda,T}^{|\alpha| + |\beta|+1} \left[ (|\alpha| + |\beta|)! \right]^\tau,
\end{equation}
with $\tau = \max\{1, \frac{1}{2s}\}$. Note that $C_{\lambda,T}$ depends on $T$, so the radius estimate is local in time. Moreover, \eqref{sgv} is not optimal compared to \eqref{localestimate}.

In this work, we first improve the local estimate \eqref{sgv} to a sharp short-time estimate analogous to \eqref{localestimate}, and then derive a large-time estimate on the spatial regularity radius motivated by \eqref{shpest}.  Throughout, we assume the initial datum $f_0$ satisfies \eqref{smallass} and \eqref{ini-law}, and we consider the unique solution $f \in L^1_k L^\infty_T L^2_v$, constructed in \cite{MR4230064}, which satisfies \eqref{v2}, \eqref{longtime}, \eqref{v33} and \eqref{sgv}. The main results can be stated as follows. 

\begin{theorem}[Local estimate] \label{local-gxt}

Assume that the cross-section satisfies hypotheses \eqref{kern} and \eqref{angu} with parameters $0 \leq \gamma \leq 1$ and $0 < s < 1$, and that the initial datum $f_0$ satisfy \eqref{smallass} and \eqref{ini-law}. Let $f \in L^1_k L^\infty_T L^2_v$    be a solution to the Boltzmann equation \eqref{3} for  any $T > 0$, which satisfies \eqref{law} and the estimates \eqref{v2}, \eqref{longtime}, \eqref{v33} and \eqref{sgv}. If the constant $\epsilon$ in \eqref{smallass} is sufficiently small, then   for all $\alpha, \beta \in \mathbb{Z}_+^3,$ it holds that
\begin{equation}\label{gx++}
\begin{aligned}
&\int_{\mathbb{Z}^3} \sup_{t \leq 1} t^{\frac{1+2s}{2s} |\alpha| + \frac{1}{2s} |\beta|} \| \mathscr F_x (\partial_x^\alpha \partial_v^\beta f)(t,k) \|_{L_v^2}  d\Sigma(k) \\
&\qquad + \int_{\mathbb{Z}^3} \left( \int_0^1 t^{\frac{1+2s}{s} |\alpha| + \frac{1}{s} |\beta|} \normm{ \mathscr F_x (\partial_x^\alpha \partial_v^\beta f)(t,k) }^2  dt \right)^{\frac12}  d\Sigma(k) \\
&\leq  C^{|\alpha| + |\beta|+1} \left[ (|\alpha| + |\beta|)! \right]^\tau,
\end{aligned}
\end{equation}
where $\tau = \max\left\{1, \frac{1}{2s}\right\}$ and the constant $C > 0$ depends only on the parameters $s$ and  $\gamma$.
\end{theorem}

\begin{theorem}[Global estimate] \label{gxt}
Under the hypotheses of Theorem \ref{local-gxt}, there exists a constant $C > 0$ depending on $s$ and $\gamma$ such that, for all $\alpha \in \mathbb{Z}_+^3$,
\begin{multline}\label{opg}
\int_{\mathbb{Z}^3} \sup_{t \geq 1} t^{\frac{1+2s}{2s} |\alpha|} \| \widehat{\partial_x^\alpha f}(t,k) \|_{L^2_v}  d\Sigma(k) \\
+ \int_{\mathbb{Z}^3} \left( \int_1^{+\infty} t^{\frac{1+2s}{s} |\alpha|} \normm{ \widehat{\partial_x^\alpha f}(t,k) }^2  dt \right)^{\frac12}  d\Sigma(k) \leq C^{|\alpha| + 1} (|\alpha|!)^{\frac{1+2s}{2s}},
\end{multline}
and
\begin{multline}\label{opa}
\int_{\mathbb{Z}^3} \sup_{t \geq 1} t^{|\alpha|} \| \widehat{\partial_x^\alpha f}(t,k) \|_{L^2_v}  d\Sigma(k) \\
+ \int_{\mathbb{Z}^3} \left( \int_1^{+\infty} t^{2|\alpha|} \normm{ \widehat{\partial_x^\alpha f}(t,k) }^2  dt \right)^{\frac12}  d\Sigma(k) \leq C^{|\alpha| + 1} (|\alpha|!)^\tau,
\end{multline}
where $\tau = \max\left\{1, \frac{1}{2s}\right\}$.
\end{theorem}

\begin{remark}
The Landau equation, obtained as the grazing limit of the Boltzmann equation, admits similar results with $s$ replaced by $1$.
\end{remark}

\subsection{Related literature}

There is extensive literature on the well-posedness of the Boltzmann equation (see, for example, \cite{MR2863853, MR2793203, MR3177640, MR2847536, MR4526062, MR4230064, MR2784329, MR2095473, MR2972454, MR2679369, MR3213301} and references therein). Here we only mention the recent    works \cite{MR4230064, MR4526062} on close-to-equilibrium problems with exponential or polynomial tails, and \cite{MR4895262} on large-data solutions.

We are interested in the regularizing effect  for kinetic equations with singular collision kernels. It is well-established that angular singularity induces velocity diffusion, as shown in \cite{MR1765272}, where the Boltzmann operator with angular singularity is shown to behave locally like the fractional Laplacian $(-\Delta_v)^s$. Spatial diffusion arises from the interaction between the collision and transport operators (cf. \cite{MR2679369, MR3950012}). Thus, the Boltzmann equation is expected to exhibit regularization similar to the heat equation.

The mathematical study of this regularization goes back to  L. Desvillettes \cite{MR1324404} for a one-dimensional model. The intrinsic velocity diffusion was later established by Alexandre-Desvillettes-Villani-Wennberg \cite{MR1765272}. Since then, substantial progress has been made; we refer here to \cite{MR2506070, MR2679369, MR2847536, MR2807092, MR2784329, MR4107942} for results on 
$C^\infty$ or Sobolev regularization. Smoothing in more regular Gevrey spaces was established in \cite{MR3348825, MR4147430, MR4356815, MR4375857}, based on the hypoelliptic structure explored in \cite{MR1949176, MR3950012, MR3456819, MR3102561, MR2763329, MR2467026, MR3193940,MR2885564}. For general large-data settings, conditional regularity has been extensively investigated in \cite{MR4033752, MR4433077, MR4229202, MR4049224, MR3551261, MR4431674} and references therein, using techniques from the De Giorgi–Nash–Moser theory and the averaging lemma.

\subsection{Methodology}

Our approach is inspired by \cite{MR4930523}, which uses directional derivatives $H_{\delta_j}$ defined by
\begin{equation*}
H_{\delta_j} = \frac{1}{\delta_j + 1} t^{\delta_j + 1} \partial_{x_1} + t^{\delta_j} \partial_{v_1}, \quad j = 1,2.
\end{equation*}
Compared with the classical derivatives $\partial_{v},$ the main advantage of $ H_{\delta_j}$ is that   the spatial derivatives are not involved in the commutator  between $H_{\delta_j}$ and the transport operator, that is,  
  \begin{equation*}\label{keyob}
  	[H_{\delta_j}, \,\,  \partial_t+v\,\cdot\,\partial_x ]=-\delta_j t^{\delta_j-1}\partial_{v_1},
  \end{equation*}
 and  more generally, 
 \begin{equation}\label{chk}
\forall\ m\geq 1,\quad 	[H_{\delta_j}^m, \,\,  \partial_t+v\,\cdot\,\partial_x ]=-\delta_j m t^{\delta_j-1} \partial_{v_1} H_{\delta_j}^{m-1}.
\end{equation}
 Moreover, the classical derivatives can be generated by the linear combination of $H_{\delta_1}$ and $H_{\delta_2}$. 
This  enables to control the classical derivatives in terms of the directional derivatives in $H_{\delta_1} $ and $H_{\delta_2} $.  Note in \cite{MR4930523} the restriction that    $\delta_j> \frac{1+2s}{2s}, j=1,2,$ is required to estimate  the term $\norm{\widehat{H_{\delta_j}^m f}(t,k)}_{L^2_{v}}$ near $t=0.$ 

In this work, we instead use the following two operators:
\begin{equation}\label{p12}
P_1 = \frac{2s}{1+2s} t^{\frac{1+2s}{2s}} \partial_{x_1} + t^{\frac{1}{2s}} \partial_{v_1} \  \text{and} \ P_2= t^{\frac{1}{2s}} \partial_{v_1}.
\end{equation}
Although
\begin{equation*}
[\partial_{v_1}^m, v \cdot \partial_x] = m \partial_{x_1} \partial_{v_1}^{m-1}
\end{equation*}
involves the spatial derivatives, the inner product
\begin{equation*}
\left( \mathscr F_x(m \partial_{x_1} \partial_{v_1}^{m-1} f), \mathscr F_x(\partial_{v_1}^m f) \right)_{L_v^2}
\end{equation*}
can be controlled   via the relation
\begin{equation*}
\partial_{x_1} = \frac{1+2s}{2s} t^{-\frac{1+2s}{2s}} P_1 - \frac{1+2s}{2s}t^{-1} \partial_{v_1}.
\end{equation*}
This allows us to 
ignore the transport part $v\cdot\partial_x$ at moment and first estimate the $v$-derivatives by treating the equation as a fractional heat equation:
\begin{equation}\label{fh}
\partial_t g + (-\Delta_v)^s g = 0, \quad g|_{t=0} = g_0.
\end{equation}
The solution $g = e^{-t(-\Delta_v)^s} g_0$ to \eqref{fh} satisfies
\begin{equation*}
\| (t^{\frac{1}{2s}} \partial_{v_j})^m g \|_{L_v^2} \leq C^m (m!)^{\frac{1}{2s}} \| g_0 \|_{L_v^2} ,
\end{equation*}
which motivates the choice of operator $t^{\frac{1}{2s}}\partial_{v_1}$ in \eqref{p12}. Meanwhile, the operator 
$P_1$ in \eqref{p12} is constructed following a similar inspiration from \eqref{chk} for 
 $H_{\delta_j},$ which satisfies 
\begin{equation}\label{pmc}
\forall\ m \geq 1, \quad [P_1^m, \partial_t + v \cdot \partial_x] = -\frac{1}{2s} m t^{\frac{1-2s}{2s}} \partial_{v_1} P_1^{m-1}.
\end{equation}
The two operators in \eqref{p12} work well for deriving local-in-time estimates for the Boltzmann equation \eqref{3}.

For global estimates, we replace $t^{\frac{1}{2s}} \partial_{v_1}$ with $\partial_{v_1}$ to overcome the lack of large-time decay in $\mu$, and accordingly replace $P_1$ with
\begin{equation*}
H = t \partial_{x_1} + \partial_{v_1}.
\end{equation*}
Estimate \eqref{opa} will then follow  from quantitative bounds involving $H$ and $\partial_{v_1}$. To prove \eqref{opg} with an improved radius, we use hypoelliptic estimates from \cite{MR3950012}, even though the resulting regularity $\mathcal{G}^{\frac{1+2s}{2s}}$ is not optimal. Further details are provided in Section \ref{sec:global}.

\subsection{Organization of the paper}

The rest of the paper is organized as follows. In Section \ref{sec:prelim}, we list  some  frequently used estimates.   Section \ref{sec:local} is devoted to the proof of Theorem \ref{local-gxt}, while Theorems \ref{gxt} is proven in Sections \ref{sec:macros} and \ref{sec:global}.

\section{Preliminaries}\label{sec:prelim}

  This section collects several preliminary estimates to be used in the sequel. Let  $\mathcal L$ be the linearized Boltzmann operator in \eqref{colli}.  The null space  of the operator $\mathcal{L}$ is given by
	  \begin{equation*}
	  \mathcal{N}=Ker \mathcal{L}=\text{span}\big\{\sqrt{\mu}, v_1\sqrt{\mu}, v_2\sqrt{\mu},v_3\sqrt{\mu},|v|^2\sqrt{\mu} \big\},
	  \end{equation*}
	  where $\mu$ is the global Maxwellian defined in \eqref{gmu}. 
	  We further define   the  orthogonal projection  $\mathbf{P}$  from $L^2(\mathbb R^3_v)$
	  onto $\mathcal{N}$:
	  \begin{eqnarray}\label{eqorp1}
	  \mathbf{P}f =\big \{a(t,x)+b(t,x)\cdot v+c(t,x)(|v|^2-3)\big \}\mu^{\frac12}, 
	  \end{eqnarray}
	  with coefficients
	  \begin{equation*}
	  	 a=\int_{\mathbb R^3}\mu^{\frac12}f dv,\quad 
	  b=\int_{\mathbb R^3}v\mu^{\frac12}f dv \ \textrm{ and }\ 
	  c =\frac{1}{6}\int_{\mathbb R^3}(|v|^2-3)\mu^{\frac12}f dv.
	  \end{equation*}

\subsection{Coercivity property  and trilinear estimate}\label{subsec:coer}
Let $\normm{\cdot}$ be the triple norm defined by \eqref{trinorm}. Then by the coercivity of $\mathcal L$ and    identification of the triple norm   (cf.\cite[Propositions 2.1 and 2.2]{MR2863853} for instance),  we have \begin{equation}\label{rela}
\forall\ h\in \mathcal S(\mathbb R_v^3),\quad  C_1\normm {\{\mathbf{I}-\mathbf{P}\}h}^2 \leq -\inner{\mathcal L h, h}_{L^2 _v},
  \end{equation}
  and  for Maxwellian molecules  and hard potential cases (i.e.$\gamma\geq 0$),
 \begin{equation}
	\label{+lowoftri}
\forall\ h\in \mathcal S(\mathbb R_v^3),\quad 	C_1 \norm{h}_{H^s_{v}} \leq \normm{h},
\end{equation}
where $s$ is the parameter in \eqref{angu}, and $C_1>0$ is a   constant, and $\mathcal S(\mathbb R_v^3)$   denotes the Schwartz space in  $v$. These estimates remain valid for any $f$ such that $\normm f<+\infty.$  As an immediate consequence of \eqref{rela}, one has 
\begin{equation}\label{+rela}
\forall\ h\in \mathcal S(\mathbb R_v^3),\quad  C_1\normm {h}^2 \leq -\inner{\mathcal L h, h}_{L^2 _v}+\norm{h}_{L_v^2}^2.
  \end{equation}

 For simplicity of notations,   we will use $C_0$ to denote a generic positive constant that may vary from line to line, possibly enlarged when necessary.
We now recall  a trilinear estimate for the collision operator (see \cite[Theorem 2.1]{MR2784329}): for any $f,g,h\in\mathcal S (\mathbb R_v^3),$
\begin{equation}\label{trin}
\big|\big(  \mathcal T ( f, g, \mu^{\frac12}),  \  h\big)_{L^2_v}\big|=\big|\big(  \Gamma ( f, g),  \  h\big)_{L^2_v}\big|\leq C_0\norm{f}_{L_v^2}\cdot \normm{ g} \cdot \normm{ h},\end{equation}
recalling $\mathcal T$ is defined in \eqref{matht}.

We will frequently use the Fourier-transformed version of this estimate.
 Taking the partial Fourier transform in 
$x$ and applying    \cite[Lemma 3.2]{MR4230064},  we obtain that, for any $k\in\mathbb Z^3$ and for   any  $f,g,h\in L_k^1(\mathcal S(\mathbb R_v^3))$, 
 \begin{multline}\label{upptrifour}
  \big|\big( \hat{\mathcal T} ( \hat f, \hat g, \mu^{\frac12}),  \  \hat h\big)_{L^2_v}\big| =	\big|\big( \hat \Gamma{(\hat f(k), \hat g(k))}, \hat h(k)\big)_{L^2_v}\big |\\
  \leq C_0 \normm{\hat h(k)} \int_{\mathbb Z^3 } \norm{\hat f(k-\ell)}_{L^2_v}   \normm{\hat g(\ell)}    d\Sigma(\ell).
  \end{multline}
  More generally,  consider a function   $\omega=\omega(v)$ satisfying the pointwise bound
  \begin{equation}\label{contild}
  \forall\ v\in\mathbb R^3, \quad |  \omega (v) | \leq
\tilde C  \mu(v)^{\frac14} 
\end{equation}
for some constant $\tilde C>0.$  Then, by the same argument used to prove \eqref{trin} with $\mu^{\frac12}$ therein replaced by $\omega,$ we obtain
\begin{equation*}
	 \forall\, f,g,h\in\mathcal S (\mathbb R_v^3), \quad
\big|\big(  \mathcal T ( f, g,   \omega),  \  h\big)_{L^2_v}\big|\leq C_0 \tilde C \norm{f}_{L_v^2}\cdot \normm{ g} \cdot \normm{ h}.
\end{equation*}
As a result, similar to \eqref{upptrifour}, we   perform  the partial Fourier transform in $x$ variable to conclude
\begin{equation}\label{ketres}
 		\big|\big( \hat{\mathcal{T}}(\hat{f}, \hat g,  \omega), \hat h(k)\big)_{L^2_v}\big |\leq C_0 \tilde C \normm{\hat h(k)} \int_{\mathbb Z^3 } \norm{\hat f(k-\ell)}_{L^2_v}   \normm{\hat g(\ell)}    d\Sigma(\ell)
 \end{equation}
 with $\tilde C$   the constant in \eqref{contild}.   In particular, if  the function $g$ in \eqref{ketres}  depends only on  $v$ variable, then \eqref{ketres} reduces to 
\begin{equation*}
		\big|\big( \hat{\mathcal{T}}(\hat{f},   g,  \omega), \hat h(k)\big)_{L^2_v}\big |\leq C_0 \tilde C \norm{\hat f(k)}_{L^2_v}   \normm{  g } \times   \normm{\hat h(k)}. 	\end{equation*}
Combined with the inequality (cf. \cite[Proposition 2.2]{MR2863853})
\begin{equation*}
	\normm{ g} \leq \tilde c \norm{ (1+|v|^{2+\gamma}-\Delta_v)  g}_{L_v^2 }
\end{equation*} 
for some constant $\tilde c>0$, we obtain (after possibly enlarging 
 $C_0$) 
\begin{equation*}
	\big|\big( \hat{\mathcal{T}}(\hat{f},   g,  \omega), \hat h(k)\big)_{L^2_v}\big |\leq C_0 \tilde C \norm{\hat f(k)}_{L^2_v}  \norm{ (1+|v|^{2+\gamma}-\Delta_v)  g}_{L_v^2 }\normm{\hat h(k)}. 
\end{equation*} 
 Now, suppose 
$g=g(v)\in \mathcal S(\mathcal R_v^3)$ depends only on 
$v$ and satisfies the following:   there exists a constant $\tilde C_\gamma>0$, depending only on the number $\gamma$ in \eqref{kern},  such that 
\begin{equation}\label{contild+++}
  \forall\ v\in\mathbb R^3, \ \forall \ m\in\mathbb Z_+, \quad |  (1+|v|^{2+\gamma}-\Delta_v)\partial_v^m g(v) | \leq \tilde C_\gamma L_m
   \mu(v)^{\frac18},
\end{equation}  
with $L_m$   constants depending only on $m.$
Then, again possibly enlarging 
$C_0,$
\begin{equation}\label{trisole}
	\forall\ m\in\mathbb Z_+,\quad 	\big|\big( \hat{\mathcal{T}}(\hat{f},  \partial_v^m g,  \omega), \hat h(k)\big)_{L^2_v}\big |\leq C_0 \tilde C  L_m \norm{\hat f(k)}_{L^2_v}     \normm{\hat h(k)}, \end{equation}
    with  $\tilde C$ as in \eqref{contild}.   Similarly,  if $ \omega=\omega(v)$ and $g=g(v) $ satisfy \eqref{contild} and \eqref{contild+++}, respectively, then
 \begin{equation}\label{tretmate}
	\forall\ m\in\mathbb Z_+,\quad 	\big|\big( \hat{\mathcal{T}}(\partial_v^m g,  \hat{f},    \omega), \hat h(k)\big)_{L^2_v}\big |\leq C_0 \tilde C  L_m \normm{\hat f(k)}\times     \normm{\hat h(k)}. \end{equation}
Finally, we recall a key estimate for controlling the nonlinear term  $\Gamma(f,g)$ (cf. \cite[Lemma 2.5]{MR4356815}). 
 For any integer $j_0 \geq 1$  and any   $f_j \in L_k^1L_T^{\infty}L_v^2$ and any $g_j$ such that $\normm{ g_j} \in L_k^1L_T^{2}$ with $1 \leq j \leq j_0$,
  the following inequality holds:
	\begin{equation}\label{MF}
	\begin{aligned}
	&\int_{\mathbb{Z}^3}\bigg[\int_{t_1}^{t_2}\Big(\int_{\mathbb{Z}^3}\sum_{1 \leq j \leq j_0}\norm{\hat{f}_j(t,k-\ell)}_{L^2_v}\normm{  \hat{g}_j(t,\ell)}d\Sigma(\ell)\Big)^{2}dt\bigg ]^{\frac12}d\Sigma(k)\\
	& \leq  \sum_{j=1}^{j_0}\Big(\int_{\mathbb{Z}^3}\sup\limits_{t_1\leq t\leq t_2}\norm{\hat{f}_j(t,k)}_{L^2_v}d\Sigma(k)\Big)\int_{\mathbb{Z}^3}\Big(\int_{t_1}^{t_2}\normm{  \hat{g}_j(t,k)}^{2}dt\Big)^{1\over2}d\Sigma(k).
	\end{aligned}
	\end{equation}
	 This estimate follows directly from Minkowski’s inequality and Fubini’s theorem; see \cite[Lemma 2.5]{MR4356815} for details.

\subsection{Technical lemmas}
 We list several lemmas, omitting the proofs, as they will be frequently used in the subsequent discussion.
	 
	\begin{lemma}[Lemma 2.6 in \cite{MR4356815}]\label{itea}
		For any $m\geq 1$ the following estimate 
		\begin{equation*}
		\langle k \rangle^m
		\leq   \sum_{j=1}^{m-1}{m\choose j}\langle k -\ell\rangle^j
		\langle \ell\rangle^{m-j}+2\langle k -\ell\rangle^m+2\langle  \ell\rangle^m\leq   2\sum_{j=0}^{m}{m\choose j}\langle k -\ell\rangle^j
		\langle \ell\rangle^{m-j}
		\end{equation*}
		holds true for any $k,\ell \in \mathbb{Z}^3$, with
		the convention that 	the summation term over $1\leq j \leq m-1$ on the 
		right hand side disappears when $m=1$.
	\end{lemma}

    \begin{lemma}\label{lem:fm}
		Let $A_j, j=1,2,$ be  two Fourier multipliers  with symbols $\varphi_j(k,\eta)=p_j(t)k+q_j(t)\eta$, that is,
 \begin{eqnarray*}
 \mathscr F_{x,v}	(A_j h)(k,\eta)=\varphi_j(k,\eta)\mathscr F_{x,v}h(k,\eta),
 \end{eqnarray*}
 where $\mathscr F_{x,v} h   $  denotes  the full Fourier transform in $(x,v)\in\mathbb T^3\times\mathbb R^3.$ Then for any $0 \leq t_1 < t_2 $
  \begin{equation}\label{+pse1}
  \begin{aligned}
	&\int_{\mathbb Z^3} \sup_{t_1 \leq t\leq t_2} \norm{\mathscr{F}_x\big( (A_1+A_2)^m h\big)(t,k)}_{L^2_{v}} d\Sigma(k)\\
	&\qquad\qquad 
	 \leq 2^{m} \sum_{j=1}^2\int_{\mathbb Z^3} \sup_{t_1\leq t\leq t_2} \norm{\mathscr{F}_x\big(  A_j^mh\big)(t,k)}_{L^2_{v}} d\Sigma(k)
	\end{aligned}
 \end{equation}
 and 
  \begin{equation}\label{pse1+}
  \begin{aligned}
	& \int_{\mathbb Z^3}\bigg(\int_{t_1}^{t_2} \norm{\mathscr{F}_x\big( (A_1+A_2)^mh\big)(t,k)}_{L^2_{v}}^2dt\bigg)^{\frac12} d\Sigma(k)\\
	&\qquad \leq 2^{m}\sum_{j=1}^2 \int_{\mathbb Z^3}\bigg(\int_{t_1}^{t_2} \norm{\mathscr{F}_x( A_j^m h)(t,k)}_{L^2_{v}}^2dt\bigg)^{\frac12} d\Sigma(k).
	\end{aligned}
 \end{equation}
 For $m,n\in\mathbb Z_+,$ $r\in \mathbb R,$
 \begin{equation}\label{fmn}
 	\norm{\widehat{A_1^mA_2^n h}}_{H_v^r}\leq \norm{\widehat{A_1^{m+n}  h}}_{H_v^r}+\norm{\widehat{A_2^{m+n}  h}}_{H_v^r}.
 \end{equation}
 where $\widehat h$ denotes the partial Fourier transform with respect to $x$ variable. 
	\end{lemma}
	
\begin{proof}
	This follows from direct verification (see Appendix \ref{app:inequ}). 
\end{proof}
	
	\begin{lemma}
		[Commutator estimate]\label{lem:comestimate} Let $m\geq 1$ be a given integer and let $\psi=\psi(D)$ be a Fourier multiplier satisfying
    \begin{equation}\label{psk}
      \norm{\widehat{\psi h}}_{L_v^2}\leq C_2\norm{\hat{h}}_{L_v^2} \ \textrm{ and }\    \normm{\widehat{\psi h}} \leq C_2 \normm{\hat{h}},
    \end{equation}
 where $C_2>0$ is a constant.  Define the differential operator	 
\begin{equation*}
	P=\xi(t)\partial_{x_1}+t^{\delta}\partial_{v_1},
	\end{equation*}
	where $\delta\geq 0$ is a parameter  and  $\xi=\xi(t)$  is a given continuous function on  $\mathbb R.$    Suppose    $f\in L^1_kL^{\infty}_TL^2_v$   
	for  any $T>0$ is a global  solution to  
  the Boltzmann equation \eqref{3}, and  assume that for given $0\leq t_1<t_2\leq T$ and  any   
	  $j\leq m-1$ 
\begin{multline*}\label{250524}
\int_{\mathbb{Z}^3}\sup\limits_{t_1\leq  t \leq t_2}\norm{\widehat{P^{j}f}(t,k)}_{L^2_v}d\Sigma(k)\\+ \int_{\mathbb{Z}^3}\left(\int_{t_1}^{t_2}\normm{  \widehat{P^{j}f}(t,k)}^{2}dt\right)^{\frac12}d\Sigma(k) \leq \frac{\eps_0{\tilde A}^{j}( j!)^{\tau}}{(j+1)^2},
\end{multline*}
where $\tau = \max\{1, \frac{1}{2s}\}$, and $ \eps_0, \tilde A>0$ are two given constants.  If   $\tilde A \geq 4t_2^\delta $,  then   there exists a constant $C$,  depending only on  $C_0, C_1, C_2$  in Subsection \ref{subsec:coer} and \eqref{psk}  but independent of $m, t_1,t_2$,  such that  for any $\eps >0$  
\begin{equation*}
\begin{aligned}
&\int_{\mathbb{Z}^3}\left(\int_{t_1}^{t_2}\big|\big(\mathscr F_{x}(P^{m}\Gamma(f,f)), \widehat{\psi P^{m}f}\big)_{L^2_v}\big|dt\right)^{\frac12}d\Sigma(k)
\\ & \leq    C{\eps}^{-1}\eps_0   \int_{\mathbb{Z}^3}\sup\limits_{t_1 \leq  t \leq t_2}\norm{\widehat{P^{m}f}(t,k)}_{L^2_v}d\Sigma(k)\\
&\quad +    \inner{\eps+C{\eps}^{-1}\eps_0}  \int_{\mathbb{Z}^3}\Big (\int_{t_1}^{t_2}\normm{ \widehat{P^{m}f}(t,k)}^{2}dt\Big )^{\frac12}d\Sigma(k) +C \eps^{-1}  \eps_0 \frac{\eps_0{\tilde A}^{m} (m!)^\tau}{(m+1)^2} 
\end{aligned}
\end{equation*}
and 
\begin{equation*}
	\begin{aligned}
	&\int_{\mathbb{Z}^3}\left(\int_{t_1}^{t_2}\big|\big(\mathscr F_{x}([P^{m},\  \mathcal{L}]f), \ \widehat{ \psi P^{m}f}\big)_{L^2_v}\big|dt\right)^{\frac12}d\Sigma(k)
	\\ & \leq   \eps\int_{\mathbb{Z}^3}\left(\int_{t_1}^{t_2}\normm{ \widehat{P^{m}f}(t,k)}^{2}dt\right)^{\frac12}d\Sigma(k)+C{\eps}^{-1}  \frac{\eps_0{\tilde A}^{m-1}(m!)^\tau}{(m+1)^2}.
	\end{aligned}
	\end{equation*} 
\end{lemma}

\begin{proof}
The proof is analogous to the ones given for
	 \cite[Propositions 3.1 and 3.3]{MR4930523}. We omit it for brevity. 
\end{proof}

\subsection{Subelliptic estimates}
The subelliptic structure, arising from the non-trivial interaction between the transport and collision operators, plays a crucial role in both the local and global estimates that follow. 

\begin{proposition}\label{sub:ell}
 Let $g$ be a given function, and  suppose $h$ satisfies the   linear Boltzmann equation
 \begin{equation}
     \label{linbol}
     \partial_t h+v\cdot\partial_x h-\mathcal L h=g,\quad h|_{t=0}=h_0. 
 \end{equation}
 Then there exists a constant $C>0$ and a Fourier multiplier $\mathcal M$ satisfying  \eqref{psk}, such that  for any $0\leq t_1<t_2  \leq 1,$   
 \begin{equation*}
     \begin{aligned}
      & \int_{\mathbb Z^3}\sup_{t\in [t_1,t_2]}  
	 \norm{\hat{h}}_{L^2_v}d\Sigma(k)+      \int_{\mathbb Z^3}   \Big(\int_{t_1}^{t_2}   	 \normm{\hat{h}(t,k)}^2dt \Big)^{\frac12}d\Sigma(k) \\
     &\qquad \qquad+      \int_{\mathbb Z^3} \comi{k}^{\frac{s}{1+2s}} \Big(\int_{t_1}^{t_2}   
	 \norm{\hat{h}(t,k)}^2_{L^2_v}dt \Big)^{\frac12}d\Sigma(k)\\
     &\leq  C\int_{\mathbb Z^3} 
	 \norm{\hat{h}(t_1,k)}_{L^2_v}d\Sigma(k) +C\int_{\mathbb Z^3}    \bigg(\int_{t_1}^{t_2}  \big| \big(\hat{g},  \widehat{\mathcal M h}\big)_{L_v^2}\big| dt \bigg)^{\frac12}d\Sigma(k),
     \end{aligned}
 \end{equation*}
 where  $C$ is a positive constant.
\end{proposition}

\begin{corollary}
 	\label{cor:sub}
 Let $f$ be a global solution satisfying the conditions in Theorem \ref{local-gxt}. Then there exists a constant $C>0$,  such that  for any $m\in\mathbb Z_+$ and any $0\leq t_1<t_2,$    
\begin{equation}\label{sube}
	\begin{aligned}
  &\int_{\mathbb Z^3} \comi{k}^{m+\frac{s}{1+2s}} \Big(\int_{t_1}^{t_2}   t^{\frac{1+2s}{s}m} 
	 \norm{\hat{f}}^2_{L^2_v}dt \Big)^{\frac12}d\Sigma(k) \\
     &\leq  C \int_{\mathbb{Z}^3} \bigg[\sup\limits_{t_1 \leq t \leq t_2} t^{\frac{1+2s}{2s}m} \comi{k}^{m} \norm{ \hat{f} }_{L^2_v} +   \comi{k}^{m}\bigg(\int_{t_1}^{t_2}     t^{\frac{1+2s}{s}m} \normm{\hat{f} }^2 dt \bigg)^{\frac12}\bigg]d\Sigma(k)\\
	&\quad+C\int_{\mathbb Z^3}   \bigg(\int_{t_1}^{t_2}   m   t^{\frac{1+2s}{s}m-1} \comi{k}^{2m}\norm{\hat{f} }_{L_v^2}^2 dt \bigg)^{\frac12}d\Sigma(k)   \\
    &\quad+C
	\sum_{j=0}^{m}\binom{m}{j}\bigg(\int_{\mathbb{Z}^3}\sup\limits_{t_1\leq t\leq t_2} t^{\frac{1+2s}{2s}j} \comi {k}^{j}\norm{\hat{f}(t,k)}_{L^2_v}d\Sigma(k)\bigg)\\
	&\qquad\qquad\quad\qquad\quad\quad  \times \int_{\mathbb{Z}^3}\Big(\int_{t_1}^{t_2} t^{\frac{1+2s}{s}(m-j)} \comi {k}^{2(m-j)}\normm{  \hat{f}(t,k)}^{2}dt\Big)^{1\over2}d\Sigma(k).  
	\end{aligned}
\end{equation}
 \end{corollary}

For the sake of completeness, we provide a sketch of the proof of Proposition \ref{sub:ell}  and Corollary \ref{cor:sub} 
in Appendix \ref{app:sub}; interested readers may refer  to \cite{MR3950012,MR4356815} for the full details.

\section{Local-in-time sharp estimate on radius}\label{sec:local}

In this section, we prove Theorem \ref{local-gxt}, establishing a local-in-time radius estimate that appears to be sharp, as it agrees with the observation \eqref{localestimate} for the toy model \eqref{fFP1}.

The higher-order regularity of $L_k^1 L_T^\infty L_v^2 $-mild solutions for positive times was established in \cite{MR4356815} using subtle pseudo-differential techniques. In particular, with the regularization operators introduced in \cite{MR4356815}, one can show that for any  $m\in\mathbb Z_+$ and any $T>0,$
\begin{equation} \label{aprio}
\int_{\mathbb{Z}^3}\sup\limits_{0 < t \leq T} \norm{ \widehat {\Phi^{m}f} (t,k)}_{L^2_v}d\Sigma(k)+ \int_{\mathbb{Z}^3}\left(\int_{0}^{T} \normm{ \widehat {\Phi^{m}f} (t,k)}^{2}dt\right)^{\frac12}d\Sigma(k) <+\infty,
\end{equation}
where, here and below,
$\Phi$ is either $P_1$ or $P_2$ as defined in \eqref{p12}. Moreover
\begin{equation} \label{shorttime}
 \forall\ m\geq 1,\  \forall\ k\in\mathbb Z^3,\quad   
    \lim_{t\rightarrow 0} \norm{ \widehat {\Phi^{m}f} (t,k)}_{L^2_v}=0.
\end{equation}
We refer to \cite{MR4356815} for the proof of \eqref{aprio} and \eqref{shorttime}. 
Since the present paper works with  smooth solutions, we will assume throughout the sequel that \eqref{aprio} and \eqref{shorttime} hold.  This ensures that all computations involving
$\Phi^m f$ remain rigorous in the subsequent discussion.

To simplify the notations, we will use the capital letter $C$ to denote a generic positive constant that may vary from line to line and depends only on the constants $C_0,C_1$ introduced in   Section \ref{sec:prelim}.  
Note that  these generic constants $C$ as below are independent of the derivative order denoted by $m$.    

\begin{proposition} \label{prp:ve}
	   Suppose the hypothesis of Theorem \ref{local-gxt} is fulfilled.      Then  there exist  two sufficiently small constants $0<\epsilon<\eps_0$ and a large constant $N\geq 4$, with $N$ depending only on  the numbers $C_0,C_1$ in  Section \ref{sec:prelim},  such that if the initial datum  $f_0$ in \eqref{3}  satisfies 
that
\begin{equation*}
	\norm{f_0}_{L_k^1 L_v^2}\leq \epsilon,
\end{equation*}
then for any $m\in\mathbb Z_+$ 
  the  following estimate holds:
 	\begin{multline}\label{pjm0}
 	\int_{\mathbb{Z}^3}\sup\limits_{   t\leq 1} \norm{  \widehat{\Phi^{m}f}(t,k)}_{L^2_v}d\Sigma(k) +  \int_{\mathbb{Z}^3}\left(\int_{0}^{1}  \normm{  \widehat{\Phi^{m}f}(t,k)}^{2}dt\right)^{1\over2}d\Sigma(k) \\
 	 + \int_{\mathbb{Z}^3}\comi k^{\frac{s}{1+2s}}\left(\int_{0}^{1}  \normm{    \widehat{\Phi^m f} (t,k)}^{2}dt\right)^{\frac12}d\Sigma(k)   \leq \frac{\eps_0 N^{m}  (m!)^\tau }{(m+1)^2},
 	\end{multline}
 	where $\tau= \max\{1,\frac{1}{2s}\}$ and 
$\Phi$ denotes either $P_1$ or $P_2$ as defined in \eqref{p12}. The above estimate remains true if we replace $\Phi$ with any  of the following operators:
\begin{equation*}
	 \frac{2s}{1+2s}t^{\frac{1+2s}{2s}} \partial_{x_i}+ t^{\frac{1}{2s}} \partial_{v_i}\ \textrm{ or }\    t^{\frac{1}{2s}} \partial_{v_i},  \quad  i=2,3.
\end{equation*}
 
\end{proposition} 

The rest part of this section is devoted to proving Proposition \ref{prp:ve}. We proceed by induction on 
$m$ to establish inequality \eqref{pjm0}. 
The case 
$m=0$
 follows directly from \eqref{v2} and \eqref{v33}.  Now
 assume 
 $m\geq 1$ and  that for all  $ j \leq m-1,$ the following estimate holds:
  \begin{multline}\label{tve}
\int_{\mathbb{Z}^3}\sup\limits_{  t\leq 1} \norm{\widehat{\Phi^j f} (t,k)}_{L^2_v}d\Sigma(k)  + \int_{\mathbb{Z}^3}\left(\int_{0}^{1}\normm{  \widehat{\Phi^{j}f}(t,k)}^{2}dt\right)^{\frac12}d\Sigma(k)  \\  
 + \int_{\mathbb{Z}^3}\comi k^{\frac{s}{1+2s}}\left(\int_{0}^{1}  \normm{    \widehat{\Phi^j f} (t,k)}^{2}dt\right)^{\frac12}d\Sigma(k) \leq \frac{\eps_0 N^{j}( j!)^\tau }{(j+1)^2}.
\end{multline} 
 It remains to  prove \eqref{tve}  for $j=m.$  To do so, applying $\Phi^m$ to equation \eqref{3} yields
 \begin{equation*}
    \big(\partial_t + v\cdot\partial_x - \mathcal{L}\big)\Phi^m f 
    = \big[\partial_t + v\cdot\partial_x, \Phi^m\big]f 
    + [\Phi^m, \mathcal{L}] f 
    + \Phi^m\Gamma(f, f).
\end{equation*}
Using Proposition \ref{sub:ell} and \eqref{shorttime}, we derive for $m\geq 1$ that  
\begin{equation}\label{phif}
    \begin{aligned}
    & \int_{\mathbb{Z}^3} \sup_{t\leq 1} \norm{\widehat{\Phi^m f}(t,k) }_{L^2_v} d\Sigma(k)
    + \int_{\mathbb{Z}^3} \left( \int_0^1 \normm{\widehat{\Phi^m f}(t,k)}^2 dt \right)^{\frac12} d\Sigma(k) \\
    &\quad + \int_{\mathbb{Z}^3} \comi{k}^{\frac{s}{1+2s}} 
        \left( \int_0^1 \norm{\widehat{\Phi^m f}(t,k)}_{L^2_v}^2 dt \right)^{\frac12} d\Sigma(k) \\
    &\leq C \int_{\mathbb{Z}^3} \left( \int_0^1 \big| \big( \widehat{g_m}, \widehat{\mathcal{M}\Phi^m f} \big)_{L_v^2} \big| dt \right)^{\frac12} d\Sigma(k),
    \end{aligned}
\end{equation}
where
\begin{equation}\label{defg}
    g_m := \big[\partial_t + v\cdot\partial_x, \Phi^m\big]f 
        + [\Phi^m, \mathcal{L}] f 
        + \Phi^m\Gamma(f, f).
\end{equation}
By Lemma \ref{lem:comestimate}, for any $\varepsilon > 0$,
\begin{equation*}
    \begin{aligned}
    & \int_{\mathbb{Z}^3} \left( \int_0^1 \left| \left( \mathscr F_{x}\big([\Phi^m, \mathcal{L}] f + \Phi^m\Gamma(f, f)\big), \widehat{\mathcal{M}\Phi^m f} \right)_{L^2_v} \right| dt \right)^{\frac12} d\Sigma(k) \\
    &\leq C\varepsilon^{-1} \eps_0\int_{\mathbb{Z}^3} \sup_{t \leq 1} \| \widehat{\Phi^m f}(t,k) \|_{L^2_v} d\Sigma(k) \\
    &\quad+ \big(\eps+C\varepsilon^{-1} \eps_0\big) \int_{\mathbb{Z}^3} \left( \int_0^1 \normm{\widehat{\Phi^m f}}^{2} dt \right)^{\frac12} d\Sigma(k)  +  C \eps^{-1}\big(  \eps_0+N^{-1}\big)\frac{\eps_0 N^{m} (m!)^\tau}{(m+1)^2}.
    \end{aligned}
\end{equation*}
Moreover we claim that,  for any $0<\varepsilon<1,$
\begin{equation}\label{sdf}
		\begin{aligned}
		 &\int_{\mathbb{Z}^3} \left( \int_0^1 \left| \left( \mathscr F_{x}\big([\partial_t + v\cdot\partial_x, \Phi^m] f\big), \widehat{\mathcal{M}\Phi^m f} \right)_{L^2_v} \right| dt \right)^{\frac12} d\Sigma(k)\\
		 &\leq 	\eps \sum_{j=1}^2\int_{\mathbb{Z}^3} \left( \int_0^1 \normm{\widehat{P_j^m f}(t,k)}^2 dt \right)^{\frac12} d\Sigma(k)\\
         &\quad+\varepsilon \sum_{j=1}^2 \int_{\mathbb{Z}^3} |k|^{\frac{s}{1+2s}} 
        \left( \int_0^1 \| \widehat{P_j^m f} \|_{L^2_v}^2 dt \right)^{\frac12} d\Sigma(k)+C  \varepsilon^{-\frac{1+s}{s}}   N^{-1} \frac{\varepsilon_0 N^{m} (m!)^\tau}{(m+1)^2}.
		\end{aligned}
	\end{equation}
Assuming \eqref{sdf} holds, we substitute  the two estimates above into \eqref{phif}, and then choose sufficiently small $\varepsilon$ and use the smallness of $\eps_0;$   this gives 
\begin{equation*}
	\begin{aligned}
		&\sum_{j=1}^2\int_{\mathbb{Z}^3}  \sup_{t\leq 1} \norm{\widehat{P_j^m f}(t,k)}_{L_v^2}  d\Sigma(k)+\sum_{j=1}^2\int_{\mathbb{Z}^3} \left( \int_0^1 \normm{\widehat{P_j^m f}(t,k)}^2 dt \right)^{\frac12} d\Sigma(k)\\
		 &\  \  \ +\sum_{j=1}^2 \int_{\mathbb{Z}^3} \comi{k}^{\frac{s}{1+2s}} 
        \left( \int_0^1 \norm{\widehat{P_j^m f}(t,k)}_{L^2_v}^2 dt \right)^{\frac12} d\Sigma(k) \leq   C  \big(\eps_0+N^{-1}\big)\frac{\eps_0 N^{m} (m!)^\tau}{(m+1)^2}.
	\end{aligned}
\end{equation*}
 Therefore, \eqref{tve} holds for $j=m$ provided that \begin{equation*}
    C  \big(\eps_0+N^{-1}\big)\leq 1. 
 \end{equation*}
 Then the desired estimate \eqref{pjm0} follows.  The proof of Proposition \ref{prp:ve} is completed  once assertion \eqref{sdf} is verified. 
In the following, we will prove  \eqref{sdf} via the three lemmas below.

\begin{lemma}[$m=1$]\label{lemm1}
	Under the inductive assumption \eqref{tve}, we have, for any $0<\eps<1,$
	\begin{equation*}
		 \begin{aligned}
    &  \int_{\mathbb{Z}^3} \left( \int_0^1 \left| \left(\mathscr F_{x}\big([\partial_t + v\cdot\partial_x, \Phi] f\big), \widehat{\mathcal{M}\Phi f} \right)_{L^2_v} \right| dt \right)^{\frac12} d\Sigma(k) \\
    &\leq \varepsilon  \int_{\mathbb{Z}^3} \left( \int_0^1 \normm{ \widehat{P_2f}(t,k)}^2 dt \right)^{\frac12} d\Sigma(k) \\
    &\qquad + \varepsilon \sum_{j=1}^2 \int_{\mathbb{Z}^3} |k|^{\frac{s}{1+2s}} 
        \left( \int_0^1 \| \widehat{P_j f} (t,k)\|_{L^2_v}^2 dt \right)^{\frac12} d\Sigma(k)  + C\varepsilon^{-\frac{1+s}{s}}\eps_0,
    \end{aligned}
	\end{equation*}
    where the Fourier multiplier $\mathcal{M}$ satisfies \eqref{psk}.
\end{lemma}

\begin{proof}
	We first derive the estimate
\begin{equation}
    \label{vvv1}
   \big|\big(\mathscr F_{x}\big([\partial_t+v\cdot\partial_x, \Phi] f\big), \widehat{\Phi f}\big)_{L^2_v}\big|
 \leq   C  t^{\frac{1+s}{s}}   \|  k_1 \hat{  f} \|_{L^2_v}^2 +  C  t^{\frac{1 }{s}-1} \|\partial_{v_1} \hat{ f}\|^2_{L^2_v},
\end{equation}
recalling that $\Phi$ is either $P_1$ or $P_2$ as defined in \eqref{p12}. 
First, let $\Phi = P_1$. From \eqref{pmc}, we have
\begin{equation*}
[\partial_t+v\cdot\partial_x, P_1]=\frac{1}{2s}t^{\frac{1-2s}{2s}}\partial_{v_1}.
\end{equation*}
Then
\begin{equation*}  
\begin{aligned}
  \big|\big(\mathscr F_{x}\big([\partial_t+v\cdot\partial_x, \Phi] f\big), \widehat{\Phi f}\big)_{L^2_v}\big|
  &=\frac{1}{2s} t^{\frac{1-2s}{2s}} \big|\big( \partial_{v_1} \hat{f}, \widehat{P_1f}\big)_{L^2_v}\big|\\
 &\leq \frac{1}{2s}  t^{-1} \|\widehat{P_1f}\|^2_{L^2_v} +   \frac{1}{2s} t^{\frac{1 }{s}-1} \|\partial_{v_1} \hat{ f}\|^2_{L^2_v}.
\end{aligned}
\end{equation*}
From the definition of $P_1$ in \eqref{p12}, we obtain
\begin{equation*}\label{Pff1}
\begin{aligned}
  t^{-1} \| \widehat{P_1 f} \|_{L^2_v}^2  
\leq  C t^{\frac{1+s}{s}}\|k_1 \hat{  f} \|_{L^2_v}^2+ C t^{\frac{1}{s}-1} \|   \partial_{v_1} \hat{  f} \|_{L^2_v}^2 .
\end{aligned}
\end{equation*}
Combining these estimates  yields  \eqref{vvv1} for $\Phi = P_1$.

Now let $\Phi = P_2 = t^{\frac{1}{2s}} \partial_{v_1}$. Then
\begin{equation*}
    [\partial_t+v\cdot\partial_x, \Phi]=\frac{1}{2s}t^{\frac{1}{2s}-1}\partial_{v_1}-t^{\frac{1}{2s}}\partial_{x_1},
\end{equation*}
and thus
\begin{equation*}  
\begin{aligned}
  \big|\big(\mathscr F_{x}\big([\partial_t+v\cdot\partial_x, \Phi] f\big), \widehat{\Phi f}\big)_{L^2_v}\big|\leq C  t^{\frac{1 }{s}-1} \|\partial_{v_1} \hat{ f}\|^2_{L^2_v}+ C  t^{\frac{1+s}{s}}   \|  k_1 \hat{  f} \|_{L^2_v}^2.
\end{aligned}
\end{equation*}
So \eqref{vvv1} holds for $\Phi=P_2.$
 
For the last term in \eqref{vvv1}, we use the interpolation inequality:  
\begin{equation*}\label{lowinter123}
    \forall\  \tilde{\varepsilon}>0,\quad \|g\|^2_{L_v^{2}} \leq \tilde{\varepsilon} \|g\|^{2}_{H^{s}_v} + \tilde{\varepsilon}^{-\frac{1-s}{s}}\|g\|^{2}_{H^{s-1}_v}.
\end{equation*}
Letting $\tilde \eps= \eps t$ and using the fact that $\|\partial_{v_1}\hat{f}\|_{H^{s}_v} \leq \normm{\partial_{v_1}\hat{f}}$ due to \eqref{+lowoftri}, we obtain
\begin{equation}\label{esti-v}
    \begin{aligned}
          t^{{\frac{1}{s} }-1}\|\partial_{v_1}\hat{f}\|_{L_v^2}^2 \leq \eps t^{{\frac{1}{s} }}\|\partial_{v_1}\hat{f}\|_{H_v^s}^2+\eps^{-\frac{1-s}{s}}  \|\partial_{v_1}\hat{f}\|_{H_v^{s-1}}^2 
        &\leq \varepsilon \normm{ \widehat{P_2f}}^2 + \varepsilon^{-\frac{1-s}{s}}  \normm{ \hat{f}}^2,
    \end{aligned}
\end{equation}
where  the last inequality uses \eqref{+lowoftri} and the fact that $P_2 = t^{\frac{1}{2s}} \partial_{v_1}$.
 
For the first term on the right-hand side of \eqref{vvv1}, we apply the inequality
\begin{align*}
   \| |k_1|^{\frac{1+s}{1+2s}} \hat{  f} \|_{L^2_v}^2 \leq \epsilon   \| k_1 \hat{  f} \|_{L^2_v}^2 +\epsilon^{-\frac{1+s}{s}} \| \hat{ f} \|_{L^2_v}^2 
\end{align*}
with $\epsilon = \eps t.$ Then for any $\eps > 0$,
\begin{equation}\label{ksub}
\begin{aligned}
& t^{\frac{1+s}{s}}\|  k_1 \hat{  f} \|_{L^2_v}^2=  t^{\frac{1+s}{s}}
 |k_1|^{\frac{2s}{1+2s}} \|\,  |k_1|^{\frac{1+s}{1+2s}} \hat{  f} \|_{L^2_v}^2  \\
&\leq   |k_1|^{\frac{2s}{1+2s}} \Big(\eps  t^{\frac{1+2s}{s}} \|  k_1 \hat{ f} \|_{L^2_v}^2 + \eps^{-\frac{1+s}{s}}  \| \hat{f} \|_{L^2_v}^2\Big) \\
&\leq |k_1|^{\frac{2s}{1+2s}}\Big( \eps    \| t^{\frac{1+2s}{2s}}   \widehat{ \partial_{x_1}f} \|_{L^2_v}^2 + \eps^{-\frac{1+s}{s}} \| \hat{f} \|_{L^2_v}^2\Big)\\
&\leq \frac{1+2s}{s} \eps |k|^{\frac{2s}{1+2s}} \|    \widehat{  P_1 f} \|_{L^2_v}^2+ \frac{1+2s}{s} \eps|k|^{\frac{2s}{1+2s}} \| \widehat{ P_2f} \|_{L^2_v}^2 + \eps^{-\frac{1+s}{s}}|k|^{\frac{2s}{1+2s}} \| \hat{f} \|_{L^2_v}^2,
\end{aligned}
\end{equation}
where the last inequality follows from the identity $\frac{2s}{1+2s}t^{\frac{1+2s}{2s}} \partial_{x_1} = P_1 - P_2$   in view of \eqref{p12}. Combining this with \eqref{esti-v} and \eqref{vvv1} yields, for any $0<\eps<1$,
\begin{equation*} \label{esti-vv2}
\begin{aligned}
  &  \big|\big(\mathscr F_{x}\big([\partial_t+v\cdot\partial_x, \Phi] f\big), \widehat{\Phi f}\big)_{L^2_v}\big|\\
&\leq    \varepsilon   \normm{ \widehat{P_2f}}^2 +\eps\sum_{j=1}^2  |k|^{\frac{2s}{1+2s}} \|    \widehat{  P_j f} \|_{L^2_v}^2+ C  \varepsilon^{-\frac{1-s}{s}}  \normm{ \hat{f}}^2   + C\eps^{-\frac{1+s}{s}}|k|^{\frac{2s}{1+2s}} \| \hat{f} \|_{L^2_v}^2\\
&\leq    \varepsilon   \normm{ \widehat{P_2f}}^2 +\eps\sum_{j=1}^2  |k|^{\frac{2s}{1+2s}} \|    \widehat{  P_j f} \|_{L^2_v}^2+ C \eps^{-\frac{1+s}{s}}  \normm{ \hat{f}}^2   + C\eps^{-\frac{1+s}{s}}|k|^{\frac{2s}{1+2s}} \| \hat{f} \|_{L^2_v}^2.
\end{aligned}
\end{equation*}
 This with the inductive assumption \eqref{tve} implies the  assertion of Lemma \ref{lemm1}, completing the proof. 
\end{proof}

\begin{lemma}\label{lemp2}
	Let $m\geq 2$ and suppose the inductive assumption \eqref{tve} holds for $j\leq m-1$.  Then, for any $0<\eps<1,$  
	\begin{equation*}
		\begin{aligned}
		 &\int_{\mathbb{Z}^3} \left( \int_0^1 \left| \left( \mathscr F_{x}\big([\partial_t + v\cdot\partial_x, P_2^m] f\big), \widehat{\mathcal{M}P_2^m f} \right)_{L^2_v} \right| dt \right)^{\frac12} d\Sigma(k)\\
		 &\leq 	\eps \sum_{j=1}^2\int_{\mathbb{Z}^3} \left( \int_0^1 \normm{\widehat{P_j^m f}}^2 dt \right)^{\frac12} d\Sigma(k)+C  \varepsilon^{-\frac{1+s}{s}}   N^{-1} \frac{\varepsilon_0 N^{m} (m!)^\tau}{(m+1)^2},
		\end{aligned}
	\end{equation*}
    where the Fourier multiplier $\mathcal{M}$ satisfies \eqref{psk}.
\end{lemma}

\begin{proof}
Recalling the definition of $P_1$ and $P_2$ in \eqref{p12},  direct verification yields 
	\begin{multline*}
[\partial_t + v\cdot\partial_x, P_2^m] = \frac{m}{2s} t^{\frac{m}{2s}-1} \partial_{v_1}^m - mt^{\frac{m}{2s}}  \partial_{x_1}\partial_{v_1}^{m-1} \\
=\frac{m}{2s} t^{-1} P_2^m-\frac{1+2s}{2s} mt^{-1}\big(P_1- P_2 \big)P_2^{m-1},
\end{multline*}
where the last identity uses \eqref{p12}. 
Hence, 
\begin{equation}\label{08oct}
\begin{aligned}
 \big| \big(\mathscr F_{x}\big([\partial_t + v\cdot\partial_x, P_2^m] f\big), \widehat{P_2^m f} \big)_{L^2_v} \big|&\leq  C mt^{-1} \|  \widehat{P_2^mf} \|_{L^2_v}^2+Cm t^{-1}\|   \widehat{P_1P_2^{m-1}f} \|_{L^2_v}^2.
\end{aligned}
\end{equation}
To treat  the first term on the right-hand side of \eqref{08oct}, we use the interpolation inequality:
\begin{equation}\label{interp}
    \forall\  \tilde{\varepsilon}>0,\quad \|g\|^2_{L_v^{2}} \leq \tilde{\varepsilon} \|g\|^{2}_{H^{s}_v} + \tilde{\varepsilon}^{-\frac{1-s}{s}}\|g\|^{2}_{H^{s-1}_v},
\end{equation}
with $\tilde{\varepsilon} =\varepsilon m^{-1} t$ and $g=\partial_{v_1}^{m}\hat{f}$. This with the definition  that $P_2=t^{\frac{1}{2s}}\partial_{v_1} $ gives, for any $\varepsilon>0,$ 
\begin{equation*}\label{eqpre}
    \begin{aligned}
        m t^{-1}\| \widehat{P_2^mf}\|_{L_v^2}^2 
        &\leq \varepsilon  \|\widehat{P_2^mf}\|_{H_v^{s}}^2 + \varepsilon^{-\frac{1-s}{s}}  m^{\frac{1}{s}} t^{-\frac{1}{s}} \|\widehat{P_2^mf}\|_{H_v^{s-1}}^2\\
        &\leq \varepsilon  \| \widehat{P_2^{m}f}\|_{H_v^{s}}^2 + \varepsilon^{-\frac{1-s}{s}}m^{\frac{1}{s}}  \|\widehat{P_2^{m-1} f}\|_{H_v^{s}}^2\\
        &\leq \varepsilon  \normm{\widehat{P_2^{m}f}}^2 + \varepsilon^{-\frac{1-s}{s}}m^{\frac{1}{s}}   \normm{\widehat{P_2^{m-1}f}}^2,
    \end{aligned}
\end{equation*}
where the last inequality follows from \eqref{+lowoftri}.
Following a similar argument  as above, we deduce the estimate for the second term on the right-hand side of \eqref{08oct} as follows.
\begin{equation}\label{fouri}
    \begin{aligned}
         m t^{ -1}\|\widehat{P_1P_2^{m-1}f}\|^2_{L_v^2} &\leq  \varepsilon  \| \widehat{P_1P_2^{m-1}f}\|^2_{H_v^s} + \varepsilon^{-\frac{1-s}{s}}m^{\frac{1}{s}} t^{-\frac{1}{s}} \| \widehat{P_1P_2^{m-1}f}\|^2_{H_v^{s-1}}\\
        &\leq  \varepsilon  \| \widehat{P_1P_2^{m-1}f}\|^2_{H_v^s} + \varepsilon^{-\frac{1-s}{s}}m^{\frac{1}{s}}  \| \widehat{P_1P_2^{m-2}f}\|^2_{H_v^{s}}\\
        & \leq  \sum_{j=1}^2\Big( \varepsilon  \| \widehat{P_j^{m}f}\|^2_{H_v^s} + \varepsilon^{-\frac{1-s}{s}}m^{\frac{1}{s}}  \| \widehat{P_j^{m-1}f}\|^2_{H_v^{s}}\Big)\\ 
        & \leq  \sum_{j=1}^2\Big( \varepsilon  \normm{ \widehat{P_j^{m}f}}^2 + \varepsilon^{-\frac{1-s}{s}}m^{\frac{1}{s}} \normm{\widehat{P_j^{m-1}f}}^2\Big),
    \end{aligned}
\end{equation}
where the third inequality uses \eqref{fmn}. 
Substituting  these estimates into \eqref{08oct}, we obtain for any $0<\eps<1,$ 
\begin{equation*} 
    \begin{aligned}
    \big| \big( \mathscr F_{x}\big([\partial_t + v\cdot\partial_x, P_2^m] f\big), \widehat{P_2^m f} \big)_{L^2_v} \big|&\leq  \sum_{j=1}^2\Big( \varepsilon  \normm{ \widehat{P_j^{m}f}}^2 +C \varepsilon^{-\frac{1-s}{s}}m^{\frac{1}{s}} \normm{\widehat{P_j^{m-1}f}}^2\Big)\\
    &\leq  \sum_{j=1}^2\Big( \varepsilon  \normm{ \widehat{P_j^{m}f}}^2 +C \varepsilon^{-\frac{1+s}{s}}m^{\frac{1}{s}} \normm{\widehat{P_j^{m-1}f}}^2\Big),
    \end{aligned}
\end{equation*}
which with the inductive assumption \eqref{tve} yields the assertion of Lemma \ref{lemp2}. The proof is completed. 
\end{proof}

\begin{lemma}\label{lem:m2}
	Let $m\geq 2$ and suppose the  inductive assumption \eqref{tve} holds for $j\leq m-1$.  Then, for any $0<\eps<1,$
	\begin{equation*}
		\begin{aligned}
		 &\int_{\mathbb{Z}^3} \left( \int_0^1 \left| \left( \mathscr F_{x}\big([\partial_t + v\cdot\partial_x, P_1^m] f\big), \widehat{\mathcal{M}P_1^m f} \right)_{L^2_v} \right| dt \right)^{\frac12} d\Sigma(k)\\
		 &\leq 	\eps \sum_{j=1}^2\int_{\mathbb{Z}^3} \left( \int_0^1 \normm{\widehat{P_j^m}f}^2 dt \right)^{\frac12} d\Sigma(k)+C  \varepsilon^{-\frac{1+s}{s}}   N^{-1} \frac{\varepsilon_0 N^{m} (m!)^\tau}{(m+1)^2},
		\end{aligned}
	\end{equation*}
    where the Fourier multiplier $\mathcal{M}$ satisfies \eqref{psk}.
\end{lemma}

\begin{proof}In view of 
\eqref{pmc} as well as \eqref{p12},
\begin{equation*}
	 [\partial_t + v\cdot\partial_x, P_1^m] f =\frac{1}{2s} m t^{\frac{1-2s}{2s}} \partial_{v_1} P_1^{m-1}=\frac{1}{2s} m t^{-1} P_2 P_1^{m-1}.
\end{equation*}
Thus, for any $\eps>0,$
\begin{equation*}
	 \big| \big( \mathscr F_{x}\big([\partial_t + v\cdot\partial_x, P_1^m] f\big), \widehat{P_1^m f} \big)_{L^2_v} \big|\leq \varepsilon \|\widehat{P_1^{m}f}\|^2_{H^{s}_v} + C\varepsilon^{-1}m^2 t^{-2} \|\widehat{P_2P_1^{m-1}f}\|^2_{H^{-s}_v}.
\end{equation*}
The assertion of Lemma \ref{lem:m2} will follow once we establish the following estimate\begin{equation}\label{p2p1m1}
	\begin{aligned}
		& \int_{\mathbb{Z}^3} \left( \int_{0}^{1} \varepsilon^{-1}m^2 t^{-2} \|\widehat{P_2P_1^{m-1}f}\|^2_{H^{-s}_v}dt \right)^{\frac12} d\Sigma(k) \\
		&\leq \eps \sum_{j=1}^2 \int_{\mathbb{Z}^3} \left( \int_{0}^{1}   \normm{\widehat{P_j^mf}}dt \right)^{\frac12} d\Sigma(k)+ C \varepsilon^{-\frac{1+s}{s}}   N^{-1} \frac{\varepsilon_0 N^{m} (m!)^\tau}{(m+1)^2}.
	\end{aligned}
\end{equation}
Next we  proceed to prove \eqref{p2p1m1} by considering  separately the cases $0 < s \leq \frac12$ and $\frac12 < s < 1$.

{\it (i) Case $0 < s \leq \frac12$.}
We apply the interpolation inequality:
\begin{equation} \label{s}
\forall\  \tilde{\varepsilon} > 0,\quad \|g\|^2_{H_v^{-s}} \leq \tilde{\varepsilon} \|g\|^{2}_{H^{s}_v} + \tilde{\varepsilon}^{-\frac{1-2s}{2s}}\|g\|^{2}_{H^{s-1}_v},
\end{equation}
with $\tilde{\varepsilon} = \varepsilon^2 m^{-2} t^{2}$ and $g =  \widehat{P_2P_1^{m-1}f}$. This yields, for  any $0<\eps<1,$
\begin{equation} \label{p2}
\begin{aligned}
\varepsilon^{-1}m^2 t^{-2} \|\widehat{P_2P_1^{m-1}f}\|^2_{H^{-s}_v} 
&\leq \varepsilon   \|  \widehat{P_2P_1^{m-1}f}\|_{H_v^{s}}^2 + \varepsilon^{-\frac{1-s}{s}}m^{\frac{1}{s}} \|\partial_{v_1} \widehat{P^{m-1}f}\|_{H_v^{s-1}}^2\\
&\leq \varepsilon \sum_{j=1}^2 \normm{\widehat{P_j^{m}f}}^2 + \varepsilon^{-\frac{1+s}{s}}m^{\frac{1}{s}} \normm{\widehat{P^{m-1}f}}^2,
\end{aligned}
\end{equation}
where the last inequality follows from an argument similar to that in \eqref{fouri}.  Since $\tau=\frac{1}{2s}$ for $0<s\leq \frac12,$  combining \eqref{p2} and the inductive assumption \eqref{tve} yields that  \eqref{p2p1m1} holds in this case.

{\it (ii) Case $\frac12<s<1.$}  We use the interpolation inequality: 
\begin{equation*}
\begin{aligned}
\forall\ \tilde\eps > 0,\quad \| g \|_{H^{-s}_v}^2 \leq \tilde\eps \norm{g}_{H^{s-1}_v}^2+\tilde\eps^{-\frac{1-s}{2s-1}} \| g\|_{H^{-1}_v}^2,
\end{aligned}
\end{equation*}
with $\tilde{\varepsilon} =t^{2-\frac{1}{s}}$ and $g =  \widehat{P_2P_1^{m-1}f}$.  Then 
\begin{multline}\label{slarge}
\varepsilon^{-1}m^2 t^{-2} \|\widehat{P_2P_1^{m-1}f}\|^2_{H^{-s}_v} \\
\leq \eps^{-1} m^2\Big(t^{-\frac{1}{s}} \| \widehat{P_2P_1^{m-1}f} \|_{H^{s-1}_v}^2 +  t^{-\frac{1}{s}-1} \| \widehat{P_2P_1^{m-1}f}\|_{H^{-1}_v}^2\Big)\\
\leq \eps^{-1} m^2 \| \widehat{P_1^{m-1}f} \|_{H^{s}_v}^2 +\eps^{-1} m^2  t^{-1} \| \widehat{P_1^{m-1}f}\|_{L_v^2}^2,
\end{multline}
where the last inequality uses the definition that $P_2=t^{\frac{1}{2s}}\partial_{v_1}$.  Moreover,   recalling the definition of $P_1$ in \eqref{p12}, we obtain 
\begin{equation}\label{Pff}
\begin{aligned}
m^2 t^{-1} \| \widehat{P_1^{m-1} f} \|_{L^2_v}^2 = & m^2 t^{-1} \big \| \mathscr{F}_x \Big ( \big ( \frac{2s}{1+2s} t^{\frac{1+2s}{2s}} \partial_{x_1} + t^{\frac{1}{2s}} \partial_{v_1} \big) P_1^{m-2} f \Big ) \big\|_{L^2_v}^2 \\
\leq & 2 m^2 t^{-1} \|   \widehat{P_2 P_1^{m-2} f} \|_{L^2_v}^2 + 2 m^2 t^{\frac{1+s}{s}}  \|k_1 \widehat{P_1^{m-2} f} \|_{L^2_v}^2.
\end{aligned}
\end{equation}
For the first term in \eqref{Pff}, we apply the interpolation inequality
\begin{equation*}
    \forall\  \tilde{\varepsilon}>0,\quad \|g\|^2_{L_v^{2}} \leq \tilde{\varepsilon} \|g\|^{2}_{H^{s}_v} + \tilde{\varepsilon}^{-\frac{1-s}{s}}\|g\|^{2}_{H^{s-1}_v},
\end{equation*}
and set $\tilde{\varepsilon} = t$ to obtain
\begin{equation*}
\begin{aligned}
 &m^2 t^{-1} \|   \widehat{P_2P_1^{m-2} f} \|_{L^2_v}^2  \leq  m^2  \|   \widehat{P_2P_1^{m-2} f} \|_{H^s_v}^2+  m^2 t^{-\frac{1}{s}} \|   \widehat{P_2P_1^{m-2} f} \|_{H^{s-1}_v}^2\\
&\leq   m^2 \|    \widehat{P_2P_1^{m-2} f} \|_{H^s_v}^2 + m^2 \| \widehat{P^{m-2} f} \|_{H^s_v}^2 \leq   \sum_{j=1}^2m^2\Big( \normm{ \widehat{P_j^{m-1} f}}^2 + \normm{ \widehat{P_j^{m-2} f}}^2\Big),
\end{aligned}
\end{equation*}
where the last inequality follows from \eqref{fmn}.

For the second term in \eqref{Pff}, we repeat the computation in \eqref{ksub} with $\hat f$ replaced by $\widehat{P_1^{m-2} f}$, yielding
\begin{align*}
   m^2 t^{\frac{1+s}{s}}  \|k_1 \widehat{P_1^{m-2} f} \|_{L^2_v}^2 & \leq  m^2 \Big(|k_1|^{\frac{2s}{1+2s}} \sum_{j=1}^2\| \widehat{P_j P_1^{m-2} f} \|_{L^2_v}^2 +    |k_1|^{\frac{2s}{1+2s}} \| \widehat{P_1^{m-2} f} \|_{L^2_v}^2\Big)\\
& \leq  m^2 |k_1|^{\frac{2s}{1+2s}} \sum_{j=1}^2\| \widehat{P_j^{m-1} f} \|_{L^2_v}^2 +  m^2 |k_1|^{\frac{2s}{1+2s}} \| \widehat{P_1^{m-2} f} \|_{L^2_v}^2,
\end{align*}
where the last inequality uses \eqref{fmn}. Combining these estimates with \eqref{Pff}, we deduce
\begin{multline*}
	m^2 t^{-1} \| \widehat{P_1^{m-1} f} \|_{L^2_v}^2\leq C\sum_{j=1}^2m^2\Big( \normm{ \widehat{P_j^{m-1} f}}^2 + \normm{ \widehat{P_j^{m-2} f}}^2\Big)\\
	+C  \sum_{j=1}^2 m^2 |k_1|^{\frac{2s}{1+2s}} \Big( \| \widehat{P_j^{m-1} f} \|_{L^2_v}^2 +  \| \widehat{P_j^{m-2} f} \|_{L^2_v}^2\Big).
\end{multline*}
Substituting this into \eqref{slarge} and applying \eqref{+lowoftri}, we conclude that for $\frac12 < s < 1$,
\begin{multline*}
	\varepsilon^{-1}m^2 t^{-2} \|\widehat{P_2P_1^{m-1}f}\|^2_{H^{-s}_v} \leq C\eps^{-1}\sum_{j=1}^2m^2\Big( \normm{ \widehat{P_j^{m-1} f}}^2 + \normm{ \widehat{P_j^{m-2} f}}^2\Big)\\
	+C\eps^{-1}  \sum_{j=1}^2 m^2 |k_1|^{\frac{2s}{1+2s}} \Big( \| \widehat{P_j^{m-1} f} \|_{L^2_v}^2 +  \| \widehat{P_j^{m-2} f} \|_{L^2_v}^2\Big).
\end{multline*}
Combining this with the inductive assumption \eqref{tve}, and noting that $\tau = 1$ in this case, we conclude that \eqref{p2p1m1} holds for $\frac12 < s < 1$.  This completes the proof of \eqref{p2p1m1}  and thus the proof of Lemma \ref{lem:m2}.
\end{proof}

Combining these estimates  in Lemmas \ref{lemm1}-\ref{lem:m2}, we conclude \eqref{sdf} holds. This completes the proof of Proposition \ref{prp:ve}.

\begin{proof}[ Completing the proof of Theorem \ref{local-gxt}]
This is a direct consequence of Proposition \ref{prp:ve}. Indeed, applying inequality \eqref{+pse1}, we obtain  \begin{align*}
  \Big(\frac{2s}{1+2s}\Big)^m\int_{\mathbb Z^3}	\sup_{ t\leq 1}t^{\frac{1+2s}{2s}m} \norm{\widehat {\partial_{x_1}^{m}f}}_{L^2_{v}}d\Sigma(k)  & \leq 2^{m} \sum_{j=1}^2\int_{\mathbb Z^3}	\sup_{ t\leq 1} \norm{\widehat {P_j^{m}f} }_{L^2_{v}}d\Sigma(k)  \\  
 	&\leq 2^{m+1} \frac{\eps_0 N^{m}  (m!)^\tau  }{(m+1)^2},
 	 	\end{align*}
 	where	the last inequality follows from Proposition \ref{prp:ve}.
 Similarly, the above estimate  remains true when $\partial_{x_1}$  is replaced by $\partial_{x_2}$ or $\partial_{x_3}$. Combining this with the inequality
\begin{eqnarray*}
\forall\ \alpha\in\mathbb Z_+^3,\quad 	\norm{\widehat{\partial_x^\alpha f}(t,k)}_{L^2_{v}}\leq \sum_{1\leq j\leq 3}\norm{\mathscr F_x \big(\partial_{x_j}^{\abs\alpha}f\big)(t,k)}_{L^2_{v}},
\end{eqnarray*}
we obtain 
\begin{eqnarray*}
	\forall\ \alpha\in\mathbb Z_+^3,\quad 	\Big(\frac{2s}{1+2s}\Big)^{\abs\alpha}\int_{\mathbb Z^3}	\sup_{ t\leq 1}t^{\frac{1+2s}{2s}\abs\alpha} \norm{\widehat {\partial_{x}^{\alpha}f}(t,k)}_{L^2_{v}}d\Sigma(k) \leq 6^{\abs\alpha+1} \frac{\eps_0 N^{\abs\alpha}  (\abs\alpha!)^\tau  }{(\abs\alpha+1)^2}.
\end{eqnarray*}
Therefore,
\begin{eqnarray*}
	\forall\ \alpha\in\mathbb Z_+^3,\quad 	\int_{\mathbb Z^3}	\sup_{ t\leq 1}t^{\frac{1+2s}{2s}\abs\alpha} \norm{\widehat {\partial_{x}^{\alpha}f}(t,k)}_{L^2_{v}}d\Sigma(k) \leq C_s^{\abs\alpha+1} \frac{\eps_0 N^{\abs\alpha}  (\abs\alpha!)^\tau  }{(\abs\alpha+1)^2},
\end{eqnarray*}
where $C_s=\frac{3(1+2s)}{s}$. 
Similarly, we have
\begin{eqnarray*}
\forall\ \beta\in\mathbb Z_+^3, \quad 	\int_{\mathbb Z^3}	\sup_{ t\leq 1}t^{\frac{1}{2s}\abs\beta} \norm{\partial_{v}^{\beta} \hat {f}(t,k)}_{L^2_{v}}d\Sigma(k) \leq C_s^{\abs\beta+1} \frac{\eps_0 N^{\abs\beta}  (\abs\beta!)^\tau  }{(\abs\beta+1)^2}.
 	\end{eqnarray*}
 	Consequently,  for any $\alpha,\beta\in\mathbb Z_+^3,$
 	\begin{equation*} 
 		\begin{aligned}
 		&	\int_{\mathbb Z^3}	\sup_{ t\leq 1}t^{\frac{1+2s}{2s}\abs\alpha+\frac{1}{2s}\abs\beta} \norm{  \mathscr F_x \big(\partial_{x}^{\alpha}\partial_v^\beta f\big)(t,k)}_{L^2_{v}}d\Sigma(k) \\
 		&\leq \int_{\mathbb Z^3}	\sup_{ t\leq 1}t^{\frac{1+2s}{2s}\abs\alpha+\frac{1}{2s}\abs\beta} \norm{  \widehat{\partial_{x}^{2\alpha}  f}(t,k)}_{L^2_{v}}^{\frac12} \norm{\partial_{v}^{2\beta} \hat {f}(t,k)}_{L^2_{v}}^{\frac12} d\Sigma(k)\\
 		&\leq \bigg(\int_{\mathbb Z^3}	\sup_{ t\leq 1}t^{\frac{1+2s}{2s}2\abs\alpha} \norm{  \widehat{\partial_{x}^{2\alpha}  f}}_{L^2_{v}} d\Sigma(k)\bigg)^{\frac12} \bigg(\int_{\mathbb Z^3}	\sup_{ t\leq 1}t^{ \frac{1}{2s}2\abs\beta}   \norm{\partial_{v}^{2\beta} \hat {f}}_{L^2_{v}}  d\Sigma(k)\bigg)^{\frac12} \\
 			&\leq \eps_0  (C_s N)^{ \abs\alpha+\abs\beta+1 }\Big(\abs{2\alpha}! \abs{2\beta}!\Big)^{\frac{\tau}{2}}\leq \eps_0  (2^\tau C_s N)^{ \abs\alpha+\abs\beta+1 }\inner{ \abs\alpha+\abs\beta }!^\tau,
 		\end{aligned}
 	\end{equation*}
where the last inequality uses the fact that $p!q!\leq(p+q)!\leq 2^{p+q}p!q!$ for any $p, q \in \mathbb{Z}$. A similar argument yields the estimate for 
\begin{equation*}
	 \int_{\mathbb Z^3}  \Big(\int_0^1 t^{\frac{1+2s}{s}|\alpha|+\frac{1}{s}\abs\beta} \normm{\mathscr F_{x}\big(\partial^{\alpha}_x\partial_v^\beta f\big)}^2dt \Big)^{\frac12} d\Sigma(k).
\end{equation*}
Therefore, the desired estimate \eqref{gx++} follows by choosing 
 $C$ large enough such that $C\geq 2^\tau C_s N $. This completes the proof of Theorem \ref{local-gxt}.
\end{proof}  
  
\section{Macroscopic estimates}\label{sec:macros}
This part, together with Section \ref{sec:global}, is devoted  to proving Theorem \ref{gxt}. As a preliminary step, we derive macroscopic estimates in analytic or Gevrey spaces. Inspired by the vector fields introduced in \cite{MR4930523} to handle the lack of spatial diffusion, we work with the following combination of $\partial_x$ and $\partial_v$ with a time-dependent coefficient
   \begin{equation}
   	\label{vecM}
   H= t  \partial_{x_1}+   \partial_{v_1}.
   \end{equation}
 Using $H$ instead of only $\partial_x$ or $\partial_v$ has the advantage that $H$ commutes with the transport operator, i.e.,
  \begin{equation*}\label{keyob}
  	[H, \,\,  \partial_t+v\cdot\partial_x ]=0,
  \end{equation*}
  recalling $[\cdot,   \cdot]$ stands for the commutators between two operators. 
  More generally, by induction on $m$, we obtain
  \begin{equation}\label{kehigher}
\forall\ m\geq 1,\quad 	[H^m, \,\,  \partial_t+v \cdot \partial_x ]=0.
\end{equation}
As will be seen below, this allows us to obtain quantitative estimates for the directional derivatives $H^m f$ of the solution $f$.
 
Let $f$ be a solution of the Boltzmann equation \eqref{3}, and let $\mathbf{P}$ be defined as in \eqref{eqorp1}. For each $m \in \mathbb{Z}+$, we write
 \begin{equation}\label{project}
 	\mathbf{P}H^mf=\big \{a_m(t,x)+b_m(t,x)\cdot v+c_m(t,x)(|v|^2-3)\big \}\mu^{\frac12} 
 \end{equation}
 and
  \begin{eqnarray*}
 	\mathbf{P}\partial_{v_1}^mf=\Big \{U_m(t,x)+V_m(t,x)\cdot v+W_m(t,x)\big (|v|^2-3\big )\Big \}\mu^{\frac12},
 \end{eqnarray*}
where
 \begin{equation} \label{mac}
 	\left\{
 	\begin{aligned}
 		 & a_m=\int_{\mathbb R^3}\mu^{\frac12}H^mf dv,  \\
 		 &b_m=(b_{m,1}, b_{m,2}, b_{m,3})=\int_{\mathbb R^3}v\mu^{\frac12}H^mf dv,\\
 		 &c_m=\frac{1}{6}\int_{\mathbb R^3}(|v|^2-3)\mu^{\frac12}H^mf dv, 
 	\end{aligned}
 	\right.
 \end{equation}
 and 
  \begin{equation} \label{mac++}
 	\begin{aligned}
 	 U_m =\int_{\mathbb R^3}\mu^{\frac12}\partial_{v_1}^mf dv,\  \   V_m=\int_{\mathbb R^3}v\mu^{\frac12}\partial_{v_1}^mf dv,\ \ W_m=\frac{1}{6}\int_{\mathbb R^3}(|v|^2-3)\mu^{\frac12}\partial_{v_1}^mf dv.
 	\end{aligned}
 \end{equation}

 With the above definitions and notations, the main result of this part is stated as follows.
 
\begin{proposition}\label{p3}
Recall $\tau=\max \big\{1,\frac{1}{2s}\big\},$ and assume the hypothesis of Theorem \ref{gxt} holds. Let $m \geq 1$ and $T > 1$ be given. Suppose that for all $j \leq m-1$,
\begin{multline}\label{indassum}
\int_{\mathbb{Z}^3}\sup\limits_{1\leq t\leq T}\norm{\widehat{H^{j}f}(t,k)}_{L^2_v}d\Sigma(k)\\+ \int_{\mathbb{Z}^3}\left(\int_{1}^{T}\normm{  \widehat{H^{j}f}(t,k)}^{2}dt\right)^{\frac12}d\Sigma(k) \leq \frac{\eps_0C_{*}^{j}( j!)^\tau}{(j+1)^2},
\end{multline}
where  the operator $H$ and the constant $\eps_0>0$ are   given  in  \eqref{vecM} and Proposition \ref{prp:ve}, respectively,  and   $ C_*\geq 4$ is a given constant.   Then there exists a constant $C$, depending only on the constants $C_0,C_1$ in Section \ref{sec:prelim} 
but independent of  $T, \eps_0, C_* $ and $m,$  such that
\begin{equation*}
	\begin{aligned}
		& \int _{\mathbb{Z}^3}\left(\int_{1}^{T} \big(|\widehat { a_m}(t,k) |^2+
		 |\widehat{b_m}(t,k) |^2+ |\widehat{ c_m }(t,k) |^2 \big)dt\right)^{\frac12}d\Sigma (k)\\
		&\leq C	\int_{\mathbb{Z}^3}\bigg( \int_1^{T}	\normm{ \{\mathbf{I}-\mathbf{P} \}\widehat{H^mf}}^2dt \bigg)^{\frac12}d\Sigma(k)+
 	C  \int_{\mathbb{Z}^3} \sup_{1\leq t\leq T}\norm{\widehat{H^mf}}_{L^2_v}d\Sigma(k)\\
 	&\quad+C\eps_0  \int_{\mathbb{Z}^3} \bigg( \int_1^{T} \normm{\widehat{H^mf}(t,k)}^2dt  \bigg)^{\frac12}d\Sigma(k)
 	+C\big(  \eps_0+C_*^{-1}\big)  \frac{ \eps_0C_*^m  (m!)^{\tau}}{(m+1)^2},
	\end{aligned}
\end{equation*}
where $ a_m ,b_m, c_m $ are defined in \eqref{mac}.
\end{proposition}

\begin{proposition}\label{+p3+}
Recall $\tau=\max \big\{1,\frac{1}{2s}\big\},$ and assume the hypothesis of Theorem \ref{gxt} holds. Let $m \geq 1$ and $T > 1$ be given. 
Suppose that  for any  $j\leq m-1$  we have
\begin{multline}\label{supp+}
\int_{\mathbb{Z}^3}\sup\limits_{1\leq t\leq T}\norm{ \partial_{v_1}^{j}\hat f (t,k)}_{L^2_v}d\Sigma(k)\\+ \int_{\mathbb{Z}^3}\left(\int_{1}^{T}\normm{   \partial_{v_1}^{j}\hat f (t,k)}^{2}dt\right)^{\frac12}d\Sigma(k) \leq \frac{\eps_0L^{j}( j!)^\tau}{(j+1)^2},
\end{multline}
where  $\eps_0>0$ is the constant  given in Proposition \ref{prp:ve}, and $L\geq 4$ is a given constant.   Then there exists a constant $C$,  independent of $T, m, \eps_0, L,$ such that
\begin{equation*}
	\begin{aligned}
		& \int _{\mathbb{Z}^3}\left(\int_{1}^{T} \big(|\widehat { U_m } (t,k)|^2+
		 |\widehat{V^m} (t,k)|^2+ |\widehat{ W_m }(t,k) |^2 \big)dt\right)^{\frac12}d\Sigma (k)\\
		&\leq  C	\int_{\mathbb{Z}^3}\bigg( \int_1^{T}	\normm{ \{\mathbf{I}-\mathbf{P} \}\partial_{v_1}^m\hat f}^2dt \bigg)^{\frac12}d\Sigma(k)+
 	C  \int_{\mathbb{Z}^3} \sup_{1\leq t\leq T}\norm{\partial_{v_1}^m\hat f}_{L^2_v}d\Sigma(k)\\
 	&\quad+C\eps_0  \int_{\mathbb{Z}^3} \bigg( \int_1^{T} \normm{\partial_{v_1}^m\hat f(t,k)}^2dt  \bigg)^{\frac12}d\Sigma(k)
 	+C\big(  \eps_0+L^{-1}\big)  \frac{ \eps_0L^m  (m!)^{\tau}}{(m+1)^2}\\
 	&\quad +C 3^{m} m! \int_{\mathbb Z^3}\bigg(\int_1^{T}\normm{\widehat{Hf}(t,k)}^2dt\bigg)^{\frac12}d\Sigma(k),
	\end{aligned}
\end{equation*}
where  $U_m , V_m, W_m$ are defined in \eqref{mac++} and the operator $H$ in the last line is given in \eqref{vecM}. 
\end{proposition}

The remainder of this section is devoted to the proofs of Propositions \ref{p3} and \ref{+p3+}. We provide a detailed proof of Proposition \ref{p3} and outline that of Proposition \ref{+p3+}, as the arguments are analogous.  

To simplify the notations, we will use the capital letter $C$ throughout this and the next  section  to denote a generic constants, that may vary from line to line and depend only on  the numbers $C_0,C_1$ in   Section \ref{sec:prelim}.   Note these generic constants $C$ as below are independent of $T$ and the derivative order denoted by $m$.   

To prove Proposition \ref{p3}, 
we define the moment functions $\Theta=(\Theta_{jl}(\cdot))_{3\times 3}$ and $ \Lambda=(\Lambda_j (\cdot))_{1 \leq j \leq 3}$ by
 \begin{equation}\label{thetalamb}
 \Theta_{jl}(f)=\int_{\mathbb{R}^3}(v_j v_l-1)\mu^{\frac12}f dv,\qquad
 \Lambda_j (f)=\frac{1}{10}\int_{\mathbb{R}^3}(|v|^2-5)v_j \mu^{\frac12}f dv.
 \end{equation}
Let $a_m$, $b_m$, and $c_m$ be given as in \eqref{mac}, and recall that $\delta_{jl}$ denotes the Kronecker delta function. For each integer $m \geq 1$, we define $\mathcal{K}_m(\hat{f}) = \mathcal{K}_m(\hat{f}(t,k))$ as
  	\begin{equation} \label{if}
 	\begin{aligned}
 	& \mathcal{K}_m  \big  (\hat{ f}\big )=  \frac{1}{1+|k|^2}\!\sum_{j,l=1}^3\!\Big(i k_j\widehat{ b_{m,l}} +ik_l\widehat{b_{m,j}}\ \big|\ \Theta_{jl}\big(\{\mathbf{I}-\mathbf{P}\}\widehat{H^m f}\big) + 2\delta_{jl}\widehat{ c_m }\Big)\\
 	&\qquad +\frac{\rho_0}{1+|k|^2}\sum_{j=1}^3 \Big(	
 	 \Lambda_j \big(\{\mathbf{I}-\mathbf{P}\}\widehat{H^mf}\big)\ \big |\ i k_j\widehat{ c_m } \Big)  + \frac{1}{1+|k|^2}\sum_{j=1}^3 \Big(	\widehat{b_{m,j}}\  \big | \ i k_j \widehat { a_m } \Big),
 	\end{aligned}
 	\end{equation}
 	where $\rho_0 \geq 1$ is a constant to be determined later, and we use the notation
 $(z_1\ | \ z_2)=z_1  \overline{z_2},$ with $\overline{z_2}$  denoting the complex conjugate of 
 $z_2\in\mathbb C$.    
 
\begin{lemma} \label{l5.2a}
  	 With   properly chosen large constant  $\rho_0$ in the representation  \eqref{if} of  $\mathcal{K}_m (\hat{ f}),$ 
 the following inequality holds for each $k \in \mathbb{Z}^3$:
   \begin{equation} \label{4b}
 	\begin{aligned}
 	&\partial_t  {\rm Re}  \mathcal{K}_m \big (\hat{ f}(t,k)\big)+\frac12 \frac{ |k|^2}{1+|k|^2}\big(|\widehat { a_m }(t,k)|^2+
 	|\widehat{b_m}(t,k)|^2+|\widehat{ c_m }(t,k)|^2\big) \\
 	&\leq C
 	\normm{ \{\mathbf{I}-\mathbf{P}\}	\widehat{H^mf}}^2
 	+C\sum_{j,l\leq 3}|\Theta_{jl}(	\widehat{G}_m)|^2 +C\sum_{j\leq 3}|\Lambda_j (	\widehat{G}_m)|^2+C\Big|\int_{\mathbb R^3}	 \phi	\widehat{G}_m  dv \Big|^2,
 	\end{aligned}
 	\end{equation}
 	where    $\phi=\phi(v)$ denotes one of the functions $\mu^{\frac12},~v\mu^{\frac12}$	and $\frac{1}{6}(|v|^2-3)\mu^{\frac12}$, and  $\Theta_{jl}, \Lambda_j$ are given in \eqref{thetalamb}, and 
\begin{equation}\label{defg}
 G_m:=  \mathcal{L}\{\mathbf{I}-\mathbf{P} \}H^mf+ H^m \Gamma(f,f)+ [ H^{m}, \ \mathcal{L}] f,
\end{equation}
with $[\cdot, \cdot]$
denoting  the commutator of two operators. 
\end{lemma}

\begin{proof}
We begin by deriving the macroscopic equations for $a_m$, $b_m$, and $c_m$. Applying $H^m$ to equation \eqref{3} gives
	\begin{equation}\label{hmf}
	\begin{aligned}
	\inner{\partial_{t} +v\cdot\partial_{x}  -\mathcal{L}} H^{m}  f &= -[H^{m}, \partial_{t}+v\cdot\partial_{x}] f + [ H^{m}, \ \mathcal{L}] f + H^{m}\Gamma(f, f)\\
	&=   [ H^{m}, \ \mathcal{L}] f + H^{m}\Gamma(f, f),
	\end{aligned}
	\end{equation}
	the last identity  using \eqref{kehigher}. Thus we perform   marco-micro decomposition in the above equation to obtain that  
 \begin{eqnarray}\label{5}
 \partial_t\mathbf{P}H^mf+v\cdot\partial_{x}\mathbf{P}H^mf=-
 \partial_t\{\mathbf{I}-\mathbf{P} \}H^mf+R_m+G_m,
 \end{eqnarray}	
where $G_m$ is given in \eqref{defg} and 
 \begin{equation*}\label{rg}
 	 R_m=-v\cdot\partial_{x} \{\mathbf{I}-\mathbf{P} \}H^mf.
 \end{equation*}
Taking the  velocity moments 
 \begin{eqnarray*}
 	\mu^{\frac12},~v\mu^{\frac12},~\frac{1}{6}(|v|^2-3)\mu^{\frac12},~(v_jv_l-1)\mu^{\frac12},
 	~\frac{1}{10}(|v|^2-5)v_j\mu^{\frac12},	
 \end{eqnarray*}	
in   equation \eqref{5}, we find that the quantities $(a_m, b_m, c_m)$ defined in \eqref{mac} satisfy the following fluid-type system: 
 \begin{eqnarray} \label{fys}
 \left\{
 \begin{aligned}
 & \partial_t a_m+\partial_x\cdot b_m=\int_{\mathbb R^3}\mu^{\frac12}  G_m dv ,\\
 &\partial_t b_m+\partial_x( a_m+ 2 c_m )+\partial_x\cdot \Theta\big (\{\mathbf{I}-\mathbf{P} \}H^mf\big)=\int_{\mathbb R^3} v\mu^{\frac12} G_m dv,\\
 &\partial_t c_m+\frac{1}{3}\partial_x\cdot b_m+
 \frac{5}{3}\partial_x\cdot \Lambda\big (\{\mathbf{I}-\mathbf{P} \}H^mf\big )=\frac{1}{6}\int_{\mathbb R^3} (|v|^2-3)\mu^{\frac12} G_m dv,\\
 &\partial_{t} \big\{\Theta_{jl}( \{\mathbf{I}-\mathbf{P}\}H^mf)+2 c_m \delta_{jl} \big\}+
 \partial_{x_j} b_{m,l}+\partial_{x_l} b_{m,j}=\Theta_{jl}(R_m+G_m),\\
 &\partial_{t}\Lambda_j ( \{\mathbf{I}-\mathbf{P}\}H^mf)+\partial_{x_j} c_m =
 \Lambda_j (R_m+G_m),
 \end{aligned}
 \right.
 \end{eqnarray}
where $1\leq j,l\leq 3$, and $G_m,  \Theta_{j,l}, \Lambda_j$ are given in \eqref{defg} and  \eqref{thetalamb}, and $\delta_{jl}$ is the Kronecker delta function.  
 
 \underline{\it Estimate of $\widehat{b_m}$.} Using the notations in \eqref{thetalamb} and \eqref{defg}, we claim that
  	\begin{equation} \label{4.05}
 	\begin{aligned}
 	&\partial_{t} \text{Re}\sum_{j,l=1}^3\Big( i k_j\widehat{b_{m,l}} + ik_l\widehat{b_{m,j}}\ \big|\ 
 	 \Theta_{jl}\big ( \{\mathbf{I}\!-\!\mathbf{P}\}\widehat{H^m f}\big ) + 2\delta_{jl}\widehat{ c_m }  \Big) + 
 	|k|^2|\widehat{b_m}|^2+2|k\cdot \widehat{b_m} |^2\\
 	& \leq  \frac14 |k|^2|\widehat { a_m }|^2 +C  |k|^2|\widehat{ c_m }|^2+
 	C  |k|^2 \normm{ \{\mathbf{I}-\mathbf{P}\}\widehat{H^mf}}^2\\
 	&\qquad+C \sum_{1\leq j,l\leq 3}|\Theta_{jl}(	\widehat{G}_m )|^2+C \sum_{1\leq j\leq 3}\big|\int_{\mathbb R^3}	 v_j\mu^{\frac12} \widehat{G}_m dv \big|^2.
 	\end{aligned}	
 	\end{equation}	
 To prove this, take the Fourier transform in the fourth equation of \eqref{fys}. Using the inner product $(z_1 | z_2) = z_1 \overline{z_2}$, we compute: 	\begin{eqnarray*} 
 		\begin{aligned}
 			&\sum_{j,l=1}^3|i k_j\widehat{b_{m,l}}+ik_l\widehat{b_{m,j}}|^2\\
 			&= 
 		 \sum_{j,l=1}^3 \Big( i k_j\widehat{b_{m,l}}+ik_l\widehat{b_{m,j}}\ \big| \  
 			  \Theta_{jl}(\widehat{R}_m+\widehat{G}_m )- \partial_{t} \Big  \{\Theta_{jl}( \{\mathbf{I}\!-\!\mathbf{P}\}\widehat{H^mf})+2\widehat{ c_m }\delta_{jl}\Big \}  \Big)\\
 			&=\sum_{j,l=1}^3 \Big( i k_j\widehat{b_{m,l}}+ik_l\widehat{b_{m,j}}\ \big|\ 
 			\Theta_{jl}(\widehat{R}_m+\widehat{G}_m )  \Big) \\
 			&\quad 
 			 -\partial_{t}\sum_{j,l=1}^3\Big ( i k_j\widehat{b_{m,l}}+ik_l\widehat{b_{m,j}}\ \big| \ 
 			\Theta_{jl}\big ( \{\mathbf{I}-\mathbf{P}\}\widehat{H^mf}\big)+2\widehat{ c_m }\delta_{jl}  \Big)\\
 			&\quad+\sum_{j,l=1}^3\Big( i k_j\partial_{t}\widehat{b_{m,l}}+ik_l\partial_{t}\widehat{b_{m,j}}\ \big|\ 
 			\Theta_{jl}\big ( \{\mathbf{I}-\mathbf{P}\}\widehat{H^mf}\big)+2\widehat{ c_m }\delta_{jl}  \Big).
 		\end{aligned}	
 	\end{eqnarray*} 
 This,  with	  the identity
 	\begin{eqnarray*} 
 		\sum_{j,l=1}^3|i k_j \widehat{b_{m,l}}+ik_l\widehat{ b_{m,j}}|^2=2|k|^2
 		|\widehat{b_m}|^2+2|k\cdot \widehat{b_m} |^2,
 	\end{eqnarray*}	
 yields that
 	\begin{equation} \label{4d}
 	\begin{aligned}
 &\partial_{t}\sum_{j,l=1}^3\Big ( i k_j\widehat{b_{m,l}}+ik_l\widehat{b_{m,j}}\ \big| \ 
 			\Theta_{jl}\big ( \{\mathbf{I}-\mathbf{P}\}\widehat{H^mf}\big)+2\widehat{ c_m }\delta_{jl}  \Big) +2|k|^2
 	|\widehat{b_m}|^2+2|k\cdot \widehat{b_m} |^2\\	
 	&= 	\sum_{j,l=1}^3\Big( i k_j\partial_{t}\widehat{b_{m,l}}+ik_l\partial_{t}\widehat{b_{m,j}}\ \big|\ 
 			\Theta_{jl}\big ( \{\mathbf{I}-\mathbf{P}\}\widehat{H^mf}\big)+2\widehat{ c_m }\delta_{jl}  \Big)\\
 	&\quad +\sum_{j,l=1}^3 \Big( i k_j\widehat{b_{m,l}}+ik_l\widehat{b_{m,j}}\ \big|\ 
 			\Theta_{jl}(\widehat{R}_m+\widehat{G}_m )  \Big) \\
 			&:=S_1+S_2.
 	\end{aligned}	
 	\end{equation}	
 	We first treat 
  $S_1$ and write
  \begin{multline*}
  		\sum_{j,l=1}^3\Big( i k_j\partial_{t}\widehat{b_{m,l}} \ \big|\ 
 			\Theta_{jl}\big ( \{\mathbf{I}-\mathbf{P}\}\widehat{H^mf}\big)+2\widehat{ c_m }\delta_{jl}  \Big)\\
 			= i	\sum_{j,l=1}^3 \Big(   \partial_{t}\widehat{b_{m,l}} \ \big|\ 
 			k_j \Theta_{jl}\big ( \{\mathbf{I}-\mathbf{P}\}\widehat{H^mf}\big)+2k_j \widehat{ c_m }\delta_{jl}  \Big).
  \end{multline*}
 From the second equation in \eqref{fys}, taking the Fourier transform in $x$ gives
 	\begin{eqnarray} \label{4.8}
 	\partial_t \widehat{b_{m,l}} +i k_l(\widehat { a_m }+2\widehat{ c_m })+\sum_{j=1}^3i k_j \Theta_{lj}\big (\{\mathbf{I}-\mathbf{P}\}\widehat{H^mf}\big)=\int_{\mathbb R^3}	  v_l\mu^{\frac12} \widehat{G}_m dv.
 	\end{eqnarray}
 	Combining the above two equations we obtain
 		\begin{equation*}
 	\begin{aligned}
 &\sum_{j,l=1}^3\Big( i k_j\partial_{t}\widehat{b_{m,l}} \ \big|\ 
 			\Theta_{jl}\big ( \{\mathbf{I}-\mathbf{P}\}\widehat{H^mf}\big)+2\widehat{ c_m }\delta_{jl}  \Big)\leq  \frac{ |k|^2 |\widehat { a_m }|^2}{8 } +C|k|^2|\widehat{ c_m }|^2\\
 			&
 	\qquad\quad \qquad\qquad\quad +C |k|^2\sum_{j,l}\big|\Theta_{jl}\big (\{\mathbf{I}-\mathbf{P}\}\widehat{H^mf}\big )\big|^2+	C \sum_{j}\Big|\int  v_j\mu^{1\over2} \widehat{G}_m dv\Big|^2 \\
 	&\leq  \frac18  |k|^2|\widehat { a_m }|^2+ C  |k|^2 |\widehat{ c_m }|^2+
 	C|k|^2\normm{ \{\mathbf{I}-\mathbf{P}\}\widehat{H^mf}}^2 +C \sum_{j}\Big|\int_{\mathbb R^3}  v_j\mu^{\frac12} \widehat{G}_m dv\Big|^2,
 	\end{aligned}	
 	\end{equation*}
 	the last inequality using   definition \eqref{thetalamb} of $\Theta_{j,l}.$  A similar bound holds for the remaining term in $S_1$, so we conclude
 	  \begin{equation}\label{4.9}
 	  	S_1\leq    \frac14 |k|^2|\widehat { a_m }|^2+ C  |k|^2 |\widehat{ c_m }|^2+
 	C |k|^2\normm{ \{\mathbf{I}-\mathbf{P}\}\widehat{H^mf}}^2 +C\sum_{j}\Big|\int_{\mathbb R^3}  v_j\mu^{\frac12} \widehat{G}_m dv\Big|^2.
 	  \end{equation}
It remains to estimate $S_2$, 	recalling 
\begin{equation*}
S_2=	\sum_{j,l=1}^3 \Big( i k_j\widehat{b_{m,l}}+ik_l\widehat{b_{m,j}}\ \big|\ 
 			\Theta_{jl}(\widehat{R}_m+\widehat{G}_m )  \Big). 
\end{equation*}
Notice that
 	$
 		R_m=-v\cdot\partial_{x}\{\mathbf{I}-\mathbf{P} \}H^mf,
$ and thus
 	\begin{equation*}
 	| \Theta_{jl}(\widehat{R}_m)|^2 \leq C |k|^2 \normm{ \{\mathbf{I}-\mathbf{P}\}\widehat{H^mf}}^2.
 	\end{equation*}	
 	Then  
 	\begin{eqnarray*} \label{4.7}
 	\begin{aligned}
 	S_2\leq&   |k|^2\sum_{l}|\widehat{b_{m,l}}|^2+C \sum_{jl}\big (| \Theta_{jl}(\widehat{R}_m)|^2+| \Theta_{jl}(\widehat{G}_m )|^2\big )\\
 	\leq&  |k|^2|\widehat{b_m}|^2+C |k|^2 \normm{ \{\mathbf{I}-\mathbf{P}\}\widehat{H^mf}}^2+C \sum_{jl}| \Theta_{jl}(\widehat{G}_m )|^2.
 	\end{aligned}	
 	\end{eqnarray*}
 	Substituting the above estimate and \eqref{4.9} into 
 	\eqref{4d} yields  that
 	\begin{equation*}
 		\begin{aligned}
 &\partial_{t}\sum_{j,l=1}^3\Big ( i k_j\widehat{b_{m,l}}+ik_l\widehat{b_{m,j}}\ \big| \ 
 			\Theta_{jl}\big ( \{\mathbf{I}-\mathbf{P}\}\widehat{H^mf}\big)+2\widehat{ c_m }\delta_{jl}  \Big) +2|k|^2
 	|\widehat{b_m}|^2+2|k\cdot \widehat{b_m} |^2\\	
 	&\leq \frac14 |k|^2|\widehat { a_m }|^2+ C  |k|^2 |\widehat{ c_m }|^2+
 	C |k|^2\normm{ \{\mathbf{I}-\mathbf{P}\}\widehat{H^mf}}^2 +C \sum_{j}\Big|\int_{\mathbb R^3} \widehat{G}_m v_j\mu^{\frac12} dv\Big|^2\\
 	&\quad+  |k|^2|\widehat{b_m}|^2 +C \sum_{jl}| \Theta_{jl}(\widehat{G}_m )|^2.
 	\end{aligned}	
 	\end{equation*}
 	This gives assertion \eqref{4.05}. 
 	
 		 \underline{\it Estimate of  $\widehat{ c_m }$}. Recall the notations in \eqref{thetalamb} and \eqref{defg}.  We now show that for any $0 < \eps < 1$, 
 	\begin{equation} \label{4.10a}
 	\begin{aligned}
 	&\partial_t 	\text{Re} \sum_{j=1}^3\Big( \Lambda_{j}(\{\mathbf{I}-\mathbf{P}\}\widehat{H^mf})\ | \ i k_j\widehat{ c_m }\Big)+\frac12|k|^2|\widehat{ c_m }|^2\leq \eps |k\cdot \widehat{b_m}|^2\\
 	&\qquad\ +C_{\eps}\bigg[\sum_j\big |\Lambda_j(\widehat{G}_m )\big |^2+ \Big| \int  (|v|^2-3)\mu^{1\over 2} \widehat{G}_m dv  \Big|^2 +  |k|^2 \normm{ \{\mathbf{I}-\mathbf{P}\}\widehat{H^mf}}^2\bigg].
 	\end{aligned}
 	\end{equation}
 	To do so we perform 
  the partial  Fourier transform  in the fifth equation of  \eqref{fys},   to get 
 	\begin{eqnarray*}
 		\partial_t \Lambda_j \big (\{\mathbf{I}-\mathbf{P}\}\widehat{H^mf}\big ) +i k_j\widehat{ c_m }=\Lambda_j\big (\widehat{R}_m+\widehat{G}_m \big),
 	\end{eqnarray*}
and thus
 	\begin{eqnarray} \label{4.10la}
 	\begin{aligned}
 	& \partial_t \Big (\Lambda_j (\{\mathbf{I}-\mathbf{P}\}\widehat{H^mf})\ \big| \  ik_j \widehat{ c_m } \Big)+|k_j|^2 |\widehat{ c_m }|^2\\
 	&=\big (\Lambda_j(\widehat{R}_m+\widehat{G}_m )\ |\ ik_j \widehat{ c_m }\big)+\big (\Lambda_j (\{\mathbf{I}-\mathbf{P}\}\widehat{H^mf})\ |\ 
 	ik_j \partial_t \widehat{ c_m }\big).
 	\end{aligned}	
 	\end{eqnarray}
 	Moreover,  
recalling 
 	$
 		R_m=-v\cdot\partial_{x}\{\mathbf{I}-\mathbf{P} \}H^mf
$  and defintion \eqref{thetalamb} of $\Lambda_j,$ we have
 	\begin{equation*}
 	| \Lambda_{j}(\widehat{R}_m)|  \leq C \big (1+|k|\big )\normm{ \{\mathbf{I}-\mathbf{P}\}\widehat{H^mf}}.
 	\end{equation*}	
 Then  
 	\begin{eqnarray} \label{4.11a}
 	\begin{aligned}
 	&\textrm{Re}\,\big (\Lambda_j(\widehat{R}_m+\widehat{G}_m )\ |\ ik_j \widehat{ c_m }\big)  \leq \frac12 |k_j|^2 |\widehat{ c_m }|^2+C 	\big (|\Lambda_j
 	(\widehat{R}_m)|^2+|\Lambda_j(\widehat{G}_m )|^2 \big)\\
 	&\leq \frac{1}{2}  |k_j|^2 |\widehat{ c_m }|^2+C\big |\Lambda_j\big (\widehat{G}_m\big )\big |^2 +C(1+|k|^2)\normm{\{\mathbf{I}-\mathbf{P}\}\widehat{H^mf}}^2.
 	\end{aligned}  				
 	\end{eqnarray}	
From the third equation in \eqref{fys} as well as definition \eqref{thetalamb} of $\Lambda_j$, taking the Fourier transform in $x$ gives
 	\begin{eqnarray*}
 		\partial_t \widehat{ c_m } +\frac{1}{3}i k \cdot \widehat{b_m}+\frac{5}{3}\sum_{j=1}^3i k_j  \Lambda_j\big(\{\mathbf{I}-\mathbf{P}\}\widehat{H^mf}\big ) =  \frac{1}{6}\int_{\mathbb R^3} (|v|^2-3)\mu^{\frac12} \widehat{G}_m dv. 	\end{eqnarray*}	
 Thus, for any $\eps>0,$
 	\begin{equation}\label{4.12a}
 	\begin{aligned}
 	&\textrm{Re}\, \big (\Lambda_j (\{\mathbf{I}-\mathbf{P}\}\widehat{H^mf})\ |\ 
 	ik_j \partial_t \widehat{ c_m }\big)\\
 	& \leq  \eps |k\cdot \widehat{b_m}|^2
 	+C_\eps \sum_j  |k|^2|  \Lambda_j (\{\mathbf{I}-\mathbf{P}\}\widehat{H^mf})|^2 +C_\eps  \Big| \int_{\mathbb R^3}  (|v|^2-3)\mu^{1\over 2} \widehat{G}_m dv  \Big|^2\\
 	& \leq \eps |k\cdot \widehat{b_m}|^2
 	+C_\eps |k|^2\normm{ \{\mathbf{I}-\mathbf{P}\}\widehat{H^mf}}^2 +C_\eps \Big| \int_{\mathbb R^3}  (|v|^2-3)\mu^{1\over 2}  \widehat{G}_m dv  \Big|^2.
 	\end{aligned}
 	\end{equation}	
 	Combining the above estimate and \eqref{4.11a} with \eqref{4.10la} we conclude   assertion \eqref{4.10a}.
 
 \underline{\it Estimate of $\widehat { a_m }$}.  Recall the notations in \eqref{defg}, we claim that
 	\begin{eqnarray} \label{4.13}
 	\begin{aligned}
 	&\partial_t \text{Re} \sum_{j=1}^3\big (\widehat{b_{m,j}}\ | \ ik_j\widehat { a_m } \big )+\frac34 |k|^2|\widehat { a_m }|^2\\
 &\leq  \frac{3}{2}|k\cdot \widehat{b_m}|^2+C  |k|^2|\widehat{ c_m }|^2 +C  |k|^2\normm{ \{\mathbf{I}-\mathbf{P}\}\widehat{H^mf}}^2\\
 	&\quad +C\sum_j \Big| \int_{\mathbb R^3}	  v_j\mu^{\frac12} \widehat{G}_m dv \Big|^2
 	+   \Big|\int_{\mathbb R^3} \mu^{\frac12}\widehat{G}_m dv\Big|^2.
 	\end{aligned} 
 	\end{eqnarray}
 	Indeed, we  take the complex inner product with $ik_j \widehat { a_m }$ in equation \eqref{4.8},   to  get
 	\begin{eqnarray}\label{4.14a}
 	\begin{aligned}
 	&\partial_t \sum_j\big (\widehat{b_{m,j}}\ |\ ik_j \widehat { a_m }\big ) +| k|^2 |\widehat { a_m }|^2
 	\\
 	& =  -2\sum_j\big (ik_j\widehat{ c_m }\ | \ i k_j\widehat { a_m }\big )-\sum_{j,l}\big(i k_l \Theta_{jl}(\{\mathbf{I}-\mathbf{P}\}\widehat{H^mf})\ |\ ik_j\widehat { a_m }\big)\\
 	&\quad +\sum_j\bigg(\int_{\mathbb R^3}	 v_j\mu^{\frac12} \widehat{G}_m dv \  \big|\ ik_j\widehat { a_m }\bigg)+ \sum_j\big (\widehat{b_{m,j}}\ | \ i k_j  \partial_t\widehat { a_m }\big ).
 	\end{aligned}
 	\end{eqnarray}
 	The real parts of the first two terms on the right hand side of \eqref{4.14a} are bounded from above by
 	\begin{eqnarray*}
 		\frac18 |k|^2 |\widehat { a_m }|^2+C |k|^2\Big( |\widehat{ c_m }|^2+\normm{\{\mathbf{I}-\mathbf{P}\}\widehat{H^mf}}^2 \Big).
 	\end{eqnarray*}	
 	Meanwhile the real part of the third term on the right hand side of \eqref{4.14a} is bounded from above by	
 	\begin{eqnarray*}
 		\frac18 |k|^2 |\widehat { a_m }|^2+C\sum_j \Big| \int_{\mathbb R^3}	  v_j\mu^{\frac12} \widehat{G}_m dv \Big|^2.
 	\end{eqnarray*}			
 	For the last term in \eqref{4.14a}, we use the fact that
 	\begin{eqnarray*}
 		\partial_t \widehat { a_m }+i k\cdot \widehat{b_m}=\int_{\mathbb R^3} \mu^{\frac12} \widehat{G}_m dv,
 	\end{eqnarray*}
 which follows by taking the  partial  Fourier transform in the first equation of \eqref{fys};  this gives 
 	\begin{eqnarray*}
 		\begin{aligned}
 		\textrm{Re} \sum_j\big (\widehat{b_{m,j}}\ |\ i k_j  \partial_t\widehat { a_m }\big ) 
 			 &= \sum_j\big(\widehat{b_{m,j}}\ |\  k_j k\cdot \widehat{b_m} \big)+\textrm{Re}\sum_j\Big(\widehat{b_{m,j}}\ \big|\ ik_j\int_{\mathbb R^3}\mu^{\frac12} \widehat{G}_m  dv\Big )\\
 			&\leq \frac32 |k\cdot \widehat{b_m}|^2+   \Big|\int_{\mathbb R^3} \mu^{\frac12} \widehat{G}_m dv\Big|^2.
 		\end{aligned}
 	\end{eqnarray*}
 Then, substituting  the above estimates into    \eqref{4.14a}  yields assertion \eqref{4.13}.
 
 We combine estimates \eqref{4.05} and \eqref{4.13} to conclude that
 \begin{equation*}
 	\begin{aligned}
 	&\partial_{t} \text{Re}\bigg[	\sum_{j,l=1}^3\Big( i k_j\widehat{b_{m,l}} + ik_l\widehat{b_{m,j}}  \big|  
 	 \Theta_{jl}\big ( \{\mathbf{I}\!-\!\mathbf{P}\}\widehat{H^m f}\big ) + 2\delta_{jl}\widehat{ c_m }  \Big) + \sum_{j=1}^3\big (\widehat{b_{m,j}}\ | \ ik_j\widehat { a_m } \big )\bigg]
 	 \\
 	 &\qquad  +\frac12 |k|^2|\widehat { a_m }|^2+ 
 	|k|^2|\widehat{b_m}|^2+\frac12|k\cdot \widehat{b_m} |^2\\
 	& \leq   C  |k|^2|\widehat{ c_m }|^2+
 	C  |k|^2 \normm{ \{\mathbf{I}-\mathbf{P}\}\widehat{H^mf}}^2 +C \sum_{1\leq j,l\leq 3}|\Theta_{jl}(	\widehat{G}_m )|^2\\
 	&\quad+C \sum_{1\leq j\leq 3}\Big|\int_{\mathbb R^3}	 v_j\mu^{\frac12} \widehat{G}_m dv \Big|^2  
 	+   \Big|\int_{\mathbb R^3} \mu^{\frac12}\widehat{G}_m dv\Big|^2,
 	\end{aligned}
 \end{equation*}
 which, together with \eqref{4.10a}, implies \eqref{4b}, provided $\eps$ in \eqref{4.12a} is chosen sufficiently small and $\rho_0$ in \eqref{if} is sufficiently large. This completes the proof of Lemma \ref{l5.2a}.
	\end{proof}

 The following lemmas are devoted to estimating the terms appearing on the right-hand side of inequality \eqref{4b}. 
 
  \begin{lemma} \label{lem: ma}
Under the hypothesis of Proposition \ref{p3},  there exists a constant $C > 0$, independent of $T$, $m$, $\eps_0$, and $C_*$, such that
 		\begin{equation} \label{4.2}
 	\begin{aligned}
 	&\int_{\mathbb{Z}^3} \bigg[ \int_1^{T}\big |\big (\mathscr{F}_x\big (H^m \Gamma(f,f)\big), (v_j v_l-1)\mu^{\frac12}\big)_{L^2_v}\big|^2 dt\bigg]^{\frac12}d\Sigma(k)\\
 	&	\leq 	C\eps_0 \int_{\mathbb{Z}^3} \sup_{1\leq t\leq T}\norm{\widehat{H^mf}(t,k)}_{L^2_v}d\Sigma(k)\\
 	&\quad+C\eps_0  \int_{\mathbb{Z}^3} \bigg( \int_1^{T} \normm{\widehat{H^mf}(t,k)}^2dt  \bigg)^{\frac12}d\Sigma(k)
 	+C  \eps_0^2 \frac{ C_*^m  (m!)^{\tau}}{(m+1)^2}
 	\end{aligned}
 	\end{equation}
 	and
 		\begin{equation} \label{4.2a}
 		\int_{\mathbb{Z}^3}\bigg[ \int_1^{T}\big|(\mathscr{F}_x([H^m,\mathcal{L}]f), (v_j v_l-1)\mu^{\frac12})_{L^2_v}\big|^2 dt\bigg]^{\frac12}d\Sigma(k)
 			\leq  \frac{C}{C_*}  \frac{ \eps_0 C_*^{m }~ (m!)^{\tau}}{(m+1)^2}.
 		\end{equation}
 \end{lemma}

\begin{proof}
Applying Leibniz formula and using \eqref{trb}, we obtain
\begin{multline*}
	H^{m}{\Gamma}(f,f)=\sum\limits_{n=0}^{ m }\sum\limits_{p=0}^{  n}\binom{m}{n}\binom{n}{p}\mathcal{T}(H^{m-n}f,\  H^{n-p}f,\  H^{p}\mu^{\frac12})\\
	=\sum\limits_{n=0}^{ m }\sum\limits_{p=0}^{  n}\binom{m}{n}\binom{n}{p}\mathcal{T}(H^{m-n}f, \ H^{n-p}f,\  \partial_{v_1}^{p}\mu^{\frac12}) 
\end{multline*}
where $\mathcal{T}$ is defined in \eqref{matht}, and the last equality follows from the identity $H^p\mu^{\frac12} = \partial_{v_1}^{p}\mu^{\frac12}$. Taking the partial Fourier transform in $x$ and applying the triangle inequality yields
\begin{equation}\label{uppboun}
\int_{\mathbb{Z}^3}\left(\int_{1}^{T}\big|\big(\mathscr{F}_x(H^{m}\Gamma(f,f)), (v_jv_l-1)\mu^{\frac12}\big)_{L^2_v}\big|^2dt\right)^{\frac12}d\Sigma(k) \leq \mathcal{J}_1+\mathcal{J}_2+\mathcal{J}_3,
\end{equation}
with
\begin{equation}\label{J1j2j3}
\left\{
\begin{aligned}
	\mathcal{J}_1&=\int_{\mathbb{Z}^3}\sum\limits_{p=0}^{m }\binom{m}{p} \bigg[\int_{1}^{T}\big|\big (\hat{\mathcal{T}}(\hat{f},  \widehat{H^{m-p}f},  \partial_{v_1}^{p}\mu^{\frac12}),   (v_jv_l-1)\mu^{\frac12}\big)_{L^2_v}\big|^2dt\bigg]^{1\over2}d\Sigma(k),\\
	\mathcal{J}_2&=\int_{\mathbb{Z}^3} \sum\limits_{n=1}^{m-1}\sum\limits_{p=0}^n\binom{m}{n}\binom{n}{p}\\
	&\quad\quad\quad\times\bigg[\int_{1}^{T}\big|\big(\hat{\mathcal{T}}(\widehat{H^{m-n}f}, \widehat{H^{n-p}f}, \ \partial_{v_1}^{p}\mu^{1\over2}), (v_jv_l-1)\mu^{\frac12}\big)_{L^2_v}\big|^2dt\bigg]^{1\over2}d\Sigma(k),\\
	\mathcal{J}_3&=\int_{\mathbb{Z}^3}\left(\int_{1}^{T}\big|\big (\hat{\mathcal T}(\widehat{H^{m}f},\ \hat{f}, \ \mu^{\frac12}),\  (v_jv_l-1)\mu^{\frac12}\big)_{L^2_v}\big|dt\right)^{\frac12}d\Sigma(k).
\end{aligned}
\right.
\end{equation}
We proceed to estimate $ \mathcal J_1,\mathcal J_2,$ and $\mathcal J_3$ as follows.

 \smallskip
\noindent\underline{\it Estimate on $\mathcal J_1$}.
We decompose $\mathcal{J}_1$ as below:
\begin{equation}\label{j1}
\begin{aligned}
\mathcal{J}_1 & \leq   \int_{\mathbb{Z}^3}\left(\int_{1}^{T}\big|\big(\hat{\mathcal T}(\hat{f}, \widehat{H^{m}f}, \mu^{\frac12}), (v_jv_l-1)\mu^{\frac12}\big)_{L^2_v}\big|^2dt\right)^{1\over 2}d\Sigma(k)
\\ & \quad +
\int_{\mathbb{Z}^3}\sum\limits_{p=1}^{m}\binom{m}{p}\bigg[\int_{1}^{T}\big|\big(\hat{\mathcal{T}}(\hat{f}, \widehat{H^{m-p}f}, \ \partial_{v_1}^{p} \mu^{\frac12}), (v_jv_l-1)\mu^{\frac12}\big)_{L^2_v}\big|^2dt\bigg]^{1\over2}d\Sigma(k)\\
&:=\mathcal J_{1,1}+\mathcal J_{1,2}.
\end{aligned}
\end{equation}
A direct verification yields that
\begin{equation}\label{mu}
\forall \ p\geq 0, \quad \big| \partial_{v_1}^{p}  \mu^{\frac12}\big| \leq
 2^p p!  \mu^{\frac14}.
\end{equation}
We now apply \eqref{ketres} with $\omega = \partial_{v_1}^p \mu^{1/2}$ to estimate $\mathcal{J}_{1,2}$ defined in \eqref{j1} as follows:
\begin{equation}\label{tecest1}
	\begin{aligned}
		&\mathcal J_{1,2} =\int_{\mathbb{Z}^3}\sum\limits_{p=1}^{m }\binom{m}{p} \bigg[\int_{1}^{T}\big|\big(\hat{\mathcal{T}}(\hat{f}, \widehat{H^{m-p}f}, \partial_{v_1}^{p}\mu^{\frac12}), (v_jv_l-1)\mu^{\frac12}\big)_{L^2_v}\big|^2dt\bigg]^{1\over2}d\Sigma(k)\\
		&\leq C \int_{\mathbb{Z}_k^3}\sum\limits_{p=1}^{m }\binom{m}{p} 2^p p! \bigg[ \int_{1}^{T} \bigg(
		 \int_{\mathbb Z_\ell^3 } \norm{\hat f(k-\ell)}_{L^2_v}     \normm{\widehat{H^{m-p}f}(\ell)} d\Sigma(\ell)\bigg) ^2 dt\bigg]^{1\over2}d\Sigma(k)\\
		&\leq C \sum\limits_{p=1}^{m }\binom{m}{p} 2^p p!  \int_{\mathbb{Z}_k^3} \bigg[\int_{\mathbb Z_\ell^3 } \Big( \int_{1}^{T}  
		 \norm{\hat f(k-\ell)}_{L^2_v}^2     \normm{\widehat{H^{m-p}f}(\ell)}^2 dt\Big)^{1\over2} d\Sigma(\ell) \bigg] d\Sigma(k)\\
		&\leq C \sum\limits_{p=1}^{ m }\binom{m}{p} 2^p p! \\
		&\qquad \times \int_{\mathbb{Z}_k^3} \bigg[\int_{\mathbb Z_\ell^3 }  \Big(\sup_{1\leq t \leq T}\norm{\hat f(k-\ell)}_{L^2_v}\Big)\Big( \int_{1}^{T}  
		     \normm{\widehat{H^{m-p}f}(\ell)}^2 dt\Big)^{1\over2} d\Sigma(\ell) \bigg] d\Sigma(k)\\
		     &\leq C \sum\limits_{p=1}^{m }\binom{m}{p} 2^p p! \bigg[\int_{\mathbb{Z}^3}  \sup_{1\leq t\leq T}\norm{\hat f}_{L^2_v} d\Sigma(k)	\bigg]	        \int_{\mathbb{Z}^3}  \Big( \int_{1}^{T}  
		     \normm{\widehat{H^{m-p}f}}^2 dt\Big)^{1\over2}  d\Sigma(k),
	\end{aligned}
\end{equation}
where the second inequality uses Minkowski's inequality and the last inequality follows from Young's inequality for discrete convolutions. Then, using the inductive assumption \eqref{indassum}, we obtain
\begin{equation}\label{tecest2}
	\begin{aligned}
	\mathcal J_{1,2} &	\leq C\sum\limits_{p=1}^{m }\binom{m}{p} 2^p p! 
	\eps_0	   \frac{\eps_0 C_*^{m -p }[(m-p)!]^\tau }{(m-p+1)^2} \\
	 	&\leq C  \eps_0^2   \sum\limits_{p=1}^m \frac{m! }{(m-p)!} 2^{p}  \frac{C_*^{m -p }[(m-p)!]^\tau }{(m-p+1)^2}\leq C  \eps_0^2C_*^{ m } (m!)^\tau \sum\limits_{p=1}^m   \frac{2^{-p}}{(m-p+1)^2},
	\end{aligned}
\end{equation}
where the last line uses the condition $C_*>4$.   
To estimate the last term in \eqref{tecest2}, denote by $[\frac{m}{2}]$  the largest integer less than or equal to $\frac{m}{2}$. Then
\begin{equation}\label{teccom}
\begin{aligned}
	 & \sum\limits_{p=1}^m   \frac{2^{ -p }}{(m-p+1)^2}   \leq        \sum\limits_{p=1}^{[\frac{m}{2}] }\frac{1}{(m-p+1)^2}2^{-p}+\sum\limits_{p=[\frac{m}{2}]+1 }^{  m }\frac{1}{(m-p+1)^2}2^{-p}   \\
&\leq   C    \bigg\{\sum\limits_{p=1}^{ [\frac{m}{2}] }\frac{1}{(m+1)^2}2^{-p} +\sum\limits_{p=[\frac{m}{2}]+1}^{   m}\frac{1}{(m+1)^2}(m+1)^22^{-p} \bigg\} \leq     \frac{C}{(m+1)^2},	 \end{aligned}
\end{equation}
where the last inequality uses the estimate
\begin{equation*}
\sum\limits_{p=[\frac{m}{2}]+1}^{ m }(m+1)^22^{-p} \leq\sum\limits_{p=[\frac{m}{2}]+1}^{m}(m+1)^22^{-\frac{m}{2}}   \leq (m+1)^3 2^{-\frac{m}{2}}   \leq C.
\end{equation*}
Substituting \eqref{teccom} into \eqref{tecest2} yields
\begin{equation}\label{j12}
	\mathcal J_{1,2}\leq   C \eps_0^2\frac{C_{*}^{m}(m!)^\tau}{(m+1)^2}.
\end{equation}
By a similar and slightly modified argument we obtain 
 \begin{equation*}
 \begin{aligned}
 	&\mathcal J_{1,1}=\int_{\mathbb{Z}^3}\left(\int_{1}^{T}\big|\big(\hat{\mathcal T}(\hat{f}, \widehat{H^{m}f}, \mu^{\frac12}),  (v_j v_l-1)\mu^{\frac12}\big)_{L^2_v}\big|dt\right)^{1\over 2}d\Sigma(k)\\
 	&\leq  C    \Big(\int_{\mathbb Z^3} \sup\limits_{1\leq t\leq T}\norm{\hat{f}(t,k)}_{L^2_v}d\Sigma(k)\Big)\int_{\mathbb{Z}^3}\Big[\int_{1}^{T}\normm{\widehat{H^{m}f}(t,k)} ^{2}dt \Big]^{1\over2}d\Sigma(k) \\
 	&\leq  C \eps_0   \int_{\mathbb{Z}^3}\bigg[ \int_{1}^{T} \normm{\widehat{H^{m}f}  (t,k)}^2     dt \bigg]^{1\over2}d\Sigma(k).
 	\end{aligned}
 \end{equation*}
Substitute the above estimate and \eqref{j12} into \eqref{j1} yields  that
\begin{equation*}\label{uppj1}
	\mathcal J_1\leq  C\eps_0  \int_{\mathbb{Z}^3}\bigg[ \int_{1}^{T} \normm{\widehat{H^{m}f}  (t,k)}^2     dt \bigg]^{1\over2}d\Sigma(k)+  C \eps_0^2\frac{C_{*}^{m}(m!)^\tau}{(m+1)^2}.
\end{equation*}

 \smallskip
\noindent
\underline{\it Estimate on $\mathcal J_2$}.  Recall  $\mathcal{J}_2$ is given in \eqref{J1j2j3}.  A similar argument as in \eqref{tecest1} and \eqref{tecest2} gives
\begin{equation}\label{j2+}
\begin{aligned}
\mathcal{J}_2 \leq&  
  C \sum\limits_{n=1}^{m-1 }\sum\limits_{p=0}^{ n}\binom{m}{n}\binom{n}{p}2^p p!\int_{\mathbb{Z}^3}\sup\limits_{1\leq t\leq T} \norm{\widehat{H^{m-n}f}(t,k)}_{L^2_v}d\Sigma(k)\\&\qquad\qquad\qquad\qquad\quad\qquad\qquad\quad\times\int_{\mathbb{Z}^3}\left(\int_{1}^{T}\normm{ \widehat{H^{n-p}f}(t,k)}^{2}dt\right)^{1\over2}d\Sigma(k).
	\end{aligned}
\end{equation}
Using assumption \eqref{indassum} and repeating the computation in \eqref{teccom}, we find that for any $1 \leq n \leq m-1$,
\begin{multline*}
		 \sum\limits_{p=0}^{ n }\binom{n}{p}2^p p!\int_{\mathbb{Z}^3}\left(\int_{1}^{T}\normm{\widehat{H^{n-p}f}(t,k)}^{2}dt\right)^{\frac12}d\Sigma(k)\\
		  \leq  \eps_0 C_*^{n} (n!)^\tau \sum\limits_{p=0}^{ n }  \frac{2^{-p} }{(n-p+1)^2}
		\leq  C\eps_0\frac{C_{*}^{n}(n!)^\tau }{(n+1)^2}.
	\end{multline*}
Substituting this into \eqref{j2+} and using \eqref{indassum} again, we compute
\begin{equation*}
	\begin{aligned}
		&\sum\limits_{n=1}^{  m-1 }\sum\limits_{p=0}^{ n}\binom{m}{n}\binom{n}{p}2^p p!\int_{\mathbb{Z}^3}\sup\limits_{1\leq t\leq T} \norm{\widehat{H^{m-n}f}(t,k)}_{L^2_v}d\Sigma(k)\\&\qquad\qquad\qquad\qquad\quad\qquad\qquad\quad\times\int_{\mathbb{Z}^3}\left(\int_{1}^{T}\normm{ \widehat{H^{n-p}f}(t,k)}^{2}dt\right)^{1\over2}d\Sigma(k)\\
		&\leq C\eps_0 \sum\limits_{n=1}^{ m-1 }  \frac{m!}{n!(m-n)!} \frac{C_{*}^{n}(n!)^\tau }{(n+1)^2}\int_{\mathbb{Z}^3}\sup\limits_{1\leq t\leq T} \norm{\widehat{H^{m-n}f}(t,k)}_{L^2_v}d\Sigma(k)  \\
		& \leq   C\eps_0^2 \sum\limits_{n=1}^{m-1 }\frac{m!}{n!(m-n)!}\frac{C_{*}^{n}(n!)^\tau}{(n+1)^2} \frac{C_{*}^{m-n}[(m-n)!]^\tau}{(m-n+1)^2}\\
		&\leq    C\eps_0^2 C_{*}^{m}  \sum\limits_{n=1}^{  m-1 } \frac{m! (n!)^{\tau-1}[(m-n )!]^{\tau-1}}{(m-n+1)^2 (n+1)^2}
		\leq   C\eps_0^2 \frac{C_{*}^{m}(m!)^\tau }{(m+1)^2} ,
	\end{aligned}
\end{equation*}
where the last inequality uses the facts that $p! q! \le (p+q)!$ and $\tau \ge 1$.  
This, together  with \eqref{j2+},  yields
\begin{equation*}\label{uppj2}
\mathcal{J}_2 \leq  	C  \eps_0^2\frac{C_{*}^{m} (m!)^\tau}{(m+1)^2}.
\end{equation*}

\smallskip
\noindent
\underline{\it Estimate on  $\mathcal J_3$}. It remains to estimate  $\mathcal J_3$  in \eqref{J1j2j3}.  Repeating the computations in \eqref{tecest1} and \eqref{tecest2}, we obtain
\begin{equation*}
\begin{aligned}
& \mathcal{J}_3  =\int_{\mathbb{Z}^3}\left(\int_{1}^{T}\big|\big (\hat{\mathcal T}(\widehat{H^{m}f},\ \hat{f}, \ \mu^{\frac12}), (v_jv_l-1)\mu^{\frac12}\big)_{L^2_v}\big|dt\right)^{\frac12}d\Sigma(k)\\
& \leq   C \Big( \int_{\mathbb{Z}^3} \sup\limits_{1\leq t\leq T}\norm{\widehat{H^{m}f}(t,k)}_{L^2_v}d\Sigma(k)\Big)  \int_{\mathbb{Z}^3}\Big[\int_{1}^{T}\normm{\hat f(t,k)} ^{2}dt \Big]^{1\over2}d\Sigma(k)\\
& \leq   C  \eps_0 \int_{\mathbb{Z}^3} \sup\limits_{1\leq t\leq T}\norm{\widehat{H^{m}f}(t,k)}_{L^2_v}d\Sigma(k).
\end{aligned}
\end{equation*}
Combining the bounds for $\mathcal{J}_1$, $\mathcal{J}_2$, and $\mathcal{J}_3$ with \eqref{uppboun}, we obtain assertion \eqref{4.2} in Lemma~\ref{lem: ma}.

It remains to prove assertion \eqref{4.2a}, which is a special case of \eqref{4.2a}. Recall the linear operator $\mathcal{L}$ defined in \eqref{colli}:
	\begin{equation}\label{linpart}
	\mathcal L f =\Gamma( \mu^{\frac12},f)+\Gamma(f,\mu^{\frac12}) =\mathcal T(\mu^{\frac12}, f, \mu^{\frac12})+\mathcal T(f,\mu^{\frac12}, \mu^{\frac12}).		
	\end{equation}
 By Leibniz's formula,
  \begin{equation}\label{lastlabe}
\begin{aligned}
[H^{m},\ \mathcal{L}  ] f=&\sum\limits_{n=1}^{m}\sum\limits_{p=0}^{n} \binom{m}{n} \binom{n}{p}\mathcal{T}(H^{n-p}\mu^{\frac12}, H^{m-n}f, H^{p}\mu^{\frac12})\\
&\qquad+\sum\limits_{n=1}^{  m }\sum\limits_{p=0}^{ n } \binom{m}{n} \binom{n}{p}\mathcal{T}(H^{m-n}f, H^{n-p}\mu^{\frac12}, H^{p}\mu^{\frac12})\\
 \stackrel{\rm def}{=}& R_1(f)+R_2(f).
   \end{aligned}
   \end{equation}
As in \eqref{uppboun}, we write
\begin{equation*}\label{r2est}
\begin{aligned}
&	\int_{\mathbb{Z}^3}\left(\int_{1}^{T}\big|\big(\mathscr{F}_x\big (R_2(f)\big ), (v_jv_l-1)\mu^{\frac12}\big)_{L^2_v}\big|^2dt\right)^{\frac12}d\Sigma(k)\\
& \leq \int_{\mathbb{Z}^3}\sum\limits_{n=1}^{m}\sum\limits_{p=0}^n\binom{m}{n}\binom{n}{p} \\
&\qquad\qquad\quad \times \bigg[\int_{1}^{T}\big|\big(\hat{\mathcal{T}}(\widehat{H^{m-n}f},  \partial_{v_1}^{n-p}\mu^{1\over2}, \partial_{v_1}^{p}\mu^{1\over2}),  (v_jv_l-1)\mu^{\frac12}\big)_{L^2_v}\big|^2 dt\bigg]^{1\over2}d\Sigma(k).
\end{aligned}
\end{equation*}
A direct verification shows that
\begin{align*}
\forall\ p\geq 0,\quad   \big|(1+|v|^{2+\gamma}-\Delta_v)  \partial_{v_1}^{p}  \mu^{\frac12}\big| \leq C
2^pp!  \mu^{\frac18}.
\end{align*}
This, with \eqref{mu}, enables to use  \eqref{trisole}  with $g=\mu^{\frac12}$ and $\omega=\partial_{v_1}^{p}\mu^{\frac12}$, to obtain
\begin{align*}
	&\sum\limits_{n=1}^{m}\sum\limits_{p=0}^n\binom{m}{n}\binom{n}{p} \big|\big(\hat{\mathcal{T}}(\widehat{H^{m-n}f},  \partial_{v_1}^{n-p}\mu^{1\over2}, \partial_{v_1}^{p}\mu^{1\over2}), (v_jv_l-1)\mu^{\frac12}\big)_{L^2_v}\big| \\
	&\leq C \sum\limits_{n=1}^{m}\sum\limits_{p=0}^n\binom{m}{n}\binom{n}{p} 2^pp! \times\big[ 2^{n-p}(n-p)! \big]  		 \norm{\widehat{H^{m-n}f}}_{L^2_v}           \\
	&\leq  C \sum \limits_{n=1}^{m}\frac{m!}{(m-n)!} (n+1)2^n\norm{\widehat{H^{m-n}f}}_{L^2_v}.    
	\end{align*}
Thus, 
\begin{equation*}
	\begin{aligned}
	& \int_{\mathbb{Z}^3}\left(\int_{1}^{T}\big|\big(\mathscr{F}_x\big (R_2(f)\big ), (v_jv_l-1)\mu^{\frac12}\big)_{L^2_v}\big|^2dt\right)^{\frac12}d\Sigma(k)	\\
	&\leq C \int_{\mathbb{Z}^3}\left(\int_{1}^{T}\Big| \sum \limits_{n=1}^{m}\frac{m!}{(m-n)!} (n+1)2^n\norm{\widehat{H^{m-n}f}}_{L^2_v}\Big|^2dt\right)^{\frac12}d\Sigma(k)\\
	&\leq  C \sum \limits_{n=1}^{m}\frac{m!}{(m-n)!} (n+1)2^n \int_{\mathbb{Z}^3}  \left(\int_{1}^{T} \norm{\widehat{H^{m-n}f}}_{L^2_v} ^2dt\right)^{\frac12}d\Sigma(k)\\
	&\leq  C \sum \limits_{n=1}^{m}\frac{m!}{(m-n)!} (n+1)2^n \frac{\eps_0 C_*^{m-n}[(m-n)!]^\tau}{(m-n+1)^2} \leq C \eps_0\frac{C_*^{ m-1 } (m!)^\tau}{(m+1)^2}. 
	\end{aligned}
\end{equation*}
Similarly, using \eqref{tretmate} instead of \eqref{trisole}, the same estimate holds for $R_1(f)$.   Thus   assertion \eqref{4.2a} in Lemma  \ref{lem: ma} follows by observing 
\begin{multline*}
	\int_{\mathbb{Z}^3}\left(\int_{1}^{T}\big|\big(\mathscr{F}_x([H^{m},\  \mathcal{L}]f), \ (v_jv_l-1)\mu^{\frac12}\big)_{L^2_v}\big|dt\right)^{\frac12}d\Sigma(k)\\\leq \sum_{r=1}^2 \int_{\mathbb{Z}^3}\left(\int_{1}^{T}\big|\big(\mathscr{F}_x\big (R_r(f)\big ), (v_jv_l-1)\mu^{\frac12}\big)_{L^2_v}\big|dt\right)^{\frac12}d\Sigma(k)
\end{multline*}
due to \eqref{lastlabe}.    This completes the proof. 
\end{proof}
 
 \begin{lemma} \label{ma}
 Under the hypothesis of Proposition \ref{p3},   there exists a constant $C>0$  independent of $T, m, C_*, \eps_0$, such that for each $1\leq j, l\leq 3$, we have 
 	\begin{equation*} \label{r1}
 	\begin{aligned}
 	&\int_{\mathbb Z^3}\bigg(\int_1^{T}|\Theta_{jl}(\widehat{G}_m)|^2dt\bigg)^{\frac12}d\Sigma(k)+\int_{\mathbb Z^3}\bigg(\int_1^{T}|\Lambda_j (	\widehat{G}_m)|^2 dt\bigg)^{\frac12}d\Sigma(k)\\
 	& \leq  C	\int_{\mathbb{Z}^3}\bigg( \int_1^{T}	\normm{ \{\mathbf{I}-\mathbf{P} \}\widehat{H^mf}}^2dt \bigg)^{\frac12}d\Sigma(k)+
 	 C  \eps_0 \int_{\mathbb{Z}^3} \sup\limits_{1\leq t\leq T}\norm{\widehat{H^{m}f}}_{L^2_v}d\Sigma(k)\\
 	&\quad+C \eps_0  \int_{\mathbb{Z}^3}\bigg[ \int_{1}^{T} \normm{\widehat{H^{m}f}  (t,k)}^2     dt \bigg]^{1\over2}d\Sigma(k)
 	+C  \big(\eps_0  +C_*^{-1}\big)\frac{  \eps_0C_*^m  (m!)^{\tau}}{(m+1)^2}.
 	\end{aligned}
 	\end{equation*}
 	Recall $G_m$ is given in \eqref{defg}. 
 \end{lemma}
 
 \begin{proof}
 		Recall that 
 	\begin{eqnarray*}
 		G_m=\mathcal{L}\{\mathbf{I}-\mathbf{P} \}H^mf+
 		H^m\Gamma(f,f)-[H^m,\mathcal{L}]f.
 	\end{eqnarray*}
 	Thus, in view of \eqref{thetalamb}, we have
 	\begin{equation}\label{thgm}
 	\begin{aligned}
 		&|\Theta_{jl}(\widehat{G}_m)|\leq   \big| \big(\mathcal{L}\{\mathbf{I}-\mathbf{P} \}\widehat{H^mf}, (v_j v_l-1)\mu^{\frac12}\big)_{L_v^2}\big | 
 		\\
 		&\quad+\Big| \Big(\mathscr{F}_x\big (H^m \Gamma(f,f)\big),  (v_j v_l-1)\mu^{\frac12}\Big)_{L_v^2}\Big |+\Big| \Big(\mathscr{F}_x([H^m,\mathcal{L}]f), (v_j v_l-1)\mu^{\frac12}\Big)_{L_v^2}\Big|.
 		\end{aligned}
 	\end{equation}
Using \eqref{linpart} and applying \eqref{trin}, we obtain for each $1 \leq j,l \leq 3$:
 		\begin{equation*}
 		\big| \big(\mathcal{L}\{\mathbf{I}-\mathbf{P} \}\widehat{H^mf}, (v_j v_l-1)\mu^{\frac12}\big)_{L_v^2}\big |
 		\leq C	\normm{  \{\mathbf{I}-\mathbf{P} \}\widehat{H^mf}},
 		\end{equation*}
and consequently,
 			\begin{equation*} \label{4.1}
 			\begin{aligned}
 			&\int_{\mathbb{Z}^3} \bigg[ \int_1^{T}\big| \big(\mathcal{L}\{\mathbf{I}-\mathbf{P} \}\widehat{H^mf}, (v_j v_l-1)\mu^{\frac12}\big)_{L_v^2}\big |^2 dt\bigg]^{1\over2}d\Sigma(k)\\
 			&	\leq C	\int_{\mathbb{Z}^3}\bigg( \int_1^{T}	\normm{  \{\mathbf{I}-\mathbf{P} \}\widehat{H^mf}(t,k)}^2dt \bigg)^{1\over2}d\Sigma(k).
 			\end{aligned}
 			\end{equation*}
 	Combining this with estimates \eqref{4.2} and \eqref{4.2a} in Lemma~\ref{lem: ma} and \eqref{thgm}, we deduce that
 \begin{equation} \label{r2}
 \begin{aligned}
 &\int_{\mathbb Z^3}\bigg(\int_1^{T}|\Theta_{jl}(\widehat{G}_m)|^2dt\bigg)^{\frac12}d\Sigma(k)\\
 &\leq C	\int_{\mathbb{Z}^3}\bigg( \int_1^{T}	\normm{ \{\mathbf{I}-\mathbf{P} \}\widehat{H^mf}}^2dt \bigg)^{\frac12}d\Sigma(k)+
 C\eps_0 \int_{\mathbb{Z}^3} \sup_{1\leq t\leq T}\norm{\widehat{H^mf}}_{L^2_v}d\Sigma(k)\\
 &\quad +C\eps_0  \int_{\mathbb{Z}^3} \bigg( \int_1^{T} \normm{\widehat{H^mf}}^2dt  \bigg)^{\frac12}d\Sigma(k)
 +C  \big(\eps_0+C_*^{-1}\big)  \frac{ \eps_0C_*^m  (m!)^{\tau}}{(m+1)^2}.
 \end{aligned}
 \end{equation}
A similar upper bound to \eqref{r2} holds for the term \begin{equation*}\int_{\mathbb Z^3}\Big(\int_1^{T}|\Lambda_j (	\widehat{G}_m)|^2dt\Big)^{\frac12}d\Sigma(k),\end{equation*}
which can be established by analogous arguments. This completes the proof of Lemma~\ref{ma}.
 \end{proof}

 \begin{lemma} \label{rem4}
  Under the hypothesis of Proposition \ref{p3},   there exists a constant $C>0$  independent of $T, m, C_*, \eps_0$, such that  
 	\begin{eqnarray*}
 		\begin{aligned}
 		&\int_{\mathbb Z^3}\bigg(\int_1^{T}\Big|\int_{\mathbb R^3}		 \phi(v) \widehat{G}_m dv \Big|^2dt\bigg)^{\frac12}d\Sigma(k)\\
 		&\leq C\eps_0 \int_{\mathbb{Z}^3} \sup_{1\leq t\leq T}\norm{\widehat{H^mf}}_{L^2_v}d\Sigma(k)  +C\eps_0  \int_{\mathbb{Z}^3} \bigg( \int_1^{T} \normm{\widehat{H^mf}}^2dt  \bigg)^{\frac12}d\Sigma(k)\\
 &\quad+C  \big(\eps_0+C_*^{-1}\big)  \frac{ \eps_0C_*^m  (m!)^{\tau}}{(m+1)^2},
 	\end{aligned}
 	\end{eqnarray*}
 recalling 	$\phi(v)$ denotes any one of the functions $\mu^{\frac12},~v\mu^{\frac12}$	and $\frac{1}{6}(|v|^2-3)\mu^{\frac12}.$
 \end{lemma}
 
 \begin{proof}
 Recall $G_m=  \mathcal{L}\{\mathbf{I}-\mathbf{P} \}H^mf+ H^m \Gamma(f,f)- [ H^{m}, \ \mathcal{L}] f.$ 
Note  $\phi\in {\rm ker} \mathcal L,$ and thus
\begin{equation*}
	\int_{\mathbb R^3}		 \phi(v) \mathcal{L}\{\mathbf{I}-\mathbf{P} \}H^mf dv=0.
\end{equation*}
Therefore,
 \begin{equation*}
	\begin{aligned}
 		\Big|\int_{\mathbb R^3}		 \phi(v) \widehat{G}_m dv \Big| \leq   
 		 \big| \big(\mathscr{F}_x\big (H^m \Gamma(f,f)\big), \   \phi(v)\big)_{L_v^2}\big | 
 		 +\big| \big(\mathscr{F}_x([H^m,\mathcal{L}]f), \ \phi(v)\big)_{L_v^2}\big|.
 		\end{aligned}
\end{equation*}
This allows us to follow the same argument as in the proof of Lemma \ref{lem: ma} to obtain the assertion in Lemma \ref{rem4}.
The proof is thus completed.
 \end{proof}

\begin{proof}[Proof of Proposition \ref{p3}]
	We split 
	\begin{eqnarray} \label{4.14}
	\begin{aligned}
	&\int_{\mathbb Z^3}\bigg(\int_1^{T}\big(|\widehat{ a_m }(t,k)|^2+
	|\widehat{b_m}(t,k)|^2+|\widehat{ c_m }(t,k)|^2\big)dt\bigg)^{\frac12}d\Sigma(k)=I_1+I_2 
	\end{aligned}
	\end{eqnarray}
	with
	\begin{equation*}
		\left\{
		\begin{aligned}
		I_1&=\bigg(\int_1^{T}\big( |\widehat{ a_m }(t,0) |^2+
	 |\widehat{b_m}(t,0) |^2+ |\widehat{ c_m }(t,0) |^2\big)dt\bigg)^{\frac12} \\
	I_2&=\int_{k\in\mathbb Z^3\setminus{\{0\}}}\bigg(\int_1^{T}\big( |\widehat{ a_m }(t,k) |^2+
	 |\widehat{b_m}(t,k) |^2+ |\widehat{ c_m }(t,k) |^2\big)dt\bigg)^{\frac12}d\Sigma(k).	
		\end{aligned}
\right.
	\end{equation*}
We proceed to estimate $I_1$ and $I_2$ as follows.

 \underline{\it Estimate on $I_1$.}	
It follows from \eqref{mac} and \eqref{vecM}  that
\begin{equation*}
\begin{aligned}
\widehat{ a_m }(t,0)
=\int_{\mathbb R^3}\mu^{\frac12}\partial^m_{v_1}\hat{f}(t,0,v) dv 
= (-1)^m \int_{\mathbb R^3} (\partial^m_{v_1}\mu^{\frac12})\hat{f}(t,0,v) dv,
\end{aligned}
\end{equation*}
which with \eqref{mu} yields that 
\begin{equation*}
|\widehat{ a_m }(t,0)|\leq \norm{\partial^m_{v_1}\mu^{\frac12}}_{L^2_v}\norm{\hat{f}(t,0)}_{L^2_v}
\leq C 2^{m} m! \norm{\hat{f}(t,0)}_{L^2_v}.
\end{equation*}
On the other hand, by \eqref{longtime},  it holds that 
\begin{equation*}
\norm{\hat{f}(t,0)}_{L^2_v}\leq  \int_{\mathbb Z^3}\norm{\hat f(t,k)}_{ L^2_v}d\Sigma(k) \leq C e^{-A t}\norm{f_0}_{L^1_k L^2_v}
\end{equation*}
for some constant $A>0$.  Combining these estimates and using the smallness assumption \eqref{smallass}, and recalling $\eps_0>\epsilon>0$     in   Proposition \ref{prp:ve}, we obtain 
\begin{equation*}
|\widehat{ a_m }(t,0)|
 \leq C 2^{m} m!\, e^{-A t}\norm{f_0}_{L^1_k L^2_v} \leq C\epsilon 2^{m+1} e^{-A t}  m! \leq C\eps_0 2^{m+1} e^{-A t}  m! ,
\end{equation*}
and similarly for $|\widehat{b_m}(t,0)|$ and $|\widehat{ c_m }(t,0)|.$   
Therefore, for any $m\geq 1,$  
	\begin{eqnarray} \label{4.15}
		\begin{aligned}
			I_1=&\bigg(\int_1^{T}\big( |\widehat{ a_m }(t,0) |^2+
			 |\widehat{b_m}(t,0) |^2+ |\widehat{ c_m }(t,0) |^2\big)dt\bigg)^{\frac12} \\
			\leq & C\eps_0 2^{m+1}(m!)^{\tau}  \leq
			C\eps_0\frac{ C_*^{m-1}~ (m!)^{\tau}}{(m+1)^2},
		\end{aligned}
	\end{eqnarray}
	provided  $C_*\geq 4$. 
	
	 \underline{\it Estimate on $I_2$.} For each multi-index  $k\neq 0,$  using the fact that 
  $\frac{|k|^2}{1+|k|^2}\geq \frac{1}{2}$ and Lemma \ref{l5.2a}, we obtain 
	\begin{align*}
			&\frac12 \int_1^{T}\big( |\widehat{ a_m }(t,k) |^2+
			 |\widehat{b_m} (t,k)|^2+ |\widehat{ c_m }(t,k) |^2\big)dt \\
			&\leq 2 \sup_{1\leq t\leq T} \big| \mathcal{K}_m(\hat{ f}(t,k))\big| +C
			\int_1^{T}\normm{\{\mathbf{I}-\mathbf{P}\}	\widehat{H^mf}(t,k)}^2dt
			\\
			& +C\sum_{j,l}\int_1^{T}|\Theta_{jl}(	\widehat{G}_m)|^2dt+C\sum_{j}\int_1^{T}|\Lambda_j (	\widehat{G}_m)|^2dt+C\int_1^{T}\Big|\int 		\phi(v)\widehat{G}_m  dv \Big|^2dt,
		\end{align*}
	where  $\phi(v)$ denotes one of the functions $\mu^{\frac12},~v\mu^{\frac12}$	and $\frac{1}{6}(|v|^2-3)\mu^{\frac12}$. From the representation \eqref{if} of $\mathcal{K}_m(\hat{ f}(t,k))$,  one has that, for any $t>0$ and any $k\in\mathbb Z^3,$ 
	\begin{equation} \label{km}
		\begin{aligned}
			 |\mathcal{K}_m(\hat{ f}(t,k))|&\leq C \big (|\widehat{ a_m }(t,k)|^2+|\widehat{b_m}(t,k)|^2+|\widehat{ c_m }(t,k)|^2\big)\\
			&\qquad+	C\sum_{j,l}\big|\Theta_{jl}\big (\{\mathbf{I}-\mathbf{P}\}\widehat{H^m f}\big)\big|^2+C\sum_{j}\big |\Lambda_j\big (\{\mathbf{I}-\mathbf{P}\}\widehat{H^mf}
			\big)\big |^2\\
			& \leq C\norm{\mathbf{P} \widehat{H^m f}}^2_{L^2_v}+C\norm{\{\mathbf{I}-\mathbf{P}\} \widehat{H^m f}}^2_{L^2_v} \leq C \norm{ \widehat{H^m f}}^2_{L^2_v},
		\end{aligned}
	\end{equation}
	where the second inequality follows from \eqref{project} and \eqref{thetalamb}.
Combining the two estimates above, we get
\begin{eqnarray*} 
		\begin{aligned}
			 &I_2 =\int_{k\neq 0}\bigg(\int_1^{T}\big( |\widehat{ a_m }(t,k) |^2+
	 |\widehat{b_m}(t,k) |^2+ |\widehat{ c_m }(t,k) |^2\big)dt\bigg)^{\frac12}d\Sigma(k)\\
		&	\leq C	\int_{\mathbb{Z}^3}\bigg( \int_1^{T}	\normm{ \{\mathbf{I}-\mathbf{P} \}\widehat{H^mf}}^2dt \bigg)^{\frac12}d\Sigma(k) +C \int_{\mathbb{Z}^3} \sup_{1\leq t\leq T} \norm{\widehat{H^mf}}_{L^2_v}d\Sigma(k)\\
			&\   +C\sum_{j,l}\int_{\mathbb Z^3}\Big (\int_1^{T}|\Theta_{jl}(\widehat{G}_m)|^2dt\Big)^{\frac12}d\Sigma(k)+C\sum_j\int_{\mathbb Z^3}\Big(\int_1^{T}|\Lambda_j (	\widehat{G}_m)|^2dt\Big)^{\frac12}d\Sigma(k)\\
			&\ \, +C\int_{\mathbb Z^3}\bigg(\int_1^{T}\Big|\int_{\mathbb R^3}		\phi(v) \widehat{G}_m dv \Big|^2dt\bigg)^{\frac12}d\Sigma(k),
		\end{aligned}
	\end{eqnarray*}
which, with 
  Lemma \ref{ma} and  \ref{rem4}, yields that 
\begin{equation*}  
\begin{aligned}
 I_2 
\leq & C	\int_{\mathbb{Z}^3}\bigg( \int_1^{T}	\normm{ \{\mathbf{I}-\mathbf{P} \}\widehat{H^mf}}^2dt \bigg)^{\frac12}d\Sigma(k)+C \int_{\mathbb{Z}^3} \sup_{1\leq t\leq T} \norm{\widehat{H^mf}}_{L^2_v}d\Sigma(k)\\
&+  C\eps_0  \int_{\mathbb{Z}^3} \bigg( \int_1^{T} \normm{\widehat{H^mf}}^2dt  \bigg)^{\frac12}d\Sigma(k)+C  \big(\eps_0+C_*^{-1}\big)  \frac{ \eps_0C_*^m  (m!)^{\tau}}{(m+1)^2}.
\end{aligned}
\end{equation*}
Substituting the above estimate and estimate  \eqref{4.15}   into \eqref{4.14} yields  the assertion in Proposition \ref{p3}.   
This completes the proof.
\end{proof}

\begin{proof}
	[Sketch of the proof of Proposition \ref{+p3+}] The proof is quite similar to that of Proposition \ref{p3}. In fact, applying $\partial_{v_1}^{m} $ to \eqref{3} implies 
	\begin{equation*}
		\begin{aligned}
	\inner{\partial_{t} +v\cdot\partial_{x}  -\mathcal{L}} \partial_{v_1}^{m}  f  = -m\partial_{x_1}\partial_{v_1}^{m-1}f+[ \partial_{v_1}^{m}, \ \mathcal{L}] f + \partial_{v_1}^{m}\Gamma(f, f). 
	\end{aligned}
	\end{equation*} 
 Performing a macro-micro decomposition in the above equation, we obtain an expression analogous to \eqref{5}:
 \begin{equation*}
 	 \partial_t\mathbf{P}\partial_{v_1}^mf+v\cdot\partial_{x}\mathbf{P}\partial_{v_1}^mf=-
 \partial_t\{\mathbf{I}-\mathbf{P} \}\partial_{v_1}^mf+\tilde R_m+\tilde G_m,
 \end{equation*}
 with 
 \begin{eqnarray*}
 	 \tilde R_m=-v\cdot\partial_{x} \{\mathbf{I}-\mathbf{P} \}\partial_{v_1}^mf 
 \end{eqnarray*}
 and
 \begin{equation*}
 	\tilde G_m= \mathcal{L}\{\mathbf{I}-\mathbf{P} \}\partial_{v_1}^mf+ \partial_{v_1}^m \Gamma(f,f)+ [ \partial_{v_1}^{m}, \ \mathcal{L}] f-m\partial_{x_1}\partial_{v_1}^{m-1}f. 
 \end{equation*}
 Compared with 
$G_m$ defined in \eqref{defg},  the only new term to handle in 
$\tilde G_m$ is the last one. To treat it, we rewrite 
\begin{equation*}
	m\partial_{x_1}\partial_{v_1}^{m-1}f=t^{-1} m H\partial_{v_1}^{m-1}f-t^{-1} m  \partial_{v_1}^{m}f,
\end{equation*}
recalling  $H=t\partial_{x_1}+\partial_{v_1}$ is given in \eqref{vecM}.  As a result, 
we use  \eqref{thetalamb} to compute 
\begin{equation*}
\begin{aligned}
	&\big | \Theta_{jl}\big (  \mathscr F_x (m \partial_{x_1}\partial_{v_1}^{m-1}f)\big )\big|\\
	&\leq  t^{-1} m \big |  \big(  \partial_{v_1}^{m-1}\widehat{H f},\ (v_j v_l-1)\mu^{\frac12} \big)_{L_v^2}\big|+t^{-1} m \big |  \big(  \partial_{v_1}^{m}\hat{f},\ (v_j v_l-1)\mu^{\frac12} \big)_{L_v^2}\big|\\
	&\leq  t^{-1} m \big |  \big(   \widehat{H f},\ \partial_{v_1}^{m-1}\big[ (v_j v_l-1)\mu^{\frac12}\big]  \big)_{L_v^2}\big|+t^{-1} m \big |  \big(   \hat{f} ,\ \partial_{v_1}^{m}\big[ (v_j v_l-1)\mu^{\frac12}\big]  \big)_{L_v^2}\big|\\
	&\leq  C t^{-1} 3^{m} m! \Big(   \norm{\widehat {Hf }}_{L_v^2}+ \norm{\hat {f}}_{L_v^2}\Big)\leq  C t^{-1} 3^{m} m! \Big(   \normm{\widehat {Hf }}+ \normm{\hat {f}}\Big),
	\end{aligned}
\end{equation*}
the last line using \eqref{mu}. Thus
\begin{equation*}
	\begin{aligned}
		&\int_{\mathbb Z^3}\bigg(\int_1^{T}\big | \Theta_{jl}\big (  \mathscr F_x (m \partial_{x_1}\partial_{v_1}^{m-1}f)\big )\big|^2dt\bigg)^{\frac12}d\Sigma(k) \\
		&\leq C 3^{m} m! \int_{\mathbb Z^3}\bigg(\int_1^{T}\normm{\widehat{Hf}}^2dt\bigg)^{\frac12}d\Sigma(k)+ C 3^{m} m! \int_{\mathbb Z^3}\bigg(\int_1^{T}\normm{\hat{f}}^2dt\bigg)^{\frac12}d\Sigma(k)\\
		&\leq C 3^{m} m! \int_{\mathbb Z^3}\bigg(\int_1^{T}\normm{\widehat{Hf}}^2dt\bigg)^{\frac12}d\Sigma(k)+ C   \frac{  \eps_0 L^{m-1}  (m!)^{\tau}}{(m+1)^2},
	\end{aligned}
\end{equation*}
 where the last inequality uses \eqref{supp+} and the condition $L\geq 4.$   Analogous estimate   holds for 
 \begin{equation*}
 	\int_{\mathbb Z^3}\bigg(\int_1^{T}\big | \Lambda_{j}\big (  \mathscr F_x (m \partial_{x_1}\partial_{v_1}^{m-1}f)\big )\big|^2dt\bigg)^{\frac12}d\Sigma(k).
 \end{equation*}
With these bounds, we may now follow the same line of argument as in the proof of Proposition \ref{p3} to obtain  the assertion in Proposition \ref{+p3+}. We omit the details for brevity. The proof of  Proposition \ref{+p3+}  is thus completed.  
\end{proof}

In particular, for $m=0$, equation \eqref{5} reduces to 
\begin{equation*}
\partial_t\mathbf{P}f+v\cdot\partial_{x}\mathbf{P} f=-
 \partial_t\{\mathbf{I}-\mathbf{P} \} f -v\cdot\partial_{x} \{\mathbf{I}-\mathbf{P} \}f+\mathcal{L}\{\mathbf{I}-\mathbf{P} \}f+   \Gamma(f,f).
\end{equation*}
Recall  that 
\begin{equation*}
    \mathbf{P}f=\big \{a(t,x)+b(t,x)\cdot v+c(t,x)(|v|^2-3)\big \}\mu^{\frac12}.
\end{equation*}
 Then repeating the proof of Proposition \ref{p3}
we have the following  result. 

\begin{proposition}\label{macmic}
   Assume the hypothesis of Theorem \ref{gxt} holds.
    Then there exists a constant $C$, depending only on the constant $C_0,C_1$ in Section \ref{sec:prelim}, 
 such that, for any $k\in\mathbb Z^3,$  
\begin{equation}\label{emm}
	\begin{aligned}
		&\frac14\Big( |\hat { a}(t,k) |^2+
		 |\hat{b}(t,k) |^2+ |\hat{ c }(t,k) |^2  \Big)\\
		& \quad \leq -\frac{d}{dt}{\rm Re} \, \mathcal{K}_0(\hat f)+  C	 	\normm{ \{\mathbf{I}-\mathbf{P} \}\hat{f}}^2    +C \Big|\int_{\mathbb Z^3}   \norm{\hat{f}(t,k-\ell)}_{L_v^2} \normm{\hat{f}(t,\ell)} d\Sigma(\ell)\Big|^2,
	\end{aligned}
\end{equation}
where $ a ,b=(b_1,b_2,b_3)$ and $ c$ are defined in \eqref{eqorp1} and
\begin{equation*}
    \begin{aligned}
 	 \mathcal{K}_0  \big  (\hat{ f}\big )=& \frac{1}{1+|k|^2}\!\sum_{j,l=1}^3\!\Big(i k_j\widehat{ b_{l}} +ik_l\widehat{b_{j}}\ \big|\ \Theta_{jl}\big(\{\mathbf{I}-\mathbf{P}\}\hat{f}\big) + 2\delta_{jl}\widehat{ c }\Big)\\
 	& +\frac{\rho_0}{1+|k|^2}\sum_{j=1}^3 \Big(	
 	 \Lambda_j \big(\{\mathbf{I}-\mathbf{P}\}\hat{f}\big)\ \big |\ i k_j\hat{ c } \Big)  + \frac{1}{1+|k|^2}\sum_{j=1}^3 \Big(	\widehat{b_{j}}\  \big | \ i k_j \hat { a } \Big).
 	\end{aligned}
\end{equation*}
\end{proposition}

\begin{proof}
 The validity of estimate \eqref{emm} for    $k=0$ follows from the fact that 
 \begin{equation*}
     \hat { a}(t,0)=\widehat { b_j}(t,0)=\hat { c}(t,0)=0,\quad 1\leq j\leq 3 
 \end{equation*}
 and
 \begin{equation*}
    \mathcal{K}_0  \big  (\hat{ f}(t,0)\big )=0.
    \end{equation*}
 For $k\neq 0,$   estimate \eqref{emm}  just follows from a similar argument to that of Lemmas \ref{l5.2a}, \ref{ma} and  \ref{rem4} as well as \eqref{upptrifour}.    
\end{proof}

  \section{Global-in-time estimate on radius}\label{sec:global}
 In this section, we prove Theorem \ref{gxt} and establish the uniform-in-time radius estimate. The proof is divided into two parts, corresponding to estimates \eqref{opa} and \eqref{opg}.

 \subsection{Radius estimate in sharp regularity space}

This part proves \eqref{opa} in Theorem \ref{gxt}, relying on the following Propositions \ref{theorem} and \ref{prp:v}.

  \begin{proposition}\label{theorem} Let 
$\tau=\max\big\{1,\frac{1}{2s}\big\}.$  Suppose the hypothesis of Theorem \ref{gxt} is fulfilled.      Then  there exist  two sufficiently small constants $0<\epsilon<\eps_0$ and a large constant $C_*\geq 4$, with $C_*$ depending only on  the numbers $C_0,C_1$ in  Section \ref{sec:prelim},  such that if the initial datum satisfies that $\norm{f_0}_{L_k^1L_v^2}\leq \epsilon$,  
   then
  the  estimate
 	\begin{multline}\label{k}
 	\int_{\mathbb{Z}^3}\sup\limits_{ t\geq 1}\norm{\widehat{H^{m}f}(t,k)}_{L^2_v}d\Sigma(k)\\
 	+ \int_{\mathbb{Z}^3}\left(\int_{1}^{+\infty}\normm{  \widehat{H^{m}f}(t,k)}^{2}dt\right)^{1\over2}d\Sigma(k) \leq \frac{\eps_0 C_*^{m}  (m!)^\tau  }{(m+1)^2}
 	\end{multline}
 	holds true for any $m\in\mathbb Z_+$, where $H=t\partial_{x_1}+\partial_{v_1}.$
 	Moreover, estimate \eqref{k} remains valid if $H$ is replaced by
 	$$  t  \partial_{x_i}+ \partial_{v_i},\quad i=2\  \textrm{or} \ 3.$$
\end{proposition}

\begin{proof}
 We proceed by induction on $m$. The case $m = 0$ follows from \eqref{v2}. By \eqref{sgv}, 
  \begin{equation*}
 \forall \ j\geq 0, \quad  	\int_{\mathbb{Z}^3}\sup\limits_{1\leq t \leq T}\norm{\widehat{H^{j}f}}_{L^2_v}d\Sigma(k) + \int_{\mathbb{Z}^3}\left(\int_{1}^{T}\normm{  \widehat{H^{j}f}}^{2}dt\right)^{\frac12}d\Sigma(k)<+\infty.
  \end{equation*}
This  ensures rigorous  rather than formal computations in the subsequent argument
	   when performing estimates involving the term $H^j f $ for each $j\in\mathbb Z_+.$ 
  
Fix $m \geq 1$ and assume that there exists a constant $C_*\geq 4$ such that  for all $j \leq m-1$ and any $T>1$, the following estimate holds:
\begin{multline}\label{suppose+}
\int_{\mathbb{Z}^3}\sup\limits_{1\leq t \leq T}\norm{\widehat{H^{j}f}(t,k)}_{L^2_v}d\Sigma(k)\\+ \int_{\mathbb{Z}^3}\left(\int_{1}^{T}\normm{  \widehat{H^{j}f}(t,k)}^{2}dt\right)^{\frac12}d\Sigma(k) \leq \frac{\eps_0C_*^{j}  (j!)^\tau }{(j+1)^2}.
\end{multline}
 We will show that \eqref{suppose+} remains valid for $j = m$ when $C_*$ is sufficiently large. Throughout, $C$ denotes a generic constant depending only on $C_0, C_1$ in Section \ref{sec:prelim}, independent of $T$ and $m$.

Taking  Fourier transform  with respect to $x$  in  equation \eqref{hmf} and  considering the real part after taking the  inner product of  $L_v^2$ with $ \widehat{H^{m} f}$, we obtain
\begin{multline*}
  \frac{1}{2}\frac{d}{dt}\norm{\widehat{H^{m}f}}_{L^2_v}^2 -\big(\mathcal{L}\widehat{H^{m}f},\  \widehat{H^{m}f}\big)_{L^2_v}
  \\  \leq   \big|\big(\mathscr{F}_x([ H^{m},\ \mathcal{L}]f), \widehat{H^{m}f}\big)_{L^2_v}\big|
 +\big|\big(\mathscr{F}_x(H^{m}\Gamma(f,f)), \ \widehat{H^{m}f}\big)_{L^2_v}\big|.
\end{multline*}
Combining this with \eqref{rela} and recalling the constant $C_1$ therein, we have
\begin{equation*}\label{enestimate}
\begin{aligned}
&\frac{1}{2}\frac{d}{dt}\norm{\widehat{H^{m}f}}_{L^2_v}^2+C_1\normm{ \{\mathbf{I}-\mathbf{P} \} \widehat{H^{m}f}}^{2}\\
&\qquad\leq   \big|\big(\mathscr{F}_x([ H^{m},\ \mathcal{L}]f), \widehat{H^{m}f}\big)_{L^2_v}\big|+\big|\big(\mathscr{F}_x(H^{m}\Gamma(f,f)), \ \widehat{H^{m}f}\big)_{L^2_v}\big|.
\end{aligned}
\end{equation*}
This implies  that 
\begin{equation}\label{rt}
\begin{aligned}
& \int_{\mathbb Z^3} \sup\limits_{1\leq t\leq T}\norm{\widehat{H^{m}f}(t,k)}_{L^2_v}d\Sigma(k)+ \sqrt{C_1} \int_{\mathbb Z^3}\Big( \int_{1}^{T}\normm{  \{\mathbf{I}-\mathbf{P} \} \widehat{H^{m}f}}^{2}dt\Big)^{\frac12}d\Sigma(k)
\\
& \leq   \int_{\mathbb Z^3}   \norm{\widehat{H^{m}f}(1,k)}_{L^2_v} d\Sigma(k)+   \int_{\mathbb{Z}^3}\left[\int_{1}^{T}\big|\big(\mathscr{F}_x([H^{m},   \mathcal{L}]f), \widehat{H^{m}f}\big)_{L^2_v}\big|dt\right]^{\frac12}d\Sigma(k)\\
&\quad+ \int_{\mathbb{Z}^3}\left(\int_{1}^{T}\big|\big(\mathscr{F}_x(H^{m}\Gamma(f,f)), \widehat{H^{m}f}\big)_{L^2_v}\big|dt\right)^{\frac12}d\Sigma(k).
\end{aligned}
\end{equation}
Using the inductive hypothesis \eqref{suppose+} and Lemma \ref{lem:comestimate}, we obtain that for any $\eps > 0$,
\begin{equation}\label{estimatecom}
\begin{aligned}
&\int_{\mathbb{Z}^3}\left(\int_{1}^{T}\big|\big(\mathscr{F}_x(H^{m}\Gamma(f,f)), \widehat{H^{m}f}\big)_{L^2_v}\big|dt\right)^{\frac12}d\Sigma(k)\\
&\qquad+\int_{\mathbb{Z}^3}\left(\int_{1}^{T}\big|\big(\mathscr{F}_x([H^{m},\  \mathcal{L}]f), \ \widehat{H^{m}f}\big)_{L^2_v}\big|dt\right)^{\frac12}d\Sigma(k)
\\ 
& \leq    C{\eps}^{-1}\eps_0   \int_{\mathbb{Z}^3}\sup\limits_{1\leq t\leq T}\norm{\widehat{H^{m}f}(t,k)}_{L^2_v}d\Sigma(k)\\
&\quad +    \inner{\eps+C{\eps}^{-1}\eps_0}  \int_{\mathbb{Z}^3}\Big (\int_{1}^{T}\normm{ \widehat{H^{m}f}}^{2}dt\Big )^{\frac12}d\Sigma(k) +C \eps^{-1}  \big(\eps_0+C_*^{-1}\big)  \frac{\eps_0C_*^{m} (m!)^{\tau}}{(m+1)^2}. 
\end{aligned}
\end{equation}
Substituting the above estimate into \eqref{rt} yields that, for any $\eps>0,$
\begin{equation}\label{00}
\begin{aligned}
&\int _{\mathbb{Z}^3}\sup\limits_{1\leq t\leq T}\norm{\widehat{H^{m}f}(t,k)}_{L^2_v}d\Sigma (k)+  \int _{\mathbb{Z}^3}\left(\int_{1}^{T} \normm{ \{\mathbf{I}-\mathbf{P} \} \widehat{H^{m}f}(t,k)}^{2}dt\right)^{\frac12}d\Sigma (k)
\\ & \leq C \int_{\mathbb Z^3}   \norm{\widehat{H^{m}f}(1,k)}_{L^2_v} d\Sigma(k)+  C{\eps}^{-1}\eps_0   \int_{\mathbb{Z}^3}\sup\limits_{1\leq t\leq T}\norm{\widehat{H^{m}f} }_{L^2_v}d\Sigma(k) \\
&\quad+    \inner{\eps+C{\eps}^{-1}\eps_0}  \int_{\mathbb{Z}^3}\Big (\int_{1}^{T}\normm{ \widehat{H^{m}f}}^{2}dt\Big )^{\frac12}d\Sigma(k) +C \eps^{-1}  \big(\eps_0+C_*^{-1}\big)  \frac{\eps_0C_*^{m} (m!)^{\tau}}{(m+1)^2}.
\end{aligned}
\end{equation}
On the other hand, for any $m \geq 1$,
		\begin{equation*} \label{5.4}
		\begin{aligned}
		 \normm{\mathbf{P} \widehat{H^{m}f} }^2 
	\leq C\big( |\widehat { a_m } |^2+
		 |\widehat{b_m} |^2+ |\widehat{ c_m } |^2\big), 
			\end{aligned}
		\end{equation*}	
with $ a_m ,b_m$ and $ c_m $ defined in \eqref{mac}. Applying Proposition \ref{p3}, we conclude
\begin{equation*}\label{micest}
\begin{aligned}
	& \int _{\mathbb{Z}^3}\left[\int_{1}^{T} \normm{  \mathbf{P} \widehat{H^{m}f}}^{2}dt\right]^{\frac12}d\Sigma (k) \leq C \int _{\mathbb{Z}^3}\left[\int_{1}^{T} (|\widehat { a_m } |^2+
		 |\widehat{b_m} |^2+ |\widehat{ c_m } |^2 )dt\right]^{\frac12}d\Sigma (k)\\
		 &\leq  C  \int_{\mathbb{Z}^3} \sup_{1\leq t\leq T}\norm{\widehat{H^mf}}_{L^2_v}d\Sigma(k)+  C	\int_{\mathbb{Z}^3}\bigg( \int_1^{T}	\normm{ \{\mathbf{I}-\mathbf{P} \}\widehat{H^mf}}^2dt \bigg)^{\frac12}d\Sigma(k)
 	 \\
 	&\quad+C\eps_0  \int_{\mathbb{Z}^3} \bigg( \int_1^{T} \normm{\widehat{H^mf}(t,k)}^2dt  \bigg)^{\frac12}d\Sigma(k)
 	+C\big(  \eps_0+C_*^{-1}\big)  \frac{ \eps_0C_*^m  (m!)^{\tau}}{(m+1)^2}\\ 
 	& \leq C \int_{\mathbb Z^3}   \norm{\widehat{H^{m}f}(1,k)}_{L^2_v} d\Sigma(k)+  C{\eps}^{-1}\eps_0   \int_{\mathbb{Z}^3}\sup\limits_{1\leq t\leq T}\norm{\widehat{H^{m}f} }_{L^2_v}d\Sigma(k) \\
&\quad+    \inner{\eps+C{\eps}^{-1}\eps_0}  \int_{\mathbb{Z}^3}\Big (\int_{1}^{T}\normm{ \widehat{H^{m}f}}^{2}dt\Big )^{\frac12}d\Sigma(k) +C \eps^{-1}  \big(\eps_0+C_*^{-1}\big)  \frac{\eps_0C_*^{m} (m!)^{\tau}}{(m+1)^2},
	\end{aligned}
\end{equation*}
where the last line uses \eqref{00}.  Combining this with \eqref{00} and using the fact 
\begin{align*}
\normm{\widehat{H^mf}}^2\leq \big(\normm{\{\mathbf{I}-\mathbf{P} \} \widehat{H^mf}}+\normm{ \mathbf{P}  \widehat{H^mf}}\big)^2 \leq 2\normm{\{\mathbf{I}-\mathbf{P} \} \widehat{H^mf}}^2+2\normm{ \mathbf{P}  \widehat{H^mf}}^2
 \end{align*}
we obtain that,  for any $0<\eps<1,$     
\begin{equation}\label{25m22}
	\begin{aligned}
& \int _{\mathbb{Z}^3}\sup\limits_{1\leq t\leq T}\norm{\widehat{H^{m}f}(t,k)}_{L^2_v}d\Sigma (k)+  \int _{\mathbb{Z}^3}\left(\int_{1}^{T} \normm{  \widehat{H^{m}f}(t,k)}^{2}dt\right)^{\frac12}d\Sigma (k)\\
  & \leq   C  \int_{\mathbb Z^3}   \norm{\widehat{H^{m}f}(1,k)}_{L^2_v} d\Sigma(k)+ C{\eps}^{-1}\eps_0   \int_{\mathbb{Z}^3}\sup\limits_{1\leq t\leq T}\norm{\widehat{H^{m}f} }_{L^2_v}d\Sigma(k) \\
&\quad+    \inner{\eps+C{\eps}^{-1}\eps_0}  \int_{\mathbb{Z}^3}\Big (\int_{1}^{T}\normm{ \widehat{H^{m}f}}^{2}dt\Big )^{\frac12}d\Sigma(k) +C \eps^{-1}  \big(\eps_0+C_*^{-1}\big)  \frac{\eps_0C_*^{m} (m!)^{\tau}}{(m+1)^2}.
\end{aligned}
\end{equation}
Using \eqref{pjm0} implies 
\begin{equation*}
	\int_{\mathbb Z^3}   \norm{\widehat{H^{m}f}(1,k)}_{L^2_v} d\Sigma(k)\leq \frac{1+6s}{2s}\frac{\eps_0N^{m} (m!)^{\tau}}{(m+1)^2}
\end{equation*}
where $N$ is the constant in \eqref{pjm0}.  Consequently, combining the two estimates above and choosing $\eps=\frac{1}{4}$ in \eqref{25m22}, we conclude 
\begin{align*}
   \int _{\mathbb{Z}^3}\sup\limits_{1\leq t\leq T}\norm{\widehat{H^{m}f}(t,k)}_{L^2_v}d\Sigma (k)+  \int _{\mathbb{Z}^3}\left(\int_{1}^{T} \normm{  \widehat{H^{m}f}}^{2}dt\right)^{\frac12}d\Sigma (k) \leq \frac{\eps_0C_*^{m} (m!)^{\tau}}{(m+1)^2},
\end{align*}
provided $\eps_0$ is small sufficiently and $C_*$  are large enough. 
 This proves \eqref{suppose+} for $j = m$, and hence for all $j \in \mathbb{Z}_+$. Since the constant  $C_*$ is independent of $T$, assertion \eqref{k} follows. The same argument applies to $t \partial_{x_i} + \partial_{v_i}$ for $i = 2$ or $3$. This completes the proof of Proposition \ref{theorem}. 
  \end{proof}

  \begin{proposition}\label{prp:v}
  Assume the hypothesis of Proposition \ref{theorem} holds. Then,   
 after enlarging the constant $C_*$ from Proposition \ref{theorem} if necessary,  the  estimate 
 	\begin{multline*} 
 	\int_{\mathbb{Z}^3}\sup\limits_{ t\geq 1}  \norm{\partial_{v_1}^m  \hat{f}(t,k)}_{L^2_v}d\Sigma(k)\\
 	+ \int_{\mathbb{Z}^3}\left(\int_{1}^{+\infty} \normm{  \partial_{v_1}^{m}\hat{f}(t,k)}^{2}dt\right)^{1\over2}d\Sigma(k) \leq \frac{\eps_0  C_*^{m}  ( m!)^\tau  }{(m+1)^2}
 	\end{multline*}
 	holds true for any $m\in\mathbb Z_+$.  Moreover, the same estimate   remains valid if   $\partial_{v_1}$  is replaced by
 	$ \partial_{v_2}$  or $ \partial_{v_3}.$
\end{proposition}

   \begin{proof}The proof is similar to that of Proposition \ref{theorem}, so we only sketch it and highlight the differences. Fix $m \geq 1$ and assume  that for all $j \leq m-1$ and any $T>1$,
\begin{multline}\label{vind}
\int_{\mathbb{Z}^3}\sup\limits_{1\leq t\leq T}\norm{ \partial_{v_1}^{j}\hat f (t,k)}_{L^2_v}d\Sigma(k)\\+ \int_{\mathbb{Z}^3}\left(\int_{1}^{T}\normm{   \partial_{v_1}^{j}\hat f (t,k)}^{2}dt\right)^{\frac12}d\Sigma(k) \leq \frac{\eps_0 C_*^{j}( j!)^\tau}{(j+1)^2}.
\end{multline}
where $C_*$ is the constant in \eqref{k} independent of $T.$  
We will prove the above estimate still holds true for $j=m.$

Observe $[\partial_t+v\cdot\partial_x, \partial_{v_1}^m]=-m\partial_{x_1}\partial_{v_1}^{m-1}.$ Similar to \eqref{rt} and \eqref{estimatecom}, we  conclude that  
\begin{equation} \label{v1}
	\begin{aligned}
& \int_{\mathbb Z^3} \sup\limits_{1\leq t\leq T}\norm{\partial_{v_1}^{m}\hat{f}(t,k)}_{L^2_v}d\Sigma(k)+  \int_{\mathbb Z^3}\Big( \int_{1}^{T}\normm{  \{\mathbf{I}-\mathbf{P} \} \partial_{v_1}^{m}\hat{f}(t,k)}^{2}dt\Big)^{\frac12}d\Sigma(k)
\\
& \leq  C \int_{\mathbb Z^3}   \norm{\partial_{v_1}^{m}\hat{f}(1,k)}_{L^2_v} d\Sigma(k) + C{\eps}^{-1}\eps_0   \int_{\mathbb{Z}^3}\sup\limits_{1\leq t\leq T}\norm{\partial_{v_1}^{m}\hat{f}(t,k)}_{L^2_v}d\Sigma(k)\\
&\quad +    \inner{\eps+C{\eps}^{-1}\eps_0}  \int_{\mathbb{Z}^3}\Big (\int_{1}^{T}\normm{ \partial_{v_1}^{m}\hat{f}}^{2}dt\Big )^{\frac12}d\Sigma(k) +C \eps^{-1}  \big(\eps_0+C_*^{-1}\big)  \frac{\eps_0C_*^{m} (m!)^{\tau}}{(m+1)^2}\\
&\quad +C\int_{\mathbb{Z}^3}\left(\int_{1}^{T}\big|\big(m \widehat{\partial_{x_1}\partial_{v_1}^{m-1}f}, \partial_{v_1}^{m}\hat{f}\big)_{L^2_v}\big|dt\right)^{\frac12}d\Sigma(k).
\end{aligned}
\end{equation}
We now estimate the last term on the right-hand side of \eqref{v1}. Using the identity $H=t\partial_{x_1}+\partial_{v_1}$ gives,  for any $t\geq 1,$
\begin{equation} \label{v2a}
\begin{aligned}
&\big|\big(m \widehat{\partial_{x_1}\partial_{v_1}^{m-1}f}, \partial_{v_1}^{m}\hat{f}\big)_{L^2_v}\big|\\
& \leq  t^{-1} m\big| \big(\widehat{H\partial_{v_1}^{m-1}f}, \partial_{v_1}^{m}\hat{f}\big)_{L^2_v}   \big|+t^{-1} m\big| \big(  \partial_{v_1}^{m}\hat f, \partial_{v_1}^{m}\hat{f}\big)_{L^2_v}   \big|\\
& \leq   \eps \norm{\widehat{H\partial_{v_1}^{m-1}f}}_{H_v^s}^2+\eps^{-1}m^2\norm{ \partial_{v_1}^{m}\hat f}_{H_v^{-s}}^2+ m\norm{\partial_{v_1}^m\hat f}_{L_v^2}^2.
\end{aligned}
\end{equation}
By the interpolation inequality \eqref{interp}, which implies
\begin{equation*}
	m\norm{\partial_{v_1}^m\hat f}_{L_v^2}^2 \leq  \eps  \norm{ \partial_{v_1}^{m}\hat f}_{H_v^{s}}^2+\eps^{-\frac{1-s}{s}}m^{2\tau }\norm{ \partial_{v_1}^{m-1}\hat f}_{H_v^{s}}^2,
\end{equation*}
and  
\begin{equation*}
	\begin{aligned}
		\eps^{-1} m^2\norm{ \partial_{v_1}^{m}\hat f}_{H_v^{-s}}^2 \leq 
		\left\{
		\begin{aligned} 
		&\eps  \norm{ \partial_{v_1}^{m}\hat f}_{H_v^{s}}^2+\eps^{1-\frac{1}{s}} m ^{ \frac{1}{s}}\norm{ \partial_{v_1}^{m-1}\hat f}_{H_v^{s}}^2\  \textrm { for }\  0<s< \frac12,\\
		&\eps^{-1}m^2  \norm{ \partial_{v_1}^{m-1}\hat f}_{H_v^{1-s}}^2\leq \eps^{-1} m^2 \norm{ \partial_{v_1}^{m-1}\hat f}_{H_v^{s}}^2 \  \textrm { for }\  \frac12\leq s<1,
		\end{aligned}
		\right. 
	\end{aligned}
\end{equation*}
where we have used the interpolation inequality \eqref{s} for the case of $0<s<\frac 12$.
Substituting these estimates into \eqref{v2a} and noting $\tau=\max\big\{\frac{1}{2s}, 1\big\},$ we obtain that, for any $0<\eps<1$,
\begin{equation*}
\begin{aligned}
	\big|\big(m \widehat{\partial_{x_1}\partial_{v_1}^{m-1}f}, \partial_{v_1}^{m}\hat{f}\big)_{L^2_v}\big|&\leq \eps   \norm{\widehat{H\partial_{v_1}^{m-1}f}}_{H_v^s}^2+2\eps  \norm{ \partial_{v_1}^{m}\hat f}_{H_v^{s}}^2+2\eps^{-\frac{1}{s}} m ^{ 2\tau}\norm{ \partial_{v_1}^{m-1}\hat f}_{H_v^{s}}^2 \\
	&\leq \eps   \norm{\widehat{H^{m}f}}_{H_v^s}^2+3\eps  \norm{ \partial_{v_1}^{m}\hat f}_{H_v^{s}}^2+2\eps^{-\frac{1}{s}} m ^{ 2\tau}\norm{ \partial_{v_1}^{m-1}\hat f}_{H_v^{s}}^2 \\
	&\leq \eps   \normm{\widehat{H^{m}f}}^2+3\eps  \normm{ \partial_{v_1}^{m}\hat f}^2+2\eps^{-\frac{1}{s}} m ^{ 2\tau}\normm{ \partial_{v_1}^{m-1}\hat f}^2, 
	\end{aligned}
\end{equation*} 
where the second inequality uses \eqref{fmn} and the last one follows from \eqref{+lowoftri}. This, along with \eqref{k}  and  the inductive assumption \eqref{vind},  yields 
\begin{equation*}
\begin{aligned}
&\int_{\mathbb{Z}^3}\left(\int_{1}^{T}\big|\big(m \widehat{\partial_{x_1}\partial_{v_1}^{m-1}f}, \partial_{v_1}^{m}\hat{f}\big)_{L^2_v}\big|dt\right)^{\frac12}d\Sigma(k)\\
&\leq  3 \eps \int _{\mathbb Z^3}\left(\int_{1}^{T} \normm{  \partial_{v_1}^{m}\hat f(t,k)}^{2}dt\right)^{\frac12}d\Sigma (k)+ \eps   \frac{\eps_0C_*^{m} (m!)^{\tau}}{(m+1)^2}+ \eps^{-\frac{1}{s}}C_*^{-1}\frac{\eps_0C_*^{m} (m!)^{\tau}}{(m+1)^2},
\end{aligned}
\end{equation*}
where $C_*$ is the constant in \eqref{k}.  Substituting the above estimate into \eqref{v1} and then repeating the argument after \eqref{25m22}, we conclude 
 that  
   \begin{equation*}
 	 \int _{\mathbb Z^3}\sup\limits_{1\leq t\leq T}\norm{ \partial_{v_1}^{m}\hat f (t,k)}_{L^2_v}d\Sigma (k)+  \int _{\mathbb Z^3}\left(\int_{1}^{T} \normm{  \partial_{v_1}^{m}\hat f}^{2}dt\right)^{\frac12}d\Sigma (k)\leq \frac{\eps_0C_*^{m} (m!)^{\tau}}{(m+1)^2}.
 \end{equation*}
Thus \eqref{vind} holds  for $j=m$ and thus for any $j\in\mathbb Z_+.$     The proof of   
  Proposition \ref{prp:v} is completed. 
\end{proof}

\begin{proof}
	[Proof of  Theorem \ref{gxt}: derivation of estimate \eqref{opa}]  The estimate follows directly from Propositions \ref{theorem} and \ref{prp:v}. Indeed, applying Lemma \ref{lem:fm}, we obtain
		\begin{equation*}
	\begin{aligned}
	&\int _{\mathbb Z^3}\sup\limits_{t\geq 1 } t^m \norm{ \widehat{\partial_{x_1}^{m}  f} (t,k)}_{L^2_v}d\Sigma (k) \\
	&\leq 2^m \int _{\mathbb Z^3}\sup\limits_{t\geq 1}  \norm{ \widehat{H^{m}  f} (t,k)}_{L^2_v}d\Sigma (k)+2^m \int _{\mathbb Z^3}\sup\limits_{t\geq 1}  \norm{ \partial_{v_1}^m \hat f (t,k)}_{L^2_v}d\Sigma (k)\\
	&\leq  2  \eps_0(2C_*)^m (m!)^\tau,
	\end{aligned}
	\end{equation*}
	where the last line follows from Propositions \ref{theorem} and \ref{prp:v}.  
 	Similarly, the same estimate remain valid with $\partial_{x_1}$ replaced by $\partial_{x_2}$ or $\partial_{x_3}$. This, with the fact that
\begin{eqnarray*}
\forall\ \alpha\in\mathbb Z_+^3,\quad 	\norm{\widehat{\partial_x^\alpha f}}_{L^2_{v}}\leq \sum_{1\leq j\leq 3}\norm{\mathscr F_x\big(\partial_{x_j}^{\abs\alpha}f\big)}_{L^2_{v}},
\end{eqnarray*}
gives
\begin{eqnarray*}
	\forall\ \alpha\in\mathbb Z_+^3,\quad 	\int _{\mathbb Z^3}\sup\limits_{t\geq 1} t^{\abs\alpha} \norm{ \widehat{\partial_{x}^{\alpha}  f} (t,k)}_{L^2_v}d\Sigma (k)\leq  6\eps_0(2C_*)^{|\alpha|} (|\alpha|!)^\tau.
\end{eqnarray*}
Similarly, 
\begin{eqnarray*}
	\forall\ \alpha\in\mathbb Z_+^3,\quad 	\int _{\mathbb Z^3} \Big(\int_1^{+\infty}  t^{2\abs\alpha} \normm{ \widehat{\partial_{x}^{\alpha}  f} (t,k)}^2dt\Big)^{\frac12}d\Sigma (k)\leq 6 \eps_0(2C_*)^{|\alpha|} (|\alpha|!)^\tau.
\end{eqnarray*}
As a result, taking $C=12+2C_*$, we obtain the desired estimate \eqref{opa}.
\end{proof}

\subsection{Improved radius in Gevrey space}
In the previous part, we proved that the radius in the optimal regularity space is bounded from below by $t$ for $t \geq 1$. We now proceed to prove estimate \eqref{opg} in Theorem \ref{gxt}, which establishes the lower bound $t^{\frac{1+2s}{2s}}$ for the radius in the Gevrey space. This estimate will follow from the proposition below.

\begin{proposition} \label{5.1}
Suppose the hypothesis of Theorem \ref{gxt} holds. Then there exists a positive constant $\bar C$ depending only on $C_0,C_1$ in Section \ref{sec:prelim} such that for any $m\in\mathbb Z_+$  it holds that
\begin{multline*}
	 \int_{\mathbb{Z}^3} \sup\limits_{t\geq 1} t^{\frac{1+2s}{2s}m} \comi{k}^{m} \norm{ \hat{f} }_{L^2_v}d\Sigma(k) 
 + \int_{\mathbb{Z}^3}\comi{k}^{m} \left(\int_{1}^{+\infty}t^{\frac{1+2s}{s}m} \normm{\hat{f}}^{2}dt\right)^{\frac12}d\Sigma(k) \\
+ \int_{\mathbb{Z}^3}\comi{k}^{m+\frac{s}{1+2s}} \left(\int_{1}^{+\infty}t^{\frac{1+2s}{s}m} \norm{  \hat{f}}_{L_v^2}^{2}dt\right)^{\frac12}d\Sigma(k)\leq \frac{\eps_0 \bar C^{m}  (m!)^{\frac{1+2s}{2s}} }{(m+1)^2}.	
	\end{multline*}
\end{proposition}

\begin{proof}
  
We  use induction on $m$ to prove Proposition \ref{5.1} and the validity of $m=0$ follows from \eqref{v2} and  \eqref{v33}. 

Fix $ m\geq 1$,    and  suppose there exists a constant $\bar C$ such that for any  $ j \leq m-1 $ and any $T>1$,  the  following estimate holds: 
\begin{equation}\label{250527}
\begin{aligned}
&\int_{\mathbb{Z}^3}\sup\limits_{1\leq t \leq T} t^{\frac{1+2s}{2s}j} \comi{k}^{j} \norm{ \hat{f} }_{L^2_v}d\Sigma(k) + \int_{\mathbb{Z}^3}\comi{k}^{j} \bigg(\int_{1}^{T}t^{\frac{1+2s}{s}j} \normm{  \hat{f}(t,k)}^{2}dt\bigg)^{\frac12}d\Sigma(k) 
\\
&\qquad\qquad + \int_{\mathbb{Z}^3}\comi{k}^{j+\frac{s}{1+2s}} \bigg(\int_{1}^{T}t^{\frac{1+2s}{s}j} \norm{  \hat{f}}_{L_v^2}^{2}dt\bigg)^{\frac12}d\Sigma(k) \leq \frac{\eps_0 \bar C^{j}  (j!)^{\frac{1+2s}{2s}} }{(j+1)^2}.
\end{aligned}
\end{equation}
 We will prove  that estimate \eqref{250527} remains  valid for $j=m,$ provided $\bar {C}$ is chosen sufficiently large. 

{\it Step 1.}  From equation 
\begin{equation*}
	\frac12\frac{d}{dt}\norm{\hat f}_{L_v^2}^2-\textrm{Re} \big ( \mathcal{L}\hat{f},\ \hat{f} \big)_{L^2_v}=\textrm{Re} \big( 
			  \hat{\Gamma}(\hat{f},\hat{f}),\  \hat{f}\big)_{L^2_v} 
\end{equation*}
together with  estimates \eqref{rela} and  \eqref{upptrifour}, we obtain 
\begin{equation*}
	\frac12\frac{d}{dt}\norm{\hat f(t,k)}_{L_v^2}^2+C_1\normm{  \{\mathbf{I}-\mathbf{P} \}\hat{f}}^2\leq C\normm{\hat f(k)} \int_{\mathbb Z^3 } \norm{\hat f(k-\ell)}_{L^2_v}   \normm{\hat f(\ell)}    d\Sigma(\ell).
\end{equation*}
Using the fact  
		\begin{equation*} \label{250526}
		\begin{aligned}
		 \normm{\mathbf{P} \hat{f}}^2 \leq 
	C\big(  |\hat { a  } |^2+
		|\hat{b}|^2+|\hat{ c }|^2\big), 
			\end{aligned}
		\end{equation*}	
and Proposition \ref{macmic}, we deduce  that 
\begin{multline*}
	\frac12\frac{d}{dt}\norm{\hat f}_{L_v^2}^2+\delta_0\normm{  \hat{f}}^2\leq   C\normm{\hat f(t,k)} \int_{\mathbb Z^3 } \norm{\hat f(k-\ell)}_{L^2_v}   \normm{\hat f(\ell)}    d\Sigma(\ell)\\ 
	  - 4\delta_0 \frac{d}{dt} {\rm Re}\, \mathcal{K}_0(\hat f)    +C\delta_0 \Big|\int_{\mathbb Z^3}   \norm{\hat{f}(k-\ell)}_{L_v^2} \normm{\hat{f}(\ell)} d\Sigma(\ell)\Big|^2
\end{multline*}
for some sufficiently small  $\delta_0>0$, where $\mathcal{K}_0(\hat f(t,k))$ is defined  in Proposition \ref{macmic}.  Multiplying both sides by  $\big(t^{ \frac{1+2s}{2s} m}\comi k^{m}\big)^2 $  yields
\begin{multline*}
	 \frac12\frac{d}{dt}\Big(t^{ \frac{1+2s}{s} m}\comi k^{2m}\norm{\hat f(t,k)}_{L_v^2}^2\Big)+\delta_0 t^{ \frac{1+2s}{s} m}\comi k^{2m} \normm{  \hat{f} (t,k)}^2\\
	 \leq   - 4\delta_0 \frac{d}{dt} \Big(t^{ \frac{1+2s}{s} m}\comi k^{2m} {\rm Re}\, \mathcal{K}_0(\hat f)\Big)+ C  m t^{ \frac{1+2s}{s} m-1}\comi k^{2m}\norm{\hat f}_{L_v^2}^2\\
       + \frac{4(1+2s)}{s} m \delta_0  t^{ \frac{1+2s}{s} m-1 }\comi k^{2m} {\rm Re}\, \mathcal{K}_0(\hat f) +C \mathscr R_m , 
\end{multline*}
where 
\begin{multline}\label{rm}
	\mathscr R_m(t,k):=t^{ \frac{1+2s}{s} m}\comi k^{2m}\normm{\hat f(t,k)} \int_{\mathbb Z^3 } \norm{\hat f(t,k-\ell)}_{L^2_v}   \normm{\hat f(t,\ell)}    d\Sigma(\ell)\\
	   +    \Big|t^{ \frac{1+2s}{2s} m}\comi k^{m}\int_{\mathbb Z^3}   \norm{\hat{f}(t,k-\ell)}_{L_v^2} \normm{\hat{f}(t,\ell)} d\Sigma(\ell)\Big|^2.
\end{multline}
Integrating over $t \in [1,T]$ and $k \in \mathbb{Z}^3$ and using 
the estimate
   \begin{equation*}\label{dt}
	4\delta_0 |\mathcal{K}_0(\hat f)|  \leq C\delta_0 \norm{\hat f}_{L_v^2}^2, 
\end{equation*}
which follows from   \eqref{km} for $m=0,$  we find that by choosing $\delta_0$ sufficiently small,  
\begin{equation}\label{f1k}
		\begin{aligned}
		& \int_{\mathbb{Z}^3}\sup\limits_{1\leq t \leq T} t^{\frac{1+2s}{2s}m} \comi{k}^{m} \norm{ \hat{f} }_{L^2_v}d\Sigma(k) 
  +  \int_{\mathbb{Z}^3}\comi{k}^{m} \left(\int_{1}^{T}t^{\frac{1+2s}{s}m} \normm{  \hat{f} }^{2}dt\right)^{\frac12}d\Sigma(k)\\
&\leq C\int_{\mathbb{Z}^3}   \comi{k}^{m} \norm{ \hat{f}(1,k)}_{L^2_v}d\Sigma(k)+C\int_{\mathbb Z^3}\Big( \int_1^T\mathscr R_m (t,k) dt\Big)^{\frac12} d\Sigma(k) \\
&\quad  + C  \int_{\mathbb Z^3} \bigg[\int_1^T   m t^{ \frac{1+2s}{s}m -1}\comi k^{2m}\norm{\hat f}_{L_v^2}^2dt\bigg]^{\frac12}d\Sigma(k) 
 	.
	\end{aligned}
\end{equation}
By Propositions \ref{theorem} and \ref{prp:v},   we may choose $\bar{C}$ large enough so that
\begin{equation*}
	 \int_{\mathbb{Z}^3}   \comi{k}^{m} \norm{ \hat{f}(1,k)}_{L^2_v}d\Sigma(k) \leq 
	C \frac{\eps_0 4^m C_*^{m}  ( m!)^\tau  }{(m+1)^2} 
	 \leq  C \frac{ \eps_0\bar C^{m-1}  (m!)^{\frac{1+2s}{2s}}}{(m+1)^2},
\end{equation*}
where $\tau=\max\big \{ \frac{1}{2s}, 1\big\}$ and $C_*$ is the constant in Proposition  \ref{theorem}.  For the last term on the right-hand of \eqref{f1k},    using the  inequality  \begin{equation*}
  m t^{ \frac{1+2s}{s}m -1}\comi k^{2m}   \leq  \eps t^{\frac{1+2s}{s}m}\comi k^{2m+\frac{2s}{1+2s}} +\eps^{-\frac{1+s}{s}} m^{\frac{1+2s}{s}} t^{\frac{1+2s}{s}(m-1) }\comi k^{2(m-1)+\frac{2s}{1+2s}}, 
   \end{equation*}
we obtain that, for any $\eps>0,$ 
  \begin{equation}\label{mm1}
  	\begin{aligned}
  		&  \int_{\mathbb Z^3} \bigg[\int_1^T   m t^{ \frac{1+2s}{ s}m -1}\comi k^{2m}\norm{\hat f}_{L_v^2}^2dt\bigg]^{\frac12}d\Sigma(k)\\
  		&\leq \eps  \int_{\mathbb Z^3} \comi k^{m+\frac{s}{1+2s}} \bigg[\int_1^T  t^{ \frac{1+2s}{s} m}  \norm{\hat f}_{L_v^2}^2dt\bigg]^{\frac12}d\Sigma(k)\\
  		&\quad +C\eps^{-\frac{1+s}{2s}}   m^{\frac{1+2s}{2s}} \int_{\mathbb Z^3} \comi k^{(m-1)+\frac{s}{1+2s}} \bigg[\int_1^T  t^{\frac{1+2s}{s}(m-1)}  \norm{\hat f}_{L_v^2}^2dt\bigg]^{\frac12}d\Sigma(k)\\
  		&\leq \eps  \int_{\mathbb Z^3} \comi k^{m+\frac{s}{1+2s}} \bigg[\int_1^T  t^{m\frac{1+2s}{s} }  \norm{\hat f}_{L_v^2}^2dt\bigg]^{\frac12}d\Sigma(k)+  C\eps^{-\frac{1+s}{2s}}    \frac{ \eps_0 \bar C^{m-1} (m!)^{\frac{1+2s}{2s}}}{m^2},
  	\end{aligned}
  \end{equation}
 where  the last line follows from the inductive assumption  \eqref{250527}. 
Substituting the above estimates into  inequality \eqref{f1k} yields that, for any $\eps>0$ 
\begin{equation}\label{difestimate}
\begin{aligned}
		& \int_{\mathbb{Z}^3}\sup\limits_{1\leq t \leq T} t^{\frac{1+2s}{2s}m} \comi{k}^{m} \norm{ \hat{f}  }_{L^2_v}d\Sigma(k) 
  + \int_{\mathbb{Z}^3}\comi{k}^{m} \left(\int_{1}^{T}t^{\frac{1+2s}{s}m} \normm{\hat{f}}^{2}dt\right)^{\frac12}d\Sigma(k)\\
&\leq  \eps  \int_{\mathbb Z^3} \comi k^{m+\frac{s}{1+2s}} \bigg[\int_1^T  t^{m\frac{1+2s}{s} }  \norm{\hat f}_{L_v^2}^2dt\bigg]^{\frac12}d\Sigma(k)+  C\eps^{-\frac{1+s}{2s}}    \frac{ \eps_0 \bar C^{m-1} (m!)^{\frac{1+2s}{2s}}}{m^2}\\
&\quad +C\int_{\mathbb Z^3}\Big( \int_1^T\mathscr R_m(t,k) dt\Big)^{\frac12} d\Sigma(k),
	\end{aligned}
\end{equation}
where $\mathscr R_m$ is defined in \eqref{rm}. 

{\it Step 2.} We now estimate the last term on the right-hand side of \eqref{difestimate} and show that
\begin{equation}\label{esrm}
	\begin{aligned}
		&\int_{\mathbb Z^3}\Big( \int_1^T\mathscr R_m(t,k) dt\Big)^{\frac12} d\Sigma(k) \leq
		C\eps_0 \int_{\mathbb{Z}^3} \sup_{1\leq t\leq T}t^{\frac{1+2s}{2s}m}\comi k^m\norm{\hat{f}}_{L^2_v}d\Sigma(k)\\
 	&\qquad\qquad +C\eps_0  \int_{\mathbb{Z}^3} \comi k^m \bigg( \int_1^{T} t^{\frac{1+2s}{s}m} \normm{\hat{f}}^2dt  \bigg)^{\frac12}d\Sigma(k)
 	+C  \eps_0^2\frac{ \bar C^m  (m!)^{\frac{1+2s}{2s}}}{(m+1)^2}.
	\end{aligned}
\end{equation}
Using Lemma \ref{itea} and  \eqref{MF} we compute 
\begin{equation*}
	\begin{aligned}
	&\int_{\mathbb Z^3}	\Big[\int_1^T \Big|  t^{\frac{1+2s}{2s}m}\comi k^m \int_{\mathbb Z^3 } \norm{\hat f(k-\ell)}_{L^2_v}   \normm{\hat f(\ell)}    d\Sigma(\ell)\Big|^2dt\Big]^{\frac12}d\Sigma(k)\\
	&\leq C \int_{\mathbb Z^3}	\bigg[\int_1^T \Big| \int_{\mathbb Z^3 }\sum_{j=0}^{m}\binom{m}{j}  t^{\frac{1+2s}{2s}j} \comi {k-\ell}^{j} \norm{\hat f(k-\ell)}_{L^2_v} \\
	&\qquad  \qquad \qquad \qquad \qquad  \qquad    \times  t^{\frac{1+2s}{2s}(m-j)} \comi \ell^{m-j}\normm{\hat f(\ell)}    d\Sigma(\ell)\Big|^2dt\bigg]^{\frac12}d\Sigma(k)\\
	&\leq C
	\sum_{j=0}^{m}\binom{m}{j}\Big(\int_{\mathbb{Z}^3}\sup\limits_{ 1\leq t\leq T} t^{\frac{1+2s}{2s}j} \comi {k}^{j}\norm{\hat{f}(t,k)}_{L^2_v}d\Sigma(k)\Big)\\
	&\qquad\qquad \times \int_{\mathbb{Z}^3}\Big(\int_{1}^{T}t^{\frac{1+2s}{s}(m-j)} \comi {k}^{2(m-j)}\normm{  \hat{f}(t,k)}^{2}dt\Big)^{1\over2}d\Sigma(k).
	\end{aligned}
\end{equation*}
Using the inductive hypothesis \eqref{250527} and following the computations in the proof of Lemma \ref{lem: ma}, one has 
  \begin{equation}\label{mjs}
  \begin{aligned}
  & \sum_{j=0}^{m}\binom{m}{j}\Big(\int_{\mathbb{Z}^3}\sup\limits_{ 1\leq t\leq T} t^{\frac{1+2s}{2s}j} \comi {k}^{j}\norm{\hat{f}(t,k)}_{L^2_v}d\Sigma(k)\Big)\\
	&\qquad\qquad \times \int_{\mathbb{Z}^3}\Big(\int_{1}^{T}t^{\frac{1+2s}{s}(m-j)} \comi {k}^{2(m-j)}\normm{  \hat{f}(t,k)}^{2}dt\Big)^{1\over2}d\Sigma(k)	\\
&	\leq C\eps_0 \int_{\mathbb{Z}^3} \sup_{1\leq t\leq T}t^{\frac{1+2s}{2s}m}\comi k^m\norm{\hat{f}(t,k)}_{L^2_v}d\Sigma(k)\\
 	&\quad+C\eps_0  \int_{\mathbb{Z}^3} \comi k^m \bigg( \int_1^{T} t^{\frac{1+2s}{s}m} \normm{\hat{f}(t,k)}^2dt  \bigg)^{\frac12}d\Sigma(k)
 	+C  \eps_0^2\frac{ \bar C^m  (m!)^{\frac{1+2s}{2s}}}{(m+1)^2}.
	\end{aligned}
  \end{equation}
Combining the two inequalities above yields that  
\begin{equation*}
	\begin{aligned}
		&\int_{\mathbb Z^3}	\Big[\int_1^T \Big|  t^{\frac{1+2s}{2s}m}\comi k^m \int_{\mathbb Z^3 } \norm{\hat f(k-\ell)}_{L^2_v}   \normm{\hat f(\ell)}    d\Sigma(\ell)\Big|^2dt\Big]^{\frac12}d\Sigma(k)\\
		&\leq C\eps_0 \int_{\mathbb{Z}^3} \sup_{1\leq t\leq T}t^{\frac{1+2s}{2s}m}\comi k^m\norm{\hat{f}(t,k)}_{L^2_v}d\Sigma(k)\\
 	&\quad+C\eps_0  \int_{\mathbb{Z}^3} \comi k^m \bigg( \int_1^{T} t^{\frac{1+2s}{s}m} \normm{\hat{f}(t,k)}^2dt  \bigg)^{\frac12}d\Sigma(k)
 	+C  \eps_0^2\frac{ \bar C^m  (m!)^{\frac{1+2s}{2s}}}{(m+1)^2}. 		
 	\end{aligned}
\end{equation*}
A similar bound holds for the first term in the definition \eqref{rm}   of $\mathscr R_m$.  This yields \eqref{esrm}.

{\it Step 3.}  Substituting \eqref{esrm} into  \eqref{difestimate}, we obtain for any $\eps>0$ 
\begin{equation*}\label{}
		\begin{aligned}
		& \int_{\mathbb{Z}^3}\sup\limits_{1\leq t \leq T} t^{\frac{1+2s}{2s}m} \comi{k}^{m} \norm{ \hat{f} }_{L^2_v}d\Sigma(k) 
  + \int\comi{k}^{m} \left(\int_{1}^{T}t^{\frac{1+2s}{s}m} \normm{  \hat{f} }^{2}dt\right)^{\frac12}d\Sigma(k)\\
&\leq \eps  \int_{\mathbb Z^3} \comi k^{m+\frac{s}{1+2s}} \bigg[\int_1^T  t^{m\frac{1+2s}{s} }  \norm{\hat f}_{L_v^2}^2dt\bigg]^{\frac12}d\Sigma(k)\\
&\quad+ C\big(\eps_0+  \eps^{-\frac{1+s}{2s}} \bar C^{-1}\big)  \frac{ \eps_0 \bar C^{m-1} (m!)^{\frac{1+2s}{2s}}}{m^2},
	\end{aligned}
\end{equation*}
provided $\eps_0$ is small sufficiently.  For the first term on the right-hand side, 
combining    \eqref{mjs} and \eqref{mm1} with \eqref{sube}  yields
\begin{equation*}
		\begin{aligned}
  &\int_{\mathbb Z^3} \comi{k}^{m+\frac{s}{1+2s}} \Big(\int_1^T   t ^{\frac{1+2s}{s}m} 
	 \norm{\hat{f}}^2_{L^2_v}dt \Big)^{\frac12}d\Sigma(k) \\
	 &\leq  C \int_{\mathbb{Z}^3}\sup\limits_{1\leq t \leq T} t^{\frac{1+2s}{2s}m} \comi{k}^{m} \norm{ \hat{f} }_{L^2_v}d\Sigma(k) \\
 	&\qquad+C  \int_{\mathbb{Z}^3} \comi k^m \bigg( \int_1^{T} t^{\frac{1+2s}{s}m} \normm{\hat{f}}^2dt  \bigg)^{\frac12}d\Sigma(k)
 +	C\big(\eps_0+   \bar C^{-1}\big)\frac{ \bar C^m  (m!)^{\frac{1+2s}{2s}}}{(m+1)^2}.  
 	\end{aligned}
	\end{equation*}
Consequently, from the two estimates above, it follows that
\begin{equation*} \label{25052801}
	\begin{aligned}
		& \int_{\mathbb{Z}^3}\sup\limits_{1\leq t \leq T} t^{\frac{1+2s}{2s}m} \comi{k}^{m} \norm{ \hat{f} }_{L^2_v}d\Sigma(k) + \int_{\mathbb{Z}^3}\comi{k}^{m} \left(\int_{1}^{T}t^{\frac{1+2s}{s}m} \normm{  \hat{f}}^{2}dt\right)^{\frac12}d\Sigma(k)
\\
&\qquad\qquad+ \int_{\mathbb{Z}^3}\comi{k}^{m+\frac{s}{1+2s}} \left(\int_{1}^{T}t^{\frac{1+2s}{s}m} \norm{  \hat{f}}_{L_v^2}^{2}dt\right)^{\frac12}d\Sigma(k) \\
&\leq C	 \big(\eps_0+    \bar C^{-1}\big)\frac{ \bar C^m  (m!)^{\frac{1+2s}{2s}}}{(m+1)^2}\leq  \frac{\bar C^m  (m!)^{\frac{1+2s}{2s}}}{(m+1)^2},
	\end{aligned}
\end{equation*}
provided   $\eps_0$ is small enough  and $\bar C$ is large sufficiently.  Thus \eqref{250527} holds true for $j=m$ and thus for any $j\in\mathbb Z_+.$
This completes the proof of Proposition \ref{5.1}. 
\end{proof}

   \subsection*{\bf Acknowledgements}
 \noindent The research of W.-X. Li was supported by Natural Science Foundation of China (Nos. 12325108,12131017, 12221001),  and  the Natural Science Foundation of Hubei Province (No. 2019CFA007).   The research of L. Liu was supported by Natural Science Foundation of China (No. 12571243) and  Anhui Provincial Natural Science Foundation (No. 2508085Y004). The research of H. Wang was supported by Natural Science Foundation of China (No. 12301284).
   
\subsection*{\bf Declarations}
\noindent 
The authors declare that there is no conflict of interest. All the authors contributed equally to this
work. The manuscript has no associated data.

\appendix

\section{Proof of Lemma \ref{lem:fm}}\label{app:inequ}
To prove \eqref{+pse1} and \eqref{pse1+} we compute
 \begin{align*}
&\Big|\mathscr F_{x,v}\big(	(A_1+A_2)^mf \big)(k,\eta)\Big|^2 =\big|  	\big (\varphi_1  (k,\eta)+\varphi_2 (k,\eta)\big )^m \mathscr F_{x,v}f (k,\eta)\big|^2  \\
&\leq \big(|\varphi_1(k,\eta)|+|\varphi_2(k,\eta)| \big)^{2m}  \times  \big| \mathscr{F}_{x,v}f(k, \eta) \big |^2 \\
&\leq 2^{2m} \big|\varphi_1(k,\eta)^m \mathscr{F}_{x,v}f(k, \eta) \big|^2  +  2^{2m} \big| \varphi_2(k,\eta)^m\mathscr{F}_{x,v}f(k, \eta) \big| ^{2}\\
&\leq 2^{2m}\big|  \mathscr{F}_{x,v}(A_1^mf)(k, \eta) \big|^2  + 2^{2m} \big|   \mathscr{F}_{x,v}(A_2^mf)(k, \eta) \big| ^{2},
\end{align*}
the second inequality using the fact that $(p+q)^{ 2m}\leq (2p)^{ 2m}+(2q)^{ 2m}$ for any numbers $p,q\geq 0$ and any $m\in\mathbb Z_+$.
 As a result, we combine the above estimate with Parseval equality, to conclude that
   \begin{equation*}
   	\begin{aligned}
   		& \norm{\mathscr{F}_x\big( (A_1+A_2)^mf\big)(k)}^2_{L^2_{v}} =\int_{  \mathbb R^3} \Big|\mathscr F_{x,v}\big(	(A_1+A_2)^mf \big)(k,\eta)\Big|^2  d\eta
   		\\ &\leq 2^{2m} \int_{ \mathbb R^3} \big|  \mathscr{F}_{x,v}(A_1^mf)(k, \eta) \big|^2   d\eta + 2^{2m} \int_{\mathbb R^3} \big|   \mathscr{F}_{x,v}(A_2^mf)(k, \eta) \big| ^{2}  d\eta
   		\\& \leq     2^{2m}\norm{\mathscr{F}_x\big( A_1^mf\big)(k) }^2_{L^2_{v}}+ 2^{2m}\norm{\mathscr{F}_x\big( A_2^mf\big)(k) }^2_{L^2_{v}}\\
   		& \leq    \Big( 2^{m}\norm{\mathscr{F}_x\big( A_1^mf\big)(k) }_{L^2_{v}}+ 2^{m}\norm{\mathscr{F}_x\big( A_2^mf\big)(k) }_{L^2_{v}}\Big)^2.
   	\end{aligned}
   \end{equation*}
   This gives \eqref{+pse1} and \eqref{pse1+}. 

Estimate \eqref{fmn} just follows from the fact that
\begin{equation*}
     |\varphi_1(k,\eta)^m \varphi_2(k,\eta)^n|\leq |\varphi_1(k,\eta)^{m+n}|+|\varphi_2(k,\eta)^{m+n}|.
\end{equation*}
 The proof of Lemma \ref{lem:fm} is completed.

\section{Subelliptic estimates}\label{app:sub}

This section is devoted to proving Proposition \ref{sub:ell} and Corollary \ref{cor:sub}, which establish the subelliptic estimates for the Boltzmann equation.

\begin{proof}
	[Proof of Proposition \ref{sub:ell}] 
Taking Fourier transform with respect to $x$ in equation \eqref{linbol}, we obtain 
\begin{equation} \label{equa}
\partial_{t}\hat{h}+iv\cdot k \hat{h}-\mathcal{L}\hat{h}
=\hat g. 
\end{equation}
For each $k$, define $\lambda_{k}(\eta)$ by setting 
\begin{equation}\label{lake}
\lambda_{k}(\eta):=\frac{-k\cdot\eta}{\comi k^{\frac{2+2s}{1+2s}}}\chi
\bigg(\frac{\abs\eta}{\comi k^{\frac{1}{1+2s}}} \bigg),
\end{equation}
where $\chi \in C_0^{\infty}(\mathbb{R};[0,1])$ such that $\chi=1$ on
$[-1,1]$ and supp $\chi\subset [-2,2].$
Let $\lambda_{k}(D_v)$ be the Fourier multiplier of the symbol $\lambda_{k}(\eta)$, namely,
\begin{equation*}
	\mathscr F_v (\lambda_k(D_v) h)(\eta)=\lambda_k(\eta)(\mathscr F_v h) (\eta).
\end{equation*}
Note for any $\alpha\in\mathbb{Z}_+^3$ there exists a constant $C_\alpha$ depending only on $\alpha,$ such that  
\begin{equation*}
\forall \ k\in\mathbb Z^3,\  \forall\ \eta\in\mathbb R^3,\quad |\partial_\eta^\alpha \lambda_k(\eta)|\leq C_\alpha.
\end{equation*}
This, with the help of pseudo-differential calculus,  yields 
\begin{equation}\label{uppb}
    \norm{\lambda_k(D_v)h}_{L_v^2}\leq    \norm{  h}_{L_v^2}\ \textrm{ and }\  \normm{\lambda_k(D_v) h} \leq C  \normm{ h},
\end{equation}
where the first inequality uses the fact that $|\lambda_k(\eta)|\leq 1$  and the second estimate follows from the characterization of the trip-norm $\normm{\cdot}$  in terms of pseudo-differential operators (cf. \cite{MR3950012}).  Taking $L^2_v$-product with $\lambda_k(D_v)\hat{h}$ on both sides of   equation \eqref{equa} and using \eqref{rela},
we obtain  
\begin{equation}\label{lr}
    \text{Re}\ \big( i(v\cdot k)\hat{h},\  \lambda_k(D_v)\hat{h} \big)_{L^2_v} 
\leq - \frac12\frac{d}{dt}    \big(\hat h, \lambda_k(D_v)\hat h\big)_{L^2_v}+ \text{Re}\ \big( \hat{g},\  \lambda_k(D_v)\hat{h} \big)_{L^2_v}.
\end{equation} 
 On the other hand,  
 \begin{align*}
     \text{Re}\ \big( i(v\cdot k)\hat{h},\  \lambda_k(D_v)\hat{h} \big)_{L^2_v}&=\text{Re}\ \big(  (\partial_\eta\cdot k)\mathscr F_{x,v} h,\  \lambda_k(\eta)\mathscr F_{x,v}{h} \big)_{L^2_\eta}\\
     &=-\sum_{1\leq j\leq 3}\big(k_j   (\partial_{\eta_j}\lambda_k(\eta))     \mathscr F_{x,v} h,\  \mathscr F_{x,v}{h} \big)_{L^2_\eta}
 \end{align*}
 In view of \eqref{lake}, direct computation yields that 
 \begin{multline*}
    \sum_{1\leq j\leq 3}  k_j   (\partial_{\eta_j}\lambda_k(\eta)) \geq  \frac{ |k|^2}{\comi k^{\frac{2+2s}{1+2s}}}\chi
\bigg(\frac{\abs\eta}{\comi k^{\frac{1}{1+2s}}} \bigg)-C\comi \eta^{2s}\\
 \geq \frac{ |k|^2}{\comi k^{\frac{2+2s}{1+2s}}}- \frac{ |k|^2}{\comi k^{\frac{2+2s}{1+2s}}}\bigg[1-\chi
\bigg(\frac{\abs\eta}{\comi k^{\frac{1}{1+2s}}} \bigg)\bigg]-C\comi \eta^{2s}\geq  \comi k^{\frac{2s}{1+2s}}-C\comi \eta^{2s}. 
 \end{multline*} 
 As a result, combining the   two estimates above we conclude
 \begin{align*}\label{RE}
     \text{Re}\ \big( i(v\cdot k)\hat{h},\  \lambda_k(D_v)\hat{h} \big)_{L^2_v}&\geq  \comi k^{\frac{2s}{1+2s}} \norm{\hat{h} }_{L^2_v}^2-C \norm{\hat{h} }_{H^s_v}^2\geq  \comi k^{\frac{2s}{1+2s}} \norm{\hat{h} }_{L^2_v}^2-C \normm{\hat{h} }^2,
 \end{align*}
where the last inequality uses \eqref{+lowoftri}. Substituting the above estimate into \eqref{lr} yields that
 \begin{equation}\label{B5}
     \begin{aligned}
          \comi k^{\frac{2s}{1+2s}} \norm{\hat{h}(t,k) }_{L^2_v}^2 \leq - \frac12\frac{d}{dt}    \big(\hat h, \lambda_k(D_v)\hat h\big)_{L^2_v}+C \normm{\hat{h} }^2+ \text{Re}\ \big( \hat{g},\  \lambda_k(D_v)\hat{h} \big)_{L^2_v}.
     \end{aligned}
 \end{equation}
 On the other hand, by \eqref{equa} and \eqref{+rela}, we have
 \begin{equation}\label{B6}
     \begin{aligned}
       \frac12 \frac{d}{dt}  \norm{\hat h}_{L^2_v}^2+ C_1   \normm{ \hat{h} }^2 \leq    \text{Re}\ \big( \hat{g},\   \hat{h} \big)_{L^2_v}+\norm{\hat h}_{L_v^2}^2.
     \end{aligned}
 \end{equation}
 Then multiply both side of \eqref{B5} by some constant $c_0>0$, we have
 \begin{equation}\label{B55}
     \begin{aligned}
          c_0\comi k^{\frac{2s}{1+2s}} \norm{\hat{h}(t,k) }_{L^2_v}^2 \leq -  \frac{c_0}{2}\frac{d}{dt}    \big(\hat h, \lambda_k(D_v)\hat h\big)_{L^2_v}+Cc_0 \normm{\hat{h} }^2+c_0 \text{Re}\ \big( \hat{g},\  \lambda_k(D_v)\hat{h} \big)_{L^2_v}.
     \end{aligned}
 \end{equation}
Together with Gronwall’s inequality, we integrate the above two  estimates \eqref{B6}-\eqref{B55} over $[t_1, t_2]$ and using
\begin{equation*}%\label{B7}
	\big| \big(\hat h, \lambda_k(D_v)\hat h\big)_{L^2_v}\big|\leq \norm{\hat h}_{L^2_v}^2 
\end{equation*}
that follows from \eqref{uppb}, 
we obtain for some small constant $ 0< c_0<1$ and $t_2 \leq 1$
\begin{align*}
         &\Big(c_0\int_{t_1}^{t_2}  \comi k^{\frac{2s}{1+2s}} \norm{\hat{h}(t,k)}_{L^2_v}^2dt\Big)^{\frac12}+\Big(\frac{1-c_0}{2}\Big)^{\frac12} \sup_{t\in[t_1,t_2]} \norm{\hat h(t,k)}_{L^2_v}\\
         &\qquad+\inner{C_1-Cc_0}^{\frac12}\Big(\int_{t_1}^{t_2}    \normm{\hat{h}(t,k)}^2dt\Big)^{\frac12}  \\
         &\leq  C \norm{\hat h(t_1,k)}_{L^2_v}+ \int_{t_1}^{t_2} \big| \big( \hat{g},\  \hat h+c_0\lambda_k(D_v)\hat{h}\big)_{L^2_v}\big|dt .
     \end{align*}
  This yields the assertion of Proposition \ref{sub:ell} by choosing 
  \begin{equation*}
  	\mathcal M=1+c_0\lambda_k(D_v)
  \end{equation*}
which satisfies condition \eqref{psk}  in view of \eqref{uppb}. The proof is completed. 
\end{proof}

\begin{proof}
	[Proof of Corollary \ref{cor:sub}]
Let $\comi{D_x}$ be the Fourier multiplier with symbol $\comi{k}$, namely,
\begin{equation*}
	\widehat{\comi{D_x}h}(k)=\comi{k}\hat h(k).
\end{equation*}
Multiplying   by $t^{\frac{1+2s}{2s}m} \comi{D_x}^{m}$ on both sides of  the Boltzmann equation  \eqref{3} yields 
\begin{multline*}
\big(\partial_{t} + v\cdot \partial_x -\mathcal{L}\big)t^{\frac{1+2s}{2s}m} \comi{D_x}^{m}f\\
= \frac{1+2s}{2s}m t^{\frac{1+2s}{2s}m-1} \comi{D_x}^{m}f+ t^{\frac{1+2s}{2s}m} \comi{D_x}^{m} \Gamma (f,f), 
\end{multline*}
Applying Proposition \ref{sub:ell} for $h=t^{\frac{1+2s}{2s}m} \comi{D_x}^{m}f$ and 
\begin{equation*}
	g=\frac{1+2s}{2s}m t^{\frac{1+2s}{2s}m-1} \comi{D_x}^{m}f+ t^{\frac{1+2s}{2s}m} \comi{D_x}^{m} \Gamma (f,f),
\end{equation*}
we obtain  
 \begin{equation*}
 		\begin{aligned}
  &\int_{\mathbb Z^3} \comi{k}^{m+\frac{s}{1+2s}} \Big(\int_{t_1}^{t_2}  t ^{\frac{1+2s}{s}m} 
	 \norm{\hat{f}}^2_{L^2_v}dt \Big)^{\frac12}d\Sigma(k) \\
     &\leq  C \int_{\mathbb{Z}^3}\sup\limits_{t_1 \leq t \leq t_2} t^{\frac{1+2s}{2s}m} \comi{k}^{m} \norm{ \hat{f} }_{L^2_v}d\Sigma(k)  \\
	&\quad+C\int_{\mathbb Z^3}   \bigg(\int_{t_1}^{t_2}   m   t^{\frac{1+2s}{s}m-1} \comi{k}^{2m}\norm{\hat{f} }_{L_v^2}^2 dt \bigg)^{\frac12}d\Sigma(k)   \\
    &\quad+C\int_{\mathbb Z^3}    \bigg(\int_{t_1}^{t_2}t^{\frac{1+2s}{s}m} \comi k^{2m} \big| \big(\mathscr F_x \inner{  \Gamma (f,f)},   \widehat{  \mathcal M f}\big)_{L_v^2}\big| dt \bigg)^{\frac12}d\Sigma(k).  
	\end{aligned}
 \end{equation*}
Then the assertion of Corolary  \ref{cor:sub} will follow if 
\begin{equation}\label{+fest}
	\begin{aligned}
	&\int_{\mathbb Z^3}    \bigg(\int_{t_1}^{t_2}t^{\frac{1+2s}{s}m} \comi k^{2m} \big| \big(\mathscr F_x \inner{  \Gamma (f,f)},   \widehat{  \mathcal M f}\big)_{L_v^2}\big| dt \bigg)^{\frac12}d\Sigma(k)\\
	&\leq C\int_{\mathbb Z^3}   \comi{k}^{m}\bigg(\int_{t_1}^{t_2}     t^{\frac{1+2s}{s}m} \normm{\hat{f} }^2 dt \bigg)^{\frac12}d\Sigma(k)\\
	&\quad +C
	\sum_{j=0}^{m}\binom{m}{j}\bigg(\int_{\mathbb{Z}^3}\sup\limits_{t_1\leq t\leq t_2} t^{\frac{1+2s}{2s}j} \comi {k}^{j}\norm{\hat{f}}_{L^2_v}d\Sigma(k)\bigg)\\
	&\quad\quad\qquad\qquad\qquad\quad  \times \int_{\mathbb{Z}^3}\Big(\int_{t_1}^{t_2} t^{\frac{1+2s}{s}(m-j)} \comi {k}^{2(m-j)}\normm{  \hat{f}}^{2}dt\Big)^{1\over2}d\Sigma(k).
	\end{aligned}
\end{equation}
To prove \eqref{+fest}, we 
use \eqref{upptrifour} and \eqref{uppb} to obtain 
 \begin{equation}\label{tkm}
     \begin{aligned}
         &t^{\frac{1+2s}{s}m} \comi k^{2m} \big| \big(\mathscr F_x \inner{  \Gamma (f,f)},   \widehat{  \mathcal M f}\big)_{L_v^2}\big| \\
         &\leq C t^{\frac{1+2s}{s}m} \comi{k}^{2m}\normm{\hat{f} }^2 +C  \bigg[t^{\frac{1+2s}{2s}m} \comi{k}^{m}\int_{\mathbb Z^3 } \norm{\hat f(t,k-\ell)}_{L^2_v}   \normm{\hat f(t,\ell)}    d\Sigma(\ell)\bigg]^2.
     \end{aligned}
 \end{equation}
 For the last term on the right-hand side, from Lemma \ref{itea} it follows that  
        \begin{align*}
           & t^{\frac{1+2s}{2s}m} \comi{k}^{m}  \int_{\mathbb Z^3 } \norm{\hat f(t,k-\ell)}_{L^2_v}   \normm{\hat f(t,\ell)}    d\Sigma(\ell)\\
           &\leq 2 \sum_{j=0}^{m}{m\choose j}\int_{\mathbb Z^3 } t^{\frac{1+2s}{2s}j}  \langle k -\ell\rangle^j
		   \norm{\hat f(t,k-\ell)}_{L^2_v}  t^{\frac{1+2s}{2s}(m-j)} \langle \ell\rangle^{m-j}\normm{\hat f(t,\ell)}    d\Sigma(\ell),
  \end{align*}
  which with \eqref{MF} yields 
  \begin{align*}
      &\int_{\mathbb Z^3}
	 \bigg[ \int_{t_1}^{t_2} \bigg(t^{\frac{1+2s}{2s}m} \comi{k}^{m}\int_{\mathbb Z^3 } \norm{\hat f(t,k-\ell)}_{L^2_v}   \normm{\hat f(t,\ell)}    d\Sigma(\ell)\bigg)^2 dt  \bigg]^{\frac12}d\Sigma(k)\\
           &\leq  C
	\sum_{j=0}^{m}\binom{m}{j}\Big(\int_{\mathbb{Z}^3}\sup\limits_{ t_1\leq t\leq t_2} t^{\frac{1+2s}{2s}j} \comi {k}^{j}\norm{\hat{f}(t,k)}_{L^2_v}d\Sigma(k)\Big)\\
	&\qquad\qquad\qquad\qquad\qquad\quad  \times \int_{\mathbb{Z}^3}\Big(\int_{t_1}^{t_2}t^{\frac{1+2s}{s}(m-j)} \comi {k}^{2(m-j)}\normm{  \hat{f}(t,k)}^{2}dt\Big)^{1\over2}d\Sigma(k).
  \end{align*}
  Combining this and \eqref{tkm}, we obtain \eqref{+fest}.  This completes the proof of Corollary \ref{cor:sub}.
\end{proof}

% \bibliographystyle{abbrv}
%\bibliography{references}

\end{document}